\documentclass{amsart}

\usepackage{latexsym,amsfonts,amssymb,exscale,enumerate,amscd}
\usepackage{amsmath,amsthm}
\usepackage{hyperref}

\usepackage{pstricks}

\psset{linewidth=0.3pt,dimen=middle}
\psset{xunit=.70cm,yunit=0.70cm}
\psset{arrowsize=1pt 5,arrowlength=0.6,arrowinset=0.7}

\usepackage[all]{xy}
\SelectTips{cm}{}

\usepackage{graphicx}

\newcommand{\ONH}{ONH}

\newcommand{\ONC}{ONC}
\newcommand{\pol}{\mathrm{Pol}}
\newcommand{\opol}{\mathrm{OPol}}
\newcommand{\sym}{\Lambda}
\newcommand{\osym}{\mathrm{O}\Lambda}
\newcommand{\qrk}{\mathrm{rk}_q}
\newcommand{\qdim}{\mathrm{dim}_q}
\newcommand{\Mat}{\mathrm{Mat}}
\newcommand{\sch}{\mathfrak{s}}

\newcommand{\undx}{\underline{x}}

\newcommand{\xt}{\widetilde{x}}
\newcommand{\ang}[1]{{\langle{#1}\rangle}}

\newcommand{\Gr}{\text{Gr}}
\newcommand{\Ss}{\mathcal{S}}

\newcommand{\onen}{{\mathbf 1}_{n}}

\newcommand{\Sq}{{\rm Sq}}
\newcommand{\und}[1]{\underline{#1}}

\newcommand{\bigb}[1]{
\begin{pspicture}(0,0)
 \rput(0,0){\psframebox[framearc=.5,fillstyle=solid]{\small $#1$}}
\end{pspicture}}

\def\nn{\notag}

\newcommand{\qbin}[2]{
\left[
 \begin{array}{c}
 #1 \\
 #2 \\
 \end{array}
 \right]
}

\newcommand{\refequal}[1]{\xy {\ar@{=}^{#1}
(-1,0)*{};(1,0)*{}};
\endxy}

\renewcommand{\to}{\rightarrow}
\newcommand{\maps}{\colon}
\newcommand{\op}{{\rm op}}

\newcommand{\End}{{\rm End}}

\newcommand{\im}{{\rm im\ }}

\newcommand{\rk}{{\rm rk\ }}

\def\det{\mathop{\rm det}}

\newcommand{\scs}{\scriptstyle}

\def\pmod{{\mathrm{-pmod}}}  

\def\mf{\mathfrak}
\def\shuffle{\,\raise 1pt\hbox{$\scriptscriptstyle\cup{\mskip
               -4mu}\cup$}\,}

\theoremstyle{definition}
\newtheorem{thm}{Theorem}[section]
\newtheorem{cor}[thm]{Corollary}
\newtheorem{conj}[thm]{Conjecture}
\newtheorem{lem}[thm]{Lemma}
\newtheorem{rem}[thm]{Remark}
\newtheorem{prop}[thm]{Proposition}
\newtheorem{defn}[thm]{Definition}

\numberwithin{equation}{section}

\def\emph#1{{\sl #1\/}}

\let\hat=\widehat

\let\phi=\varphi
\let\theta=\vartheta

\usepackage{bbm}
\def\C{{\mathbbm C}}
\def\N{{\mathbbm N}}
\def\R{{\mathbbm R}}
\def\Z{{\mathbbm Z}}
\def\Q{{\mathbbm Q}}

\usepackage{stmaryrd}

\normalfont\upshape

\title{The odd nilHecke algebra and its diagrammatics}

\author{Alexander P.\ Ellis}
\address{Department of Mathematics, Columbia University, New1
York, NY 10027, USA}
\email{ellis@math.columbia.edu}

\author{Mikhail Khovanov}
\address{Department of Mathematics, Columbia University, New
York, NY 10027, USA}
\email{khovanov@math.columbia.edu}

      \author{Aaron D.\ Lauda}
      \address{Department of Mathematics, University of Southern California, Los Angeles, CA 90089, USA}
\email{lauda@usc.edu}

\date{November 5, 2011}

\begin{document}
%

\begin{abstract}
We introduce an odd version of the nilHecke algebra and develop an odd analogue of the thick diagrammatic calculus for nilHecke algebras. We graphically describe idempotents which give a Morita equivalence between odd nilHecke algebras and the rings of odd symmetric functions in finitely many variables.  Cyclotomic quotients of odd nilHecke algebras are Morita equivalent to rings which are odd analogues of the cohomology rings of Grassmannians.  Like their even counterparts, odd nilHecke algebras categorify the positive half of quantum sl(2).
\end{abstract}

\maketitle

\tableofcontents

\newpage

%
\section{Introduction}
%

%
\subsection{Background}
%

The nilHecke algebra plays a central role in the theory of categorified
quantum groups, giving rise to an integral categorification of the negative
half of $\mathbf{U}_q(\mf{sl}_2)$~\cite{Lau1,KL,Rou2}. One of the original
motivations for categorifying quantum groups was to provide a unified
representation
theoretic explanation of the link homology theories that categorify
various quantum link invariants. Various steps in this direction have
since been completed, most notably by Webster~\cite{Web,Web2}.

Khovanov homology is one of the simplest of these link homology theories,
categorifying a certain normalization of the Jones
polynomial. Just as the quantum group $\mf{sl}_2$ plays a fundamental role
in explaining the Jones polynomial, the categorification of quantum
$\mf{sl}_2$ should play an equally important role in Khovanov homology.

Surprisingly, the categorification of the Jones polynomial is not unique.
Ozsv\'{a}th, Rasmussen, and Szab\'{o} found an {\em odd} analogue of Khovanov
homology~\cite{ORS}. This odd homology theory for links agrees with the
original Khovanov homology modulo 2. Both of these theories categorify the
Jones polynomial, and results of Shumakovitch~\cite[Section 3.1]{Shum} show
that these categorified link invariants are not equivalent---both can
distinguish knots that are indistinguishable in the other theory.

This odd analogue of Khovanov homology hints that there should be
a corresponding odd analogue of categorified quantum groups. We expect that
these odd categorified quantum groups will not be equivalent in any
appropriate sense to their even counterparts, though they will still
categorify the same quantum group.

In this paper we begin the program of odd categorification of quantum
groups by studying an odd analogue of the nilHecke algebra that
categorifies the negative half of $\mathbf{U}_q(\mf{sl}_2)$. This algebra
was introduced by Kang, Kashiwara, and Tsuchioka~\cite{KKT} several months ago.
We arrived at the odd nilHecke algebra independently via the categorification of positive half of quantum $\mf{gl}(1|1)$ by way of
Lipshitz-Ozsv\'{a}th-Thurston rings \cite{KhoGL12}. LOT rings provide
the odd counterpart of the nilCoxeter algebra, the subalgebra of the nilHecke
algebra generated by the divided difference operators. Enlarging LOT rings via
suitable generators (represented diagrammatically by dots) and relations
results in the odd counterpart of the nilHecke algebra. A different motivation
and approach to this odd nilHecke algebra can be found in~\cite{KKT} together
with its extension to odd quiver Hecke algebras for certain root systems.  Several years ago Weiqiang Wang \cite{Wang} (see also the work of Khongsap and Wang \cite{KW1,KW2,KW4}) introduced a so-called spin Hecke algebra and, as we learned via private communication, he was aware of a nil version of his construction which he would call the ``spin nilHecke algebra'' and which is identical to the odd nilHecke algebra.

Just as the nilHecke algebra has a close relationship with the
combinatorics of symmetric functions, the odd nilHecke algebra gives rise to an
odd analogue of the ring of symmetric functions. This ring is
noncommutative, yet shares many of the same combinatorics of symmetric
functions. Odd symmetric functions were defined by the first two authors in
\cite{EK}. Here we connect the theory of odd symmetric functions with the
odd nilHecke algebra and construct odd analogues of Schubert polynomials,
Schur polynomials and further develop their combinatorics, including the odd
Pieri rule.

 The nilHecke algebra is also pervasive throughout the theory of
Grassmannians and partial flag varieties. In particular, the nilHecke
algebra admits so-called cyclotomic quotients that are Morita
equivalent to the cohomology rings of Grassmannians. Here we study the
corresponding quotients of the odd nilHecke algebra and show that they are
Morita equivalent to certain noncommutative rings that seem to be odd
analogues of  cohomology rings of complex Grassmannians. We hope that
these rings will show a path to a rather special case of quantum (noncommutative)
geometry.

%
\subsection{Outline}
%

Section \ref{sec_osymm} introduces the main characters: odd analogues of divided difference operators, the nilHecke algebra, and symmetric polynomials.  As in the even case, the odd nilHecke algebra $\ONH_a$ is a matrix algebra over the ring of odd symmetric polynomials $\osym_a$.  Basic results relating these to their even counterparts are established, and odd analogues of some of the usual bases of symmetric functions are given.  An important technical result is the Omission Word Lemma (Lemma \ref{lem-owl}), which implies that the longest divided difference operator $D_a$ is left linear over the ring of odd symmetric polynomials, up to an involution (it is automatically right linear).  This investigation culminates in the definition of odd Schur polynomials and an odd Pieri rule describing products $s_\lambda\varepsilon_k$.

In the following Section \ref{sec-graphical-calculus}, we give a graphical description of the odd nilHecke algebra in the spirit of \cite{KL}.  The familiar planar isotopy relations only hold up to sign in the odd case; distant dots and crossings anticommute.  Much of Subsection \ref{subsec-box-notation} is devoted to keeping track of the resulting signs.  A useful tool here is a family of elements \eqref{eqn-0Hecke} which obey the (usual, not odd) 0-Hecke relations.

These 0-Hecke generators can be used to define idempotents $e_a$ which, when expressed diagrammatically, we refer to as ``thick lines.''  Section \ref{sec-thick} studies the resulting ``thick calculus,'' in analogy with \cite{KLMS}.  The twisted left linearity of $D_a$ is a key ingredient here, as it allows us to make sense of labelling thick lines by elements of $\osym_a$.  In Subsection \ref{subsec-orthogonal-idempotents}, using some technical lemmas, we are able to decompose the unit element of $ONH_a$ and $e_a\otimes e_b\in\ONH_a\otimes\ONH_b\subset\ONH_{a+b}$ into indecomposable orthogonal idempotents, see Theorems \ref{thm_nil_matrix} and \ref{thm_nil-eaeb}.  In Section \ref{sec-categorification}, we describe how these decompositions categorify the relations between divided powers in the positive part of Lusztig's integral form of quantum $\mathfrak{sl}_2$.

In the even case, categorified quantum $\mathfrak{sl}_2$ acts on a bimodule bicategory coming from the cohomology rings of complex Grassmannians \cite{Lau1,Lau2}.  Section \ref{sec-cyclotomic} gives odd analogues of these rings; as expected, they are quotients of the odd symmetric polynomials, and they have bases consisting of certain odd Schur functions.

\bigskip
\noindent {\bf Acknowledgments:}

The first author was supported by the NSF Graduate Research Fellowship Program.
The second author was partially supported by NSF grant DMS-0706924.  The third author was partially supported by NSF grant DMS-0855713 and by the Alfred P. Sloan foundation. All the authors would like to acknowledge partial support from Columbia University's RTG grant DMS-0739392.  A conversation between the first author and Pedro Vaz led to a substantial simplification in the proof of the OWL (Lemma \ref{lem-owl}).

%
\section{Odd symmetric polynomials and the odd nilHecke algebra} \label{sec_osymm}
%

%
\subsection{Odd symmetric polynomials}
%

%
\subsubsection{Defining the odd symmetric polynomials} \label{subsec-osymm-defn}
%

We define the ring of \textit{odd polynomials} to be the free unital associative algebra on skew-commuting variables $x_1,\ldots,x_a$,
\begin{equation}
\opol_a=\Z\langle x_1,\ldots,x_a\rangle/\langle x_ix_j+x_jx_i=0\text{ for }i\neq j\rangle.
\end{equation}
We let the symmetric group $S_a$ act on the degree $k$ part of $\opol_a$ as the tensor product of the permutation representation and the $k$-th tensor power of the sign representation.  That is, for $1\leq j\leq a-1$, the transposition $s_j\in S_a$ acts as the ring endomorphism
\begin{equation}\label{eqn-S-action}
s_i(x_j)=\begin{cases}
-x_{i+1}&\text{if }j=i,\\
-x_i&\text{if }j=i+1,\\
-x_j&\text{otherwise.}
\end{cases}
\end{equation}
The \textit{odd divided difference operators} are the linear operators $\partial_i$ ($1\leq i\leq a-1$) on $\Z\langle x_1,\ldots,x_a\rangle$ defined by
\begin{equation}\begin{split}
&\partial_i(1)=0,\\
&\partial_i(x_j)=\begin{cases}
1&\text{if }j=i,i+1,\\
0&\text{otherwise,}
\end{cases}\end{split}
\end{equation}
and the Leibniz rule
\begin{equation}\label{eqn-leibniz}
\partial_i(fg)=\partial_i(f)g+s_i(f)\partial_i(g)\text{ for all }f,g\in\Z\langle x_1,\ldots,x_a\rangle.
\end{equation}
It is easy to check from the definition of $\partial_i$ that for all $i$,
\begin{equation}
\partial_i(x_jx_k+x_kx_j)=0\text{ for }j\neq k,
\end{equation}
so $\partial_i$ descends to an operator on $\opol_a$.

Note that $\opol_a$ is left and right Noetherian and has no zero divisors, but does not satisfy the unique factorization property if $a\geq2$:
\begin{equation*}
x_1^2+x_2^2=(x_1-x_2)^2=(x_1+x_2)^2.
\end{equation*}

The following basic formulae in $\opol_a$ can be derived from the above.
\begin{equation}\begin{split}
&\partial_i(x_i-x_{i+1})=0,\qquad\partial_i(x_ix_{i+1})=0,\\
&\partial_i(x_i^{2m}+x_{i+1}^{2m})=0,\qquad\partial_i(x_i^mx_{i+1}^m)=0,\\
&\partial_i(x_i^m)=\sum_{j=0}^{m-1}(-1)^jx_{i+1}^jx_i^{m-1-j},\\
&\partial_i(x_{i+1}^m)=\sum_{j=0}^{m-1}(-1)^jx_i^jx_{i+1}^{m-1-j}.\\
\end{split}\end{equation}

\begin{prop} Considering $\partial_i$ and (multiplication by) $x_j$ as operators on $\opol_a$, the following relations hold in $\End(\opol_a)$:
\begin{eqnarray}
& & \partial_i^2 = 0 , \quad \partial_i \partial_{i+1}\partial_i =
 \partial_{i+1}\partial_i \partial_{i+1}, \label{eqn-onh-def1} \\
 & &  x_i \partial_i + \partial_i  x_{i+1} =1, \quad
   \partial_i x_i + x_{i+1}\partial_i = 1,   \label{eqn-onh-def2}  \\
& & x_i x_j + x_j x_i =0 \quad (i\neq j),  \quad
\partial_i \partial_j + \partial_j \partial_i =0 \quad (|i-j|>1), \label{eqn-onh-def3}  \\
& & x_i \partial_j +\partial_j x_i = 0 \quad (i\neq j,j+1). \label{eqn-onh-def4}
\end{eqnarray}
\end{prop}
\begin{proof} We first prove that  $\partial_i^2 (f) =0$ for any $f\in \opol_a$.
We can reduce to the case $i=1$ and $f$ being a monomial in $x_1,x_2$, and
then proceed by  induction on the degree of $f$. When the degree is zero, $f=1$
and $\partial_1^2(1)=0$.
Assume $\partial_1^2(f)=0$ for any monomial of degree $m$. Then
\begin{equation*}\begin{split}
\partial_1^2(x_1 f) &= \partial_1( f - x_2 \partial_1(f)) = \partial_1(f)
 - \partial_1(x_2) \partial_1(f) + x_1 \partial_1^2(f) = \partial_1(f) - \partial_1(f) = 0, \\
\partial_1^2(x_2 f) &= \partial_1( f - x_1 \partial_1(f)) = \partial_1(f)
 - \partial_1(x_1) \partial_1(f) + x_2 \partial_1^2(f) = 0 ,
\end{split}\end{equation*}
which takes care of the inductive step. Next, we verify relations of the second
type from \eqref{eqn-onh-def1}.
It suffices to assume that $i=1$ and the action is on a monomial in
$x_1, x_2, x_3$. We proceed by induction on the degree of the monomial.
When $f=1$, both sides act by $0$. Assuming that
$$\partial_1 \partial_2 \partial_1(f) = \partial_2 \partial_1 \partial_2(f)$$
we compute
\begin{equation*}\begin{split}
\partial_1 \partial_2 \partial_1 (x_1 f) &= \partial_1 \partial_2 (f - x_2 (\partial_1 (f) )) =
  \partial_1 ( \partial_2(f) - \partial_1 (f) + x_3 ( \partial_2 \partial_1 (f))) \\
&= \partial_1 \partial_2 (f)  -\partial_1^2 (f) - x_3(\partial_1 \partial_2 \partial_1 (f))=
 \partial_1 \partial_2 (f) - x_3(\partial_1 \partial_2 \partial_1 (f)) , \\
\partial_2 \partial_1 \partial_2 (x_1 f) &= \partial_2 \partial_1 ( - x_2 \partial_2(f)) =
\partial_2 (-\partial_2(f) + x_2 \partial_1 \partial_2 (f)) \\
& = \partial_1 \partial_2 (f) - x_3 (\partial_2 \partial_1 \partial_2(f)),
\end{split}\end{equation*}
implying $\partial_1 \partial_2 \partial_1 (x_1 f) = \partial_2 \partial_1 \partial_2 (x_1 f).$
Likewise,
\begin{equation*}\begin{split}
\partial_1 \partial_2 \partial_1 (x_2 f) &= \partial_1 \partial_2 (f - x_1 (\partial_1 (f) )) =
  \partial_1 ( \partial_2(f) + x_1 ( \partial_2 \partial_1 (f))) \\
& = \partial_1 \partial_2 (f)  +\partial_2\partial_1 (f) - x_2(\partial_1 \partial_2
\partial_1 (f)), \\
\partial_2 \partial_1 \partial_2 (x_2 f) &= \partial_2 \partial_1 ( f - x_3 \partial_2(f)) =
\partial_2 (\partial_1(f) + x_3 \partial_1 \partial_2 (f)) \\
& = \partial_1 \partial_2 (f) +\partial_2 \partial_1 (f) - x_2 \partial_2 \partial_1 \partial_2(f),
\end{split}\end{equation*}
proving the inductive step in this case.
Checking that the two actions are the same on $x_3 f$ is equally simple.

Relations \eqref{eqn-onh-def2} follow from the Leibniz rule \eqref{eqn-leibniz}:
\begin{eqnarray*}
\partial_i (x_{i+1} f) & = & \partial(x_{i+1}) f - x_i \partial_i(f) = f - x_i \partial_i(f),  \\
\partial_i (x_i f) & = &  f - x_{i+1} \partial_i (f).
\end{eqnarray*}
The first type of relation in \eqref{eqn-onh-def3} consists of defining relations in $\opol_a$, and the second
type in \eqref{eqn-onh-def3} can be checked by applying
$\partial_i \partial_j + \partial_j \partial_i$ to monomials in $x_1, \dots, x_a$ to get $0$.
The relation \eqref{eqn-onh-def4} is again a special case of \eqref{eqn-leibniz}.
\end{proof}

We can consider $\opol_a$ as a graded ring by declaring each $x_i$ to be homogeneous of degree $2$; the operators $\partial_i$ are all homogeneous of degree $-2$.  Since $\partial_i^2=0$, we can consider $\opol_a$ as a chain complex (taking the homological grading to be one-half of the grading $\deg(x_i)=2$, so $(\opol_a)_{2k}$ sits in homological degree $k$).  The multiplication operators
\begin{equation}\begin{split}
&\psi_k:(\opol_a)_k\to(\opol_a)_{k+1}\\
&\psi_k=\begin{cases}
x_i&k\text{ is even}\\
x_{i+1}&k\text{ is odd}
\end{cases}
\end{split}\end{equation}
give a chain homotopy between the identity and the zero maps, so the complex $((\opol_a)_k,\partial_i)$ is contractible for each $i,a$.  In particular, $\ker(\partial_i)=\im(\partial_i)$ for each $i$.  We define the ring of \textit{odd symmetric polynomials} to be the subring
\begin{equation}
\osym_a=\bigcap_{i=1}^{a-1}\ker(\partial_i)=\bigcap_{i=1}^{a-1}\im(\partial_i)
\end{equation}
of $\opol_a$.  Observe that
\begin{equation}\begin{split}
&\opol_a\otimes_\Z(\Z/2)\cong\pol_a\otimes_\Z(\Z/2)\\
&\osym_a\otimes_\Z(\Z/2)\cong\sym_a\otimes_\Z(\Z/2),
\end{split}\end{equation}
the usual (commutative) rings of polynomials and symmetric polynomials in $a$ variables over $\Z/2$.  In particular, both $\opol_a$ and $\osym_a$ are free abelian groups whose ranks in each degree are bounded below by those of $\pol_a$ and $\sym_a$, since $\qrk(\pol_a)=(\qdim)_{\Z/2}(\pol_a)$ and $\qrk(\sym_a)=(\qdim)_{\Z/2}(\sym_a)$.  Here, for a graded abelian group $V$ of finite rank in each degree and a graded vector space $W$ over a field $\mathbbm{F}$ of finite dimension in each degree,
\begin{equation}\begin{split}
&\qrk(V)=\sum_{i\in\Z}\rk(V_i)q^i,\\
&(\qdim)_{\mathbbm{F}}(W)=\sum_{i\in\Z}\dim_{\mathbbm{F}}(W_i)q^i.
\end{split}\end{equation}
Both are power series in $q^{\pm1}$.  It is clear that $\pol_a$ and $\opol_a$ have the same graded rank.

%
\subsubsection{Odd elementary symmetric polynomials and the size of $\opol_a$} \label{subsec_osymm_size}
%

Let
\begin{equation}\begin{split}
&[n]=\frac{q^n-q^{-n}}{q-q^{-1}},\\
&[n]!=[n][n-1]\cdots[2][1]
\end{split}\end{equation}
be the balanced $q$-numbers and $q$-factorial; both are elements of $\N[q,q^{-1}]$ and both reduce to their non-quantum analogues $n, n!$ as $q\to1$.    They are especially useful for keeping track of graded ranks.  For instance, the ring of symmetric polynomials is isomorphic to a graded polynomial algebra,
\begin{equation}
\sym_a\cong\Z[\varepsilon_1^\text{even},\varepsilon_2^\text{even},\ldots,\varepsilon_a^\text{even}],
\end{equation}
where $\varepsilon_k^\text{even}$ is the usual elementary symmetric polynomial
\begin{equation}
\varepsilon_k^\text{even}(x_1,\ldots,x_a)=\sum_{1\leq i_1<\cdots<i_k\leq a}x_{i_1}\cdots x_{i_a}
\end{equation}
of degree $2k$ (since we are declaring $\deg(x_i)=2$).  Hence the graded rank of $\sym_a$ is
\begin{equation}\label{eqn_qrk_symm}\begin{split}
\qrk(\sym_a)&=\prod_{i=1}^a\frac{1}{1-q^{2i}}\\
&=\frac{1}{(1-q^2)^a}\prod_{i=1}^a\left(q^{-i+1}\frac{q-q^{-1}}{q^a-q^{-a}}\right)\\
&=q^{-a(a-1)/2}\frac{1}{[a]!}\frac{1}{(1-q^2)^a},
\end{split}\end{equation}
an element of $\Z_{\geq0}\llbracket q^2\rrbracket$.  This equals
\begin{equation}
\qrk(\sym_a)=\frac{\qrk(\pol_a)}{\sum_{\sigma\in S_a}q^{\ell(\sigma)}},
\end{equation}
where $\ell$ denotes the Coxeter length function on $S_a$; this makes one think of ``taking a quotient of the ring of all polynomials by the symmetric group.''\\

\begin{prop}\label{prop_size_osymm} The rings of symmetric and odd symmetric polynomials have the same graded ranks:
\begin{equation}
\qrk(\osym_a)=\qrk(\sym_a)=q^{-a( a-1)/2}\frac{1}{[a]!}\frac{1}{(1-q^2)^a}.
\end{equation}
\end{prop}
To prove this proposition, we will have to develop odd analogues of the elementary symmetric polynomials.

\indent By analogy with the even case, we introduce the \text{odd elementary symmetric polynomials}
\begin{equation}\label{eqn-defn-e}
\varepsilon_k(x_1,\ldots,x_a)=\sum_{1\leq i_1<\cdots<i_k\leq a}\widetilde{x}_{i_1}\cdots\widetilde{x}_{i_k},\qquad\text{where }\widetilde{x}_i=(-1)^{i-1}x_i,
\end{equation}
for $1\leq k\leq a$.  If $k=0$ define $\varepsilon_k=1$, and if $k<0$ or $k>a$ define $\varepsilon_k=0$.  If we want to emphasize the number of variables in shorthand, we will write $\varepsilon_k^{(a)}$ for $\varepsilon_k(x_1,\ldots,x_a)$.  Modulo 2, the elementary symmetric and odd elementary symmetric polynomials are the same.  But with signs, the relations among the odd $\varepsilon_k$ are more complicated than mere commutativity.  The following lemma will give us enough relations to write down a presentation of $\osym_a$.

\begin{lem}\label{lemma-osymm-relations} The polynomial $\varepsilon_k(x_1,\ldots,x_a)$ is odd symmetric for $1\leq k\leq a$.  Furthermore, the following relations hold in the ring $\osym_a$:
\begin{equation}\label{eqn-e-relations}\begin{split}
&\varepsilon_i\varepsilon_{2m-i}=\varepsilon_{2m-i}\varepsilon_i\qquad(1\leq i,2m-i\leq a),\\
&\varepsilon_i\varepsilon_{2m+1-i}+(-1)^i\varepsilon_{2m+1-i}\varepsilon_i
=(-1)^i\varepsilon_{i+1}\varepsilon_{2m-i}+\varepsilon_{2m-i}\varepsilon_{i+1}\qquad(1\leq i,2m-i\leq a-1),\\
&\varepsilon_1\varepsilon_{2m}+\varepsilon_{2m}\varepsilon_1=2\varepsilon_{2m+1}\qquad(1<2m\leq a-1).
\end{split}\end{equation}
Note that the third is the $i=0$ case of the second.
\end{lem}

By the first relation, odd subscripts commute with odd subscripts and even subscripts commute with even subscripts.  These relations are enough to reduce any word $\varepsilon_{i_1}\cdots \varepsilon_{i_r}$ to a $\Z$-linear combination of terms of the form $\varepsilon_{j_1}\cdots \varepsilon_{j_s}$ with $j_1\geq\cdots\geq j_s$; hence the rank of $\osym_a$ in each degree is bounded above by that of $\sym_a$ in the same rank (cf. the proof of Proposition \ref{prop_size_osymm} below).

\begin{proof} The relations are all true in degrees $0,1$.  The third relation is the $i=0$ case of the second.  In $a=1$ variable, all the relations are clear.  We now prove the first and second relations simultaneously by induction on the number of variables $a$.  The equation
\begin{equation}\label{eqn-e-vars}
\varepsilon_m^{(a)}=\varepsilon_m^{(a-1)}+\varepsilon_{m-1}^{(a-1)}\widetilde{x}_a
\end{equation}
allows us to reduce calculations to fewer variables.  Now we compute (writing just $\varepsilon_j$ for $\varepsilon_j^{(a-1)}$):
\begin{equation}\begin{split}
\varepsilon_i^{(a)}\varepsilon_{2m-i}^{(a)}-\varepsilon_{2m-i}^{(a)}\varepsilon_i^{(a)}&=\left(\varepsilon_i+\varepsilon_{i-1}\widetilde{x}_a\right)\left(\varepsilon_{2m-i}+\varepsilon_{2m-i-1}\widetilde{x}_a\right)-\left(\varepsilon_{2m-i}+\varepsilon_{2m-i-1}\widetilde{x}_a\right)\left(\varepsilon_i+\varepsilon_{i-1}\widetilde{x}_a\right)\\
&=\left(\varepsilon_i\varepsilon_{2m-i}-\varepsilon_{2m-i}\varepsilon_i\right)+(-1)^{i-1}\left(\varepsilon_{i-1}\varepsilon_{2m-i-1}-\varepsilon_{2m-i-1}\varepsilon_{i-1}\right)\widetilde{x}_a^2\\
&\qquad\qquad+\left((-1)^i\varepsilon_{i-1}\varepsilon_{2m-i}-\varepsilon_{2m-i}\varepsilon_{i-1}+\varepsilon_i\varepsilon_{2m-i-1}-(-1)^i\varepsilon_{2m-i-1}\varepsilon_i\right)\widetilde{x}_a\\
&=0.
\end{split}\end{equation}
The second equality was grouping terms into powers of $\widetilde{x}_a$; the third equality used the induction hypothesis on the coefficients of 1 and of $\widetilde{x}_a^2$ (first relation) and on the coefficient of $\widetilde{x}_a$ (second relation).  The second relation is proved similarly.
\end{proof}

\begin{rem}The relations \eqref{eqn-e-relations} allow one to sort a product $\varepsilon_{i_1}\cdots\varepsilon_{i_k}$ into non-increasing order of subscripts (up to sign), modulo a set of terms which are ``lexicographically higher'' (that is, the word formed by their subscripts is greater in the lexicographic ordering) and in $2\Z\cdot\osym_a$.  This follows from a consequence of the odd degree relation: suppose $k<\ell$ and $k+\ell$ is odd.  Then
\begin{equation}\label{eqn-e-odd-sort}\begin{split}
&\varepsilon_k\varepsilon_\ell=\varepsilon_\ell\varepsilon_k+2\sum_{i=1}^k(-1)^{\binom{i}{2}}\varepsilon_{\ell+i}\varepsilon_{k-i}\qquad\text{if }k\text{ is even,}\\
&\varepsilon_k\varepsilon_\ell=-\varepsilon_\ell\varepsilon_k+2\sum_{i=1}^k(-1)^{\binom{i-1}{2}}\varepsilon_{\ell+i}\varepsilon_{k-i}\qquad\text{if }k\text{ is odd.}
\end{split}\end{equation}\end{rem}

\begin{proof}[Proof of Proposition \ref{prop_size_osymm}] Let $\osym_a^{\text{elem}}$ be the subring of $\osym_a$ generated by the odd elementary symmetric polynomials $\varepsilon_k(x_1,\ldots,x_a)$.  For a partition $\alpha=(\alpha_1,\ldots,\alpha_m)$ of $a$ written in non-increasing order, let
\begin{equation}
\varepsilon_\alpha=\varepsilon_{\alpha_1}\cdots \varepsilon_{\alpha_m}.
\end{equation}
Lemma \ref{lemma-osymm-relations} and the remark which follows it imply that these products span $\osym_a^{\text{elem}}$.  Hence the rank of each graded piece of $\osym_a^{\text{elem}}$ is bounded above by that of the corresponding graded piece of $\sym_a$.  Conversely, we have an isomorphism
\begin{equation}
\sym_a\otimes_\Z(\Z/2)\cong\osym_a\otimes_\Z(\Z/2)\cong\osym_a^{\text{elem}}\otimes_\Z(\Z/2).
\end{equation}
The identification of the first and third follows because the generators and relations of $\osym_a^{\text{elem}}$ and $\sym_a$ coincide mod 2; the identification of the first and second follows because mod 2, the action of the divided difference operators on $\pol_a$ and the action of the odd divided difference operators on $\opol_a$ coincide.  The graded rank of $\sym_a$ and the graded dimension of $\sym_a\otimes_\Z(\Z/2)$ are equal.  This bounds the graded rank of $\osym_a^{\text{elem}}$ below by that of $\sym_a$: indeed, dividing any linear relation between the $\varepsilon_\alpha$ in $\osym_a^{\text{elem}}$ by a high enough power of 2, we would obtain a linear relation between their reductions mod 2, a contradiction.  Thus $\qrk(\sym_a)=\qrk(\osym_a^{\text{elem}})$.\\
\indent To conclude that $\osym_a=\osym_a^{\text{elem}}$, we prove that both have free abelian complements in $\opol_a$.  For $\osym_a$, this is because if there were no free complement, some free direct summand (as a $\Z$-submodule) would be wholly divisible by an integer $d>1$.  But then we could divide generators of this summand by $d$.  The result would still be in the kernel of all the operators $\partial_i$, a contradiction.  As for $\osym_a^{\text{elem}}$, one checks that with respect to a lexicographic order on monomials, the highest order term of $\varepsilon_\alpha$ is $\undx^{\alpha}=x_1^{\alpha_1}\cdots x_r^{\alpha_m}$ with coefficient 1.  Now since $\osym_a^{\text{elem}}\subseteq\osym_a$ and both have free complements, the graded dimensions over $\Z/2$ of their reductions mod 2 coincide if and only if they are equal.  As both rings have mod 2 reduction isomorphic to $\sym_a$, these graded dimensions do coincide.  So $\osym_a^{\text{elem}}=\osym_a$, and we have established the formula
\begin{equation}
\qrk(\osym_a)=q^{-a(a-1)/2}\frac{1}{[a]!}\frac{1}{(1-q^2)^a}.
\end{equation}
\end{proof}

The following lemma is useful for passing between the rings $\osym_a$ and $\osym_{a-1}$.

\begin{lem}\label{lem_elem_vars} For any $k\geq0$,
\begin{equation}\begin{split}
\varepsilon_k^{(a-1)}&=\varepsilon_k^{(a)}-\varepsilon_{k-1}^{(a)}\widetilde{x}_a+\varepsilon_{k-2}^{(a)}
\widetilde{x}_a^2-\ldots+(-1)^k\widetilde{x}_a^k\\
&=\sum_{j=0}^k(-1)^j\varepsilon_{k-j}^{(a)}\widetilde{x}_a^j.
\end{split}\end{equation}
\end{lem}

\begin{proof} The polynomial $\varepsilon_k^{(a)}$ consists of terms without $\widetilde{x}_a$ and terms with $\widetilde{x}_a$.  The former add up precisely to $\varepsilon_k^{(a-1)}$.  The latter all appear as terms in $\varepsilon_{k-1}^{(a)}\widetilde{x}_a$; when we subtract off $\varepsilon_{k-1}^{(a)}\widetilde{x}_a$, the extra terms subtracted are precisely those which already had a factor of $\widetilde{x}_a$ in $\varepsilon_{k-1}^{(a)}$.  The extra terms are now those with a factor of $\widetilde{x}_a^2$, and they appear in $\varepsilon_{k-2}^{(a)}\widetilde{x}_a^2$.  Continuing in this fashion (inclusion-exclusion), the formula follows.
\end{proof}

By a slight abuse of notation, if $R\subseteq S$ is a subring and $s\in S$, we write $R[s]$ for the subring of $S$ generated by $R$ and $s$.
\begin{cor}\label{cor_adjoin_xn}
Inside $\opol_a$, $\osym_a[x_a]=\osym_{a-1}[x_a]$.
\end{cor}

\begin{proof} ``$\subseteq$'': Let $f\in\osym_a$.  Using skew-commutativity, we can move all factors of $x_a$ in any term of $f$ all the way to the right.  Collecting powers of $x_a$, we see each coefficient of a given power of $x_a$ is an element of $\osym_{a-1}$, so $\osym_a\subseteq\osym_{a-1}[x_a]$.  The converse ``$\supseteq$'' follows from the previous lemma.
\end{proof}

%
\subsubsection{Odd complete symmetric polynomials} \label{subsec_completes}
%

\begin{defn} For $k\geq1$, the $k$-th \textit{odd complete symmetric polynomial} is defined to be
\begin{equation}
h_k(x_1,\ldots,x_a)=\sum_{1\leq i_1\leq\cdots\leq i_k\leq a}\widetilde{x}_{i_1}\cdots\widetilde{x}_{i_k}.
\end{equation}
Also define $h_0=1$ and $h_k=0$ for $k<0$.
\end{defn}

\begin{lem}\label{lemma-e-h-relation} The polynomials $\varepsilon_k$ and $h_k$ in $\opol_a$ satisfy
\begin{equation}\label{eqn-e-h-relation}
\sum_{k=0}^m(-1)^{\frac{1}{2}k(k+1)}\varepsilon_kh_{m-k}=0
\end{equation}
for all $m\geq1$.\end{lem}
\begin{proof} We proceed by induction on the number of variables $a$, the case $a=1$ being clear.  Let $\varepsilon_k,h_k$ denote the odd elementary and odd complete polynomials in $a$ variables, so that
\begin{equation*}
\varepsilon_k^{(a+1)}=\varepsilon_k+\varepsilon_{k-1}\xt_{a+1},\qquad h_k^{(a+1)}=\sum_{j=0}^kh_{k-j}\xt_{a+1}^j.
\end{equation*}
Plugging these expressions into the left-hand side of equation \eqref{eqn-e-h-relation} in $a+1$ variables,
\begin{equation*}\begin{split}
\sum_{k=0}^n(-1)^{\frac{1}{2}k(k+1)}\varepsilon_k^{(a+1)}h_{n-k}^{(a+1)}&=\sum_{k=0}^n(-1)^{\frac{1}{2}k(k+1)}(\varepsilon_k+\varepsilon_{k-1}\xt_{a+1})\left(\sum_{j=0}^{n-k}h_{n-k-j}\xt_{a+1}^j\right)\\
&=\sum_{k=0}^n\sum_{j=0}^{n-k}(-1)^{\frac{1}{2}k(k+1)}\left(\varepsilon_kh_{n-k-j}+(-1)^{n-k-j}\varepsilon_{k-1}h_{n-k-j}\xt_{a+1}\right)\xt_{a+1}^j\\
&=\sum_{j=0}^n\sum_{k=0}^{n-j}(-1)^{\frac{1}{2}k(k+1)}\varepsilon_kh_{n-k-j}\xt_{a+1}^j\\
&\qquad+\sum_{j=1}^{n+1}\sum_{k=0}^{n-j+1}(-1)^{\frac{1}{2}k(k+1)+n-k+j-1}\varepsilon_{k-1}h_{n-k-j+1}\xt_{a+1}^j.
\end{split}\end{equation*}
In the last equality, the order of summation was reversed and the second term was re-indexed, $j\mapsto j-1$.  Consider the last line: the inner sum of the first term is zero unless $j=n,k=0$ (induction hypothesis).  Also, the second term is 0 when $k=0$.  Removing these vanishing terms, removing boundary terms from the summations, re-combining the two summation terms, and re-indexing $k\mapsto k+1$, this equals
\begin{equation*}
\ldots=\xt_{a+1}^n+\sum_{j=1}^{n+1}(-1)^{n+j-1}\left(\sum_{k=0}^{n-j}(-1)^{\frac{1}{2}k(k+1)}\varepsilon_kh_{n-k-j}\right)\xt_{a+1}^j.
\end{equation*}
The inner sum here is zero unless $j=n,k=0$ (induction hypothesis), in which case it cancels the $\xt_{a+1}^n$.  So the entire expression equals zero, and we are done.
\end{proof}

Equation \eqref{eqn-e-h-relation} can be used inductively to solve for each $h_m$ as a polynomial in the various $\varepsilon_k$, so each $h_m$ is indeed odd symmetric.

\begin{lem}\label{lem-h-odrs} The polynomials $h_k$ satisfy the same relations as the $\varepsilon_k$ in the ring $\osym_a$:
\begin{equation}\begin{split}
&h_ih_{2m-i}=h_{2m-i}h_i\qquad(1\leq i,2m-i\leq a),\\
&h_ih_{2m+1-i}+(-1)^ih_{2m+1-i}h_i=(-1)^ih_{i+1}h_{2m-i}+h_{2m-i}h_{i+1}\qquad(1\leq i,2m-i\leq a-1),\\
&h_1h_{2m}+h_{2m}h_1=2h_{2m+1}\qquad(1<2m\leq a-1).
\end{split}\end{equation}
Furthermore, we have the following mixed relations:
\begin{equation}\begin{split}
&\varepsilon_ih_{2m-i}=h_{2m-i}\varepsilon_i\qquad(1\leq i,2m-i\leq a),\\
&\varepsilon_ih_{2m+1-i}+(-1)^ih_{2m+1-i}\varepsilon_i=(-1)^i\varepsilon_{i+1}h_{2m-i}+h_{2m-i}\varepsilon_{i+1}\qquad(1\leq i,2m-i\leq a-1).
\end{split}\end{equation}
\end{lem}
\begin{proof} The proofs of all these relations are similar to the proof of Lemma \ref{lemma-osymm-relations}, using equation \eqref{eqn-e-vars} and its complete polynomial analogue
\begin{equation}
h_m^{(a)}=\sum_{j=0}^mh_{k-j}^{(a-1)}\widetilde{x}_a^j.
\end{equation}
\end{proof}
As in the case of elementary functions, define $h_\alpha=h_{\alpha_1}\cdots h_{\alpha_m}$ for a partition $\alpha=(\alpha_1,\ldots,\alpha_m)$.

In the even case, perhaps the most manifestly symmetric functions are the so-called monomial symmetric functions
\begin{equation}
m_\lambda=\sum_\alpha \undx^{\alpha}\qquad\text{(sum over all distinct permutations }\alpha\text{ of }\lambda\text{)}
\end{equation}
associated to a partition $\alpha$.  However, for certain $\lambda$, no such analogous sum yields an odd symmetric polynomial.  An equivalent definition of these functions in the even setting is that they form the basis dual to the basis $\lbrace h_\lambda\rbrace_\lambda$ with respect to a standard bilinear form.  For an approach to odd symmetric functions along these lines, see \cite{EK}.

%
\subsection{The odd nilHecke algebra} \label{sec_onh}
%

By analogy with the even case \cite{Man}, we define the \text{odd nilHecke ring} $\ONH_a$ to be the graded unital associative ring generated by elements $x_1,\ldots,x_a$ of degree 2 and elements $\partial_1,\ldots,\partial_{a-1}$ of degree $-2$, subject to the relations \eqref{eqn-onh-def1}, \eqref{eqn-onh-def2}, \eqref{eqn-onh-def3}, \eqref{eqn-onh-def4}, which we repeat here:
\begin{eqnarray*}
& & \partial_i^2 = 0 , \quad \partial_i \partial_{i+1}\partial_i =
 \partial_{i+1}\partial_i \partial_{i+1},\\
 & &  x_i \partial_i + \partial_i  x_{i+1} =1, \quad
   \partial_i x_i + x_{i+1}\partial_i = 1,\\
& & x_i x_j + x_j x_i =0 \quad (i\neq j),  \quad
\partial_i \partial_j + \partial_j \partial_i =0 \quad (|i-j|>1),  \\
& & x_i \partial_j +\partial_j x_i = 0 \quad (i\neq j,j+1).
\end{eqnarray*}
We define the \text{odd nilCoxeter ring} $\ONC_a$ to be the graded subring generated by the $\partial_i$'s (this is the LOT ring of \cite{LOT}).  As a consequence of these relations, $\ONH_a$ and $\ONC_a$ have natural representations on $\opol_a$. The $\Z$-grading on $\ONH_a$ induces a $\Z/2$-grading given by dividing the $\Z$-grading by 2 and then reducing mod 2. For $ f \in \ONH_a$ we write $\deg_s(f)$ for the super degree of $f$.

For each $w\in S_a$, choose a reduced expression $w=s_{i_1}\cdots s_{i_\ell}$ in terms of simple transpositions $s_i=(i\quad i+1)$.  Define
\begin{equation}
\partial_w=\partial_{i_1}\cdots\partial_{i_\ell}.
\end{equation}
For $w=w_0$ we make a particular choice of word and re-name the operator,
\begin{equation*}
D_a=\partial_{w_0}=\partial_1(\partial_2\partial_1)\cdots(\partial_{a-1}\cdots\partial_1).
\end{equation*}
Since the $\partial_i$ satisfy a signed version of the relations of the simple transpositions $s_i$, the definition of $\partial_w$ is almost independent of choice of reduced expression for $w$---the only ambiguity is an overall sign.  For $w,w'\in S_a$, the formula
\begin{equation}\label{eqn_compose_oddops}
\partial_w\partial_{w'}=\begin{cases}\pm\partial_{ww'}&\text{if }\ell(ww')=\ell(w)+\ell(w'),\\
0&\text{otherwise}\end{cases}
\end{equation}
is proven just as in the even case \cite{Man}.

When no confusion will result, if $A=(r_1,\ldots,r_a)$ is an $a$-tuple of integers we define
\begin{equation}
\undx^A=x_1^{r_1}\cdots x_a^{r_a}.
\end{equation}
Now it is clear that the sets
\begin{equation}\label{eqn-onh-bases}
\lbrace \undx^A\partial_w\rbrace_{w\in S_a,A\in\N^n}\text{ and }\lbrace\partial_w \undx^A\rbrace_{w\in S_a,A\in\N^n}
\end{equation}
generate $\ONH_a$ (for us, $0\in\N$).  In fact, they are linearly independent as well.  To prove this, we will introduce odd Schubert polynomials.

Let $\delta_a = (a-1,a-2, \dots, 1,0)$. In what follows we will make repeated use of the monomial
\begin{equation}\begin{split}
&\undx^{\delta_a}=x_1^{a-1}x_2^{a-2}\cdots x_{a-1}^1x_a^0.
\end{split}\end{equation}
For $w\in S_a$, define the corresponding \textit{odd Schubert polynomial} $\sch_w\in\opol_a$ by
\begin{equation}
\sch_w(x_1,\ldots,x_a)=\partial_{w^{-1}w_0}(\undx^{\delta_a}),
\end{equation}
where $w_0$ is the longest element of $S_a$.  As in the definition of the $\partial_w$, this is independent of choice of reduced expression for $w$, up to sign.  The degree of $\sch_w$ is $2\ell(w)$.  Equation \eqref{eqn_compose_oddops} implies
\begin{equation}\label{eqn_oddop_schubert}
\partial_u\sch_w=\begin{cases}\pm \sch_{wu^{-1}}&\text{if }\ell(wu^{-1})=\ell(w)-\ell(u),\\
0&\text{otherwise.}
\end{cases}\end{equation}
\begin{lem} Let $e\in S_a$ be the identity.  Then $\sch_e=(-1)^{\binom{a}{3}}$.\end{lem}
\begin{proof}  This is a simple calculation, by induction on $a$.  Alternatively, it follows from the graphical arguments in Proposition \ref{prop_ea_idemp} (whose proof does not depend on the present claim).
\end{proof}
It follows that
\begin{equation}\label{eqn_onh_schubert}\begin{split}
&\text{for }\ell(w)<\ell(u),\quad(\undx^A\partial_u)(\sch_w)=0,\\
&\text{for }\ell(w)=\ell(u),\quad(\undx^A\partial_u)(\sch_w)=\begin{cases}\pm \undx^A&\text{if }w=u^{-1}\\0&\text{otherwise.}
\end{cases}\end{split}\end{equation}
\begin{prop}\label{prop_onh_li} There are no linear relations among the images of the $\lbrace \undx^A\partial_w\rbrace_{w\in S_a,A\in\N^n}$ in $\End(\opol_a)$, nor among the images of the $\lbrace\partial_w \undx^A\rbrace_{w\in S_a,A\in\N^n}$.  Thus the natural representations of the odd nilCoxeter and odd nilHecke rings on $\opol_a$ are faithful, and these rings have graded ranks
\begin{equation}\begin{split}
&\qrk(\ONC_a)=q^{-a(a-1)/2}[a]!,\\
&\qrk(\ONH_a)=q^{-a(a-1)/2}[a]!\frac{1}{(1-q^2)^a}.
\end{split}\end{equation}
\end{prop}
\begin{proof} Since $\ONH_a$ is finite dimensional in each degree, the relations defining $\ONH_a$ imply that it suffices to prove the Proposition for either one of the two spanning sets; we do so for the set $\lbrace \undx^A\partial_w\rbrace$.  By equation \eqref{eqn_onh_schubert},
\begin{equation}
(\undx^A\partial_u)(\sch_e)=\begin{cases}(-1)^{\binom{a}{3}}\undx^A&\text{if }u=e,\\
0&\text{otherwise.}\end{cases}
\end{equation}
Thus no element $\undx^A\partial_e$ is a linear combination of any other spanning set elements.  Proceeding by induction on the Coxeter length of $w\in S_a$, suppose that for all $v\in S_a$ with $\ell(v)<\ell(w)$, no element $\undx^A\partial_v$ is a linear combination of any other spanning set elements.  Then by equation \eqref{eqn_onh_schubert}, if $\undx^A\partial_u$ is a linear combination of other terms $\undx^A\partial_v$, all the $v$ must be of shorter Coxeter length than $u$; but by induction, there is no such relation.\end{proof}

Our next goal is to express $\opol_a$ as a free left and right module over $\osym_a$.  The following lemma describes the basis we will use.  Our proofs of Lemma \ref{lem_sch_zbasis} and Proposition \ref{prop_opol_free_mod} closely follow the proofs of Propositions 2.5.3 and 2.5.5 of \cite{Man}, respectively.
\begin{lem}\label{lem_sch_zbasis} Let
\begin{equation}\begin{split}
\mathcal{H}_a&=\text{span}_\Z\lbrace \undx^A\in\opol_a:A\leq\delta_a\text{ termwise}\rbrace\\
&=\text{span}_\Z\lbrace x_1^{a_1}\cdots x_a^{a_a}:a_i\leq a-i\text{ for }1\leq i\leq a\rbrace.
\end{split}\end{equation}
Then the odd Schubert polynomials $\lbrace\sch_w(x)\rbrace_{w\in S_a}$ are an integral basis for $\mathcal{H}_a$.\end{lem}
\begin{proof} It is immediate from the definition of the Schubert polynomials that they are all contained in $\mathcal{H}_a$, as the operators $\partial_i$ only decrease exponents from $\undx^{\delta_a}$.  Since the set of Schubert polynomials and the defining basis of $\mathcal{H}_a$ both have $a!$ elements, it suffices to show linear independence and unimodularity.  Suppose
\begin{equation}
\sum_{w\in S_a}c_w\sch_w(x)=0
\end{equation}
with $c_w\in\Z$.  Then by applying various operators $\partial_u$ (as in the proof of Proposition \ref{prop_onh_li}), we see that all the $c_w=0$, proving linear independence over $\Q$.  So any $f\in \mathcal{H}_a$ is a rational linear combination of Schubert polynomials; but applying $\partial_u$'s to an expression $f=\sum_wc_w\sch_w$, we see that each $c_w\in\Z$.\end{proof}

\begin{prop}\label{prop_opol_free_mod} $\opol_a$ is a free left and right $\osym_a$-module of graded rank $q^{a(a-1)/2}[a]!$, with a homogeneous basis given by the odd Schubert polynomials $\lbrace\sch_w\rbrace_{w\in S_a}$.\end{prop}
\begin{proof} We will show that multiplication
\begin{equation}
\osym_a\otimes \mathcal{H}_a\to\opol_a
\end{equation}
is an isomorphism of abelian groups, realizing $\opol_a$ as a free left $\osym_a$-module; a similar proof shows it is a free right module with the same basis.\\
\indent Our first claim is that any $f\in\opol_a$ can be expressed in the form
\begin{equation}
f=\sum_{k=1}^{a-1}\sum_j\ell_{k,j}h_{k,j}x_a^k\qquad h_{k,j}\in \mathcal{H}_{a-1},\quad\ell_{k,j}\in\osym_a.
\end{equation}
This being clear for $a=1$, we proceed by induction.  Expand a given $f\in\opol_a$ in powers of $x_a$,
\begin{equation}
f=\sum_kx_a^kf_k=\sum_k\sum_{i=1}^{a-2}\sum_jx_a^k\ell_{i,j,k}h_{i,j,k}x_{a-1}^i,
\end{equation}
where $h_{i,j,k}\in \mathcal{H}_{a-2}$, $\ell_{i,j,k}\in\osym_{a-1}$, and $f_k\in\opol_{a-1}$ for all $i,j,k$.  Since $h_{i,j,k}x_{a-1}^i\in \mathcal{H}_{a-1}$, we can re-label and re-index to obtain an expression
\begin{equation}
f=\sum_{j,k}x_a^k\ell_{k,j}h_{k,j}\qquad h_{k,j}\in \mathcal{H}_{a-1},\quad\ell_{k,j}\in\osym_{a-1}.
\end{equation}
By Corollary \ref{cor_adjoin_xn}, each $x_a^k\ell_{k,j}$ can be re-written as a sum over terms of the form $x_a^{k'}\ell$, where $\ell\in\osym_a$ and $1\leq k'\leq a-1$.  Re-indexing and collecting terms again, this proves the claim.\\
\indent The above claim implies surjectivity of the multiplication map.  Injectivity follows from a graded rank count.  We have shown that $\opol_a$ is a free (left and right) $\osym_a$-module of graded rank
\begin{equation}
(\qrk)_{\osym_a}(\opol_a)=q^{a(a-1)/2}[a]!,
\end{equation}
since the right-hand side is equal to $\sum_wq^{2\ell(w)}=\sum_wq^{\deg(\sch_w(x))}$.
\end{proof}

We briefly recall the grading on matrix algebras over graded rings.  Let $f(q)$ be a Laurent series in $q$ and let $h(q)$ be a Laurent polynomial in $q$.  If $A$ is a graded ring with $\qrk(A)=f(q)$ and $M$ is a graded $A$-module with $(\qrk)_A(M)=h(q)$, then
\begin{equation}
\qrk(\End_A(M))=f(q)h(q)h(q^{-1}).
\end{equation}

\begin{cor}  \label{cor_nilmatrix}
The natural action of $\ONH_a$ on $\opol_a$ by multiplication and odd divided difference operators is an isomorphism
\begin{equation}
\varphi:\xymatrix{\ONH_a\ar[r]^-\cong&\End_{\osym_a}(\opol_a)\cong\Mat_{q^{a(a-1)/2}[a]!}(\osym_a).}
\end{equation}
\end{cor}

\begin{proof} Since the $\ONH_a$ action is by linearly independent operators (see proof of Proposition \ref{prop_onh_li}), $\varphi\otimes_\Z\Q$ (and hence $\varphi$) is injective.  Both $\ONH_a$ and $\End_{\osym_a}(\opol_a)$ have the same graded rank, so $\varphi$ is surjective as well.\end{proof}

\begin{prop} For $a\geq2$, the center of $\ONH_a$ is the ring of symmetric polynomials in the squared variables $x_1^2,\ldots,x_a^2$.  This coincides with the center of $\osym_a$.\end{prop}
\begin{proof} Since $\ONH_a$ is a matrix ring over $\osym_a$, their centers are the same when we view $\osym_a\subset\ONH_a$ as scalar multiples of the identity.  For the duration of this proof, then, we will just refer to ``central'' (skew) polynomials.  First claim: \textit{if} $f$ \textit{is central, then} $f$ \textit{is a polynomial in the squared variables} $x_1^2,x_2^2,\ldots,x_a^2$.  To see this, fix some $1\leq j\leq a$ and expand the skew polynomial $f$ as a polynomial in $x_j$,
\begin{equation*}
f=\sum_{k\geq0}f_kx_j^k.
\end{equation*}
Then for each $k,\ell\geq0$, let $f_{k,\ell}$ be the degree $\ell$ part of $f_k$, so
\begin{equation*}
f=\sum_{k,\ell\geq0}f_{k,\ell}x_j^k.
\end{equation*}
Multiplying this equation by $x_j$ on the left and on the right and comparing the results, it follows that the degree of each $f_{k,\ell}$ must be even.  Doing this for each $j$ separately, the first claim follows.

Second claim: \textit{if} $f$ \textit{is a polynomial in the squared variables, then} $\partial_i(f)=0$ \textit{if and only if} $s_i(f)=f$.  First, suppose $s_i(f)=f$.  Then $f$ is a linear combination of terms which are a product of a factor not involving $x_i,x_{i+1}$ and a factor of the form $x_i^{2k}x_{i+1}^{2\ell}+x_i^{2\ell}x_{i+1}^{2k}$.  The operator $\partial_i$ annihilates both sorts of factor, so $\partial_i(f)=0$.  Conversely, suppose $\partial_i(f)=0$; without loss of generality suppose $f$ is of homogeneous degree.  Expand
\begin{equation*}
f=\sum_{k,\ell\geq0}f_{k,\ell}x_i^{2k}x_{i+1}^{2\ell},
\end{equation*}
where $f_{k,\ell}$ is a polynomial in the variables $x_j$ for $j\neq i,i+1$.  So
\begin{equation*}\begin{split}
0&=\partial_i(f)=\sum_{k,\ell\geq0}f_{k,\ell}\partial_i(x_i^{2k}x_{i+1}^{2\ell})\\
&=\sum_{k,\ell\geq0}(x_i^2x_{i+1}^2)^\ell\left(f_{k,\ell}\partial_i(x_i^{2(k-\ell)})+f_{\ell,k}\partial_i(x_{i+1}^{2(k-\ell)})\right).
\end{split}\end{equation*}
By decreasing induction on $|k-\ell|$, this implies that $f_{k,\ell}=f_{\ell,k}$ for all $k,\ell$, that is, $s_i(f)=f$.

We now use the second claim to show that a skew polynomial $f$ is central if and only if it is a symmetric polynomial in the squared variables $x_1^2,\ldots,x_a^2$.  By the second claim, $f$ a symmetric polynomial in the squared variables if and only if it is a polynomial in the squared variables and $\partial_i(f)=0$ for all $i$.  For such an $f$,
\begin{equation*}
(\partial_if-f\partial_i)(g)=\partial_i(f)g+s_i(f)\partial_i(g)-f\partial_i(g)=0,
\end{equation*}
so $f$ commutes with all divided difference operators.  So $f\in Z(\ONH_a)$ if $f$ is a symmetric polynomial in the squared variables.  Conversely, using the above observations, it is easy to see that all symmetric polynomials in the squared variables are in $Z(\osym_a)$.\end{proof}

%
\subsubsection{Odd symmetrization}
%

\begin{lem}\label{lemma-w0-e} The longest element of $S_a$ acts on odd elementary polynomials as
\begin{equation}\label{eqn-w0-e}
(\varepsilon_k)^{w_0}=(-1)^{\binom{k}{2}+k\binom{a-1}{2}}\varepsilon_k.
\end{equation}
In particular, $(\osym_a)^{w_0}=\osym_a$.\end{lem}
\begin{proof} The set of monomials appearing in $\varepsilon_k$ is unchanged by $w_0$, so we need only consider the incurred sign.  The action on monomials is
\begin{equation*}
x_{i_1}\cdots x_{i_k}\mapsto(-1)^{k\binom{a}{2}}x_{a+1-i_1}\cdots x_{a+1-i_k}.
\end{equation*}
The sign appears because elementary transpositions act on variables $x_i$ with a minus sign.  Next, a sign of $(-1)^{\binom{k}{2}}$ is incurred in sorting the right hand side into ascending order.  Finally, remember that the monomials in $\varepsilon_k$ appear with tildes, $\widetilde{x}_i=(-1)^{i-1}x_i$.  The sign difference between removing the tildes on the left monomial and adding them in on the right monomial is $(-1)^{k(a-1)}$.  Putting all these signs together, the sign is as described in the statement of the lemma.  Since products of odd elementary polynomials are a basis of $\osym_a$, it follows that $(\osym_a)^{w_0}=\osym_a$.\end{proof}

\begin{rem}
For $a>2$ it is easy to see that given $s_j \in S_a$, the action by $s_j$ does not preserve the ring of odd symmetric functions.  For example
\[
 s_1( \varepsilon_1(\xt_1, \xt_2, \xt_3)) = \xt_2 + \xt_1 -\xt_3,
\]
which is not odd symmetric.
\end{rem}

Recall our fixed choice of Coxeter word for the longest Weyl group element,
\begin{equation}
D_a=\partial_{w_0}=\partial_1(\partial_2\partial_1)\cdots(\partial_{a-1}\partial_{a-2}\cdots\partial_1).
\end{equation}
A useful $\Z$-linear map is
\begin{equation}\label{eqn-odd-symmetrization}\begin{split}
&\Ss:\opol_a\to\osym_a\\
&f\mapsto(-1)^{\binom{a}{3}}\left(D_a(f\undx^{\delta_a})\right)^{w_0},
\end{split}\end{equation}
which we call \textit{odd symmetrization}.  The name comes from the fact that, as we will prove in this section, $\Ss$ is the projection operator from $\opol_a$ onto its lowest-degree indecomposable summand, the subring $\osym_a$.  In order to prove this, we first establish a few lemmas.

A word $w=s_{i_1}\cdots s_{i_r}$ in the symmetric group $S_a$ can act on a skew polynomial $f$ in two ways:
\begin{itemize}
\item act by $\partial_{i_1}\cdots\partial_{i_r}$, as an element of the odd nilCoxeter ring,
\item act by $s_{i_1}\cdots s_{i_r}$, as an element of $S_a$ (equation \eqref{eqn-S-action}).
\end{itemize}
One way to hybridize and keep track of these two actions is to equip $w$ with a function $\xi:\lbrace1,2,\ldots,r\rbrace\to\lbrace0,1\rbrace$ and to say
\begin{equation*}
\text{the simple transposition }s_{i_j}\text{ acts as}
\begin{cases}
s_{i_j}		&	\text{if }\xi(j)=0	\\
\partial_{i_j}	&	\text{if }\xi(j)=1.
\end{cases}
\end{equation*}
We will refer to the resulting operator as the \textit{generalized action} of the pair $(w,\xi)$, and denote its action on a polynomial $f$ by $w^\xi\cdot f$.  Give such a pair, define its \textit{omission word} $w^\xi_\text{om}$ to be the sub-word of $w$ consisting of just those $s_{i_j}$ such that $\xi(j)=0$; that is, the sub-word of $w$ corresponding to those transpositions which act via $S_a$ rather than the odd nilCoxeter ring.  In this language, the Leibniz rule \eqref{eqn-leibniz} can be generalized,
\begin{equation}\label{eqn-generalized-leibniz}
\partial_w(fg)=\partial_{i_1}\cdots\partial_{i_r}(fg)=\sum_\xi(w^\xi\cdot f)\partial_{w^\xi_\text{om}}(g).
\end{equation}
The sum is over all $2^{\ell(w)}$ possible choices of $\xi$.  Note that the action of $w^\xi$ is a generalized action, while the action of $w^\xi_\text{om}$ is an action by odd divided difference operators.

\begin{lem}[Omission Word Lemma (OWL)]\label{lem-owl} Suppose $w$ is a reduced expression for $w_0$.  For any $\xi$ as above, either
\begin{enumerate}
\item $w^\xi\cdot f=0$ for all $f\in\osym_a$,
\item $\xi(j)=0$ for all $j=1,\ldots,r$, or
\item the omission word $w^\xi_\text{om}$ is non-reduced.
\end{enumerate}\end{lem}

In order to prove the OWL, we first introduce odd divided difference operators for non-adjacent transpositions.  Notation: let $s_{k,\ell}$ be the transposition of $k$ and $\ell$ in the symmetric group $S_a$, even if $|k-\ell|>1$.  For $1\leq i,j\leq a$ and $i\neq j$, define the corresponding odd divided difference operator $\partial_{i,j}$ by
\begin{equation}\begin{split}
&\partial_{i,j}(x_h)=\begin{cases}1&\text{if }h=i,j,\\0&\text{if }h\neq i,j,\end{cases}\\
&\partial_{i,j}(fg)=\partial_{i,j}(f)g+s_{i,j}(f)\partial_{i,j}(g).
\end{split}\end{equation}
Note that $\partial_{i,i+1}=\partial_i$ and that $\partial_{i,j}$ is homogeneous of degree $-2$.  It is not true in general that the kernel of $\partial_{i,j}$ contains $\osym_a$ (unless, of course, $j=i+1$).

\begin{lem}\label{lem-oddop-1}
\begin{enumerate}
\item For any $i\neq j$ and $k\neq\ell$, we have
\begin{equation}
\partial_{i,j}s_{k,\ell}+s_{k,\ell}\partial_{s_{k,\ell}(i,j)}=0
\end{equation}
as operators on $\opol_a$.  By $s_{k,\ell}(i,j)$, we mean the pair obtained by applying the transposition $s_{k,\ell}$ to $i$ and $j$.  In particular, $\partial_{i,j}$ and $s_{k,\ell}$ anticommute if $\lbrace i,j\rbrace\cap\lbrace k,\ell\rbrace=\varnothing$.
\item If $\lbrace i,j\rbrace\cap\lbrace k,\ell\rbrace=\varnothing$, then
\begin{equation}
\partial_{i,j}\partial_{k,\ell}+\partial_{k,\ell}\partial_{i,j}=0
\end{equation}
as operators on $\opol_a$.
\end{enumerate}\end{lem}
\begin{proof} For both statements: first check on $x_h$ and then induct using the Leibniz rule.\end{proof}

\begin{lem}\label{lem-oddop-2} Suppose $1\leq i<j\leq a$ and $f\in\osym_a$.  Then
\begin{equation}\label{eqn-lem-oddop-2}
\partial_{i+1}\partial_{i+2}\cdots\partial_{j-1}\partial_{i,j}(f)=0.
\end{equation}
\end{lem}
\begin{proof} The proof is by a slightly complicated induction.  For now, consider all $i,j$ simultaneously.

\textit{Step 1:} We first prove the lemma in the case $f=\varepsilon_k$, $k\geq0$.  Let $D_{i,j}'=\partial_{i+1}\partial_{i+2}\cdots\partial_{j-1}\partial_{i,j}$.  A divided difference operator $\partial_{p,q}$ applied to a monomial $x_J=x_{j_1}\cdots x_{j_\ell}$ vanishes if and only if $x_p$ and $x_q$ occur the same number of times in $J$.  It follows that if the tuple $J$ has no repeated entries, then $D_{i,j}'(x_J)\neq0$ if and only if $J$ contains either: all of $i+1,i+2,\ldots,j$ but not $i$, or all of $i,i+1,\ldots,j-1$ but not $j$.  The terms of $\varepsilon_k$ on which $D_{i,j}'$ does not vanish match into pairs
\begin{equation*}
\xt_I\xt_{i+1}\cdots\xt_j\xt_J,\quad\xt_I\xt_i\cdots\xt_{j-1}\xt_J,
\end{equation*}
where $I$ and $J$ are tuples whose degrees sum to $k-(j-i)$.  We compute
\begin{equation*}\begin{split}
&D_{i,j}'\left(\xt_I\xt_{i+1}\cdots\xt_j\xt_J+\xt_I\xt_i\cdots\xt_{j-1}\xt_J\right)\\
&\qquad=\partial_{i+1}\cdots\partial_{j-1}\partial_{i,j}\left(\xt_I\xt_{i+1}\cdots\xt_j\xt_J+\xt_I\xt_i\cdots\xt_{j-1}\xt_J\right)\\
&\qquad=\pm\partial_{i+1}\cdots\partial_{j-1}\partial_{i,j}\left(x_Ix_{i+1}\cdots x_jx_J+(-1)^\ell x_Ix_i\cdots x_{j-1}x_J\right)\\
&\qquad=\pm\partial_{i+1}\cdots\partial_{j-1}\left((-1)^{|I|+\ell-1}x_Ix_{i+1}\cdots x_{j-1}x_J+(-1)^\ell(-1)^{|I|}x_Ix_{i+1}\cdots  x_{j-1}x_J\right)\\
&\qquad=0.
\end{split}\end{equation*}
It follows that $D_{i,j}'(\varepsilon_k)=0$.

\textit{Step 2:} We now induct on the degree of $f$.  In each degree, we may assume $f$ is a product $\varepsilon_{k_1}\cdots\varepsilon_{k_r}$.  Step 1 covered the base cases $r=1$ and degree 1.  In suffices, then, to take $f=gh$, where both $g$ and $h$ are odd symmetric and have positive degree.  We will prove equation \eqref{eqn-lem-oddop-2} simultaneously with the following claim: for $1\leq\ell\leq j-i$,
\begin{equation}\label{eqn-oddop-subclaim}
\partial_{i+1}\cdots\partial_{j-\ell}\left(s_{j-\ell+1}\cdots s_{j-1}s_{i,j}(g)\cdot\partial_{j-\ell+1}\cdots\partial_{j-1}\partial_{i,j}(h)\right)=0.
\end{equation}
The $\ell=j-i$ case of \eqref{eqn-oddop-subclaim},
\begin{equation*}
s_{j-i+1}\cdots s_{j-1}s_{i,j}(g)\cdot\partial_{i+1}\cdots\partial_{j-1}\partial_{i,j}(h)=0,
\end{equation*}
follows from \eqref{eqn-lem-oddop-2} in lower degree.  Before proceeding, we fix $i$ and induct on $j$, the base case $j=i+1$ being obvious.

\textit{Step 3:} To prove \eqref{eqn-oddop-subclaim} by decreasing induction on $\ell$ (keeping $i,j$ both fixed), we compute
\begin{equation*}\begin{split}
&\partial_{i+1}\cdots\partial_{j-\ell}\left(s_{j-\ell+1}\cdots s_{j-1}s_{i,j}(g)\cdot\partial_{j-\ell+1}\cdots\partial_{j-1}\partial_{i,j}(h)\right)\\
&\qquad=\partial_{i+1}\cdots\partial_{j-\ell-1}(\pm s_{j-\ell+1}\cdots s_{j-1}s_{i,j}\partial_{i,j-\ell}(g)\cdot\partial_{j-\ell+1}\cdots\partial_{j-1}\partial_{i,j}(h)\\
&\qquad\qquad+s_{j-\ell}\cdots s_{j-1}s_{i,j}(g)\cdot\partial_{j-\ell}\cdots\partial_{j-1}\partial_{i,j}(h)),
\end{split}\end{equation*}
by the Leibniz rule and part 1 of Lemma \ref{lem-oddop-1}.  Any one of the operators $\partial_{i+1},\ldots\partial_{j-\ell-1}$ annihilates the second factor of the first term on the right-hand side by part 2 of Lemma \ref{lem-oddop-1}, so this equals
\begin{equation*}\begin{split}
&\ldots=\pm s_{j-\ell+1}\cdots s_{j-1}s_{i,j}\partial_{i+1}\cdots\partial_{j-\ell-1}\partial_{i,j-\ell}(g)\cdot\partial_{j-\ell+1}\cdots\partial_{j-1}\partial_{i,j}(h)\\
&\qquad\qquad+\partial_{i+1}\cdots\partial_{j-\ell-1}\left(s_{j-\ell}\cdots s_{j-1}s_{i,j}(g)\cdot\partial_{j-\ell}\cdots\partial_{j-1}\partial_{i,j}(h)\right).
\end{split}\end{equation*}
This vanishes by induction: the first factor of the first term by \eqref{eqn-lem-oddop-2} for lower $j$ and the second term by \eqref{eqn-oddop-subclaim} for higher $\ell$.  This proves \eqref{eqn-oddop-subclaim} in this degree and for this $\ell$, hence for all $j,\ell$.

\textit{Step 4:} It remains to prove \eqref{eqn-lem-oddop-2}.  We have
\begin{equation*}
\partial_{i+1}\cdots\partial_{j-1}\partial_{i,j}(gh)=\partial_{i+1}\cdots\partial_{j-1}\left(\partial_{i,j}(g)\cdot h\right)+\partial_{i+1}\cdots\partial_{j-1}\left(s_{i,j}(g)\cdot\partial_{i,j}(h)\right).
\end{equation*}
The second term on the right-hand side is zero by \eqref{eqn-oddop-subclaim} at $\ell=1$.  Applying each of $\partial_{i+1},\ldots\partial_{j-1}$ to the first term on the right-hand side and using the Leibniz rule, we always get zero for the term in which the divided difference operator hits $h$.  Therefore this term equals $\partial_{i+1}\cdots\partial_{j-1}\partial_{i,j}(g)\cdot h$, which is zero by \eqref{eqn-lem-oddop-2} in lower degree.  This completes the proof.
\end{proof}

\begin{lem}\label{lem-cox-first} For any $1\leq i\leq a-1$, there is a reduced expression for $w_0$ in $S_a$ with $s_i$ acting first.\end{lem}
\begin{proof} This is clear for $a=2,3$.  Let $w_0^{(a)}$ denote $w_0$ in $S_a$, $w_0^{(a-1)}$ denote the longest word among strands $1,\ldots,a-1$ in $S_a$, and $w_0^{(a-1)'}$ denote the longest word among strands $2,\ldots,a$ in $S_a$.  Since
\begin{equation*}
w_0^{(a)}=s_1\cdots s_{a-1}w_0^{(a-1)}=s_{a-1}\cdots s_1w_0^{(a-1)'},
\end{equation*}
the lemma follows by induction.\end{proof}

\begin{proof}[Proof of the OWL, Lemma \ref{lem-owl}] We are free to reorder $D_a$ up to sign; we choose the ordering $D_a'$ such that $D_2'=\partial_1$ and $D_a'=\partial_{a-1}\cdots\partial_1D_{a-1}'$.  (In fact, $D_a'=(-1)^{\binom{a}{3}}D_a$.)  Explicitly,
\begin{equation*}
D_a=(\partial_{a-1}\cdots\partial_1)\cdots(\partial_{a-1}\partial_{a-2})\partial_{a-1}.
\end{equation*}
Fix some $\xi$ and consider the generalized action $w^\xi\cdot f$.  Let $s_{w_0,b}$ denote the longest element of the symmetric group among strands $a-b+1,\ldots,a$.  If the rightmost $\partial_{a-1}$ in $D_a'$ acts with $\xi=1$, then any $f$ is annihilated and we are in Case 1.  By induction on $b\geq1$, then, suppose the bottom $\binom{b}{2}$ crossings all act with $\xi=0$.  The next $b$ crossings are between strand pairs $(a-b,a-b+1),(a-b+1,a-b+2),\ldots,(a-1,a)$.  Let $i$ be the number of the left strand of any of these crossings which acts with $\xi=1$.  If $i<a-1$ and the $(i+1,i+2)$ crossing acts with $\xi=0$, then the omission word $w^\xi_\text{om}$ is non-reduced by Lemma \ref{lem-cox-first} (pull a crossing $s_{i+1}$ to the top of $s_{w_0,b}$), so we are in Case 3 and are done.  We are therefore reduced to the case in which once one of these crossings acts with $\xi=1$, then so do all others to the left.  That is,
\begin{equation*}
w^\xi\cdot f=(\ldots)\partial_{a-1}\cdots\partial_{a-b+\ell}s_{a-b+\ell-1}\cdots s_{a-b}s_{w_0,b}(f).
\end{equation*}
By part 1 of Lemma \ref{lem-oddop-1} this equals
\begin{equation*}
\pm(\ldots)s_{a-b+\ell-1}\cdots s_{a-b}s_{w_0,b}\partial_{a-b+1}\cdots\partial_{a-b+\ell-1}\partial_{a-b,a-b+\ell}(f),
\end{equation*}
which vanishes by Lemma \ref{lem-oddop-2}.
\end{proof}

An immediate corollary of the OWL is that for any $f\in\osym_a$ and $g\in\opol_a$,
\begin{equation}\label{eqn-Da-left-linearity}
D_a(fg)=f^{w_0}D_a(g),
\end{equation}
since $\partial_v=0$ for any nonreduced word $v$.

\begin{cor} The odd symmetrization operator $\Ss$ is left $\osym_a$-linear and is a projection operator onto the subring $\osym_a\subset\opol_a$.\end{cor}
\begin{proof} Left $\osym_a$-linearity of $\Ss$ follows immediately from equation \eqref{eqn-Da-left-linearity}.  The image of $\Ss$ is inside $\osym_a$ due to the presence of the $D_a$ in its definition and by Lemma \ref{lemma-w0-e}.  This image is all of $\osym_a$ and $\Ss^2=\Ss$ because, by the OWL and the calculation $D_a(\undx^{\delta_a})=(-1)^{\binom{a}{3}}$ (see \eqref{eq_Da-undxdelta}),
\begin{equation}
\Ss(f\undx^{\delta_a})=(-1)^{\binom{a}{3}}D_a(f\undx^{\delta_a})^{w_0}
=(-1)^{\binom{a}{3}}\left(f^{w_0}D_a(\undx^{\delta_a})\right)^{w_0}=f
\end{equation}
for any $f\in\osym_a$.
\end{proof}

We conclude this section with another useful corollary of the OWL.
\begin{cor} For any $f\in\opol_a$,
\begin{equation}
D_a(f)^{w_0}=(-1)^{\binom{a}{2}}D_a(f^{w_0}).
\end{equation}\end{cor}
\begin{proof} Expand $f$ in the odd Schubert polynomial basis,
\begin{equation*}
f=\sum_{w\in S_a}f_w\sch_w,
\end{equation*}
where each $f_w\in\osym_a$.  Since the action of $w_0$ preserves $\osym_a\subset\opol_a$, equation \eqref{eqn-Da-left-linearity} implies
\begin{equation*}\begin{split}
&D_a(f)^{w_0}=\sum_wf_wD_a(\sch_w)^{w_0}=f_{w_0}D_a(\undx^{\delta_a})^{w_0},\\
&D_a(f^{w_0})=\sum_wD_a((f_w)^{w_0}(\sch_w)^{w_0})=f_{w_0}D_a((\undx^{\delta_a})^{w_0}).
\end{split}\end{equation*}
The last equality on each line is by degree reasons: the only Schubert polynomial of high enough degree not to be annihilated by $D_a$ is $\sch_{w_0}$.  The corollary then follows from the relations
\begin{equation*}
D_a(\undx^{\delta_a})=(-1)^{\binom{a}{3}},\quad D_a((\undx^{\delta_a})^{w_0})=(-1)^{\binom{a+1}{3}},
\end{equation*}
proved in Section~\ref{sec_switch}.
\end{proof}

%
\subsection{Odd Schur polynomials} \label{sec_oschur}
%

%
\subsubsection{Partitions}
%

By a partition $\alpha=(\alpha_1,\alpha_2,\ldots)$ we mean
a nonincreasing sequence of nonnegative integers
$\alpha_1\ge\alpha_2\ge\ldots\ge0$ with finite sum. We define $|\alpha|=\sum_{i=1}^a
{\alpha_i}$ when $\alpha_i=0$ for $i>a$.  For short, we represent repeated entries with exponents: for example, $(\ldots,2^3,\ldots)$ means $(\ldots,2,2,2,\ldots)$.

By $P(a,b)$ we denote the set of all partitions $\alpha$ with at
most $a$ parts (that is, with $\alpha_{a+1}=0$) such that $\alpha_1\le
b$. That is, $P(a,b)$ consists of partitions whose corresponding Young diagram fits into a box of
size $a \times b$. Moreover, the set of all partitions with at most
$a$ parts (that is, the set $P(a,\infty)$) we denote simply by $P(a)$.  $P(0)=\{ \emptyset\}$ is the set of all partitions with at most $0$ parts, so that $P(0)$ contains only the empty partition.

The cardinality of $P(a,b)$ is ${a+b \choose a}$. We call
\begin{equation} \label{eq_Pab_card}
  |P(a,b)|_q:=\sum_{\alpha\in P(a,b)} {q^{2|\alpha|-ab}}=\left[{a+b
\atop a}\right]
\end{equation}
the $q$-cardinality of $P(a,b)$. The last term in
the above equations is the balanced $q$-binomial.

The \textit{dual} (or \textit{conjugate}) partition of  $\alpha$, denoted
$\overline{\alpha}=(\overline{\alpha}_1,\overline{\alpha}_2,\ldots)$ with
$\overline{\alpha}_j=\#\{i|\alpha_i\ge j\}$, has Young diagram given by reflecting
the Young diagram of $\alpha$ along the diagonal. For a partition
$\alpha\in P(a,b)$ we define the {\it complementary partition}
$\alpha^c=(b-\alpha_a,\ldots,b-\alpha_2,b-\alpha_1)$. To be more explicit we let $K=(b^a)\in P(a,b)$ and write $K-\alpha$ in place of $\alpha^c$ to emphasize that $\alpha \in P(a,b)$. Finally we
define $\hat{\alpha}:=\overline{\alpha^c}$. Note that $\overline{\alpha}$ and
$\hat{\alpha}$ belong to $P(b,a)$ and $\alpha^c $ to $P(a,b)$.

%
\subsubsection{Odd Schur polynomials}
%

In the even case, one definition of the Schur polynomial corresponding to a partition $\alpha$ of length at most $a$ is
\begin{equation}
s_\alpha(x_1,\ldots,x_a)=\partial_{w_0}(x^{\delta_a+\alpha})=\partial_{w_0}(x_1^{a-1+\alpha_1}
x_2^{a-2+\alpha_2} \dots x_a^{\alpha_a}).
\end{equation}
In the odd case, we must be careful about the ordering of the terms in the above expression.

\begin{defn} The \textit{odd Schur polynomial} corresponding to a partition $\alpha$ of length at most $a$ is the element of $\opol_a$ given by odd-symmetrizing the monomial $\undx^\alpha$,
\begin{equation}\label{eqn-defn-schur}
s_\alpha(x_1,\ldots,x_a)=\Ss(\undx^\alpha)=(-1)^{\binom{a}{3}}D_a(\undx^\alpha \undx^{\delta_a})^{w_0},
\end{equation}
see also equation \eqref{eqn-odd-symmetrization}.  As usual, $\undx^{\alpha}=x_1^{\alpha_1}\cdots x_a^{\alpha_a}$.\end{defn}

Since $\Ss$ is the projection onto the subring of odd symmetric polynomials, $s_\alpha$ is odd symmetric.

%
\subsubsection{The odd Pieri rule}
%

In what follows, we will refer to the \textit{normal ordering} of a monomial $x_{i_1}x_{i_2}\cdots x_{i_r}$; this is the monomial obtained by sorting the subscripts into non-decreasing order (and not introducing a sign).

\vspace{0.07in}

By analogy with the even case, we have the following.
\begin{lem}\label{lem-schur-elementary} For $1\leq k\leq a$, we have
\begin{equation} \label{eq_schur_elem}
s_{(1^k)}(x_1,\ldots,x_a)=(-1)^{\frac{1}{2}k(k-1)}\varepsilon_k(x_1,\ldots,x_a).
\end{equation}\end{lem}
\begin{proof} Let $i_1<\cdots<i_k$ be a subset of $\lbrace1,\ldots,a\rbrace$.  Unless $i_1=1,\ldots,i_k=k$, the normal ordering of the monomial $\xt_{i_1}\cdots\xt_{i_k}\undx^{\delta_a}$ will have a factor of the form $x_j^\ell x_{j+1}^\ell$ somewhere, so that $D_a(\xt_{i_1}\cdots\xt_{i_k}\undx^{\delta_a})=\partial_j(\xt_{i_1}\cdots\xt_{i_k}\undx^{\delta_a})=0$.  Hence, using odd-symmetrization and the definitions of $\varepsilon_k$ and of $s_{(1^k)}$,
\begin{equation*}
\varepsilon_k=\Ss(\varepsilon_k)=\Ss(\xt_{i_1}\cdots\xt_{i_k})=(-1)^{\frac{1}{2}k(k-1)}s_{(1^k)}.
\end{equation*}
\end{proof}

For a partition $\alpha$, let $\frac{\alpha}{m}$ denote the partition whose Young diagram is obtained from that of $\alpha$ by removing the $m$-th and all lower rows; let $\frac{m}{\alpha}$ be likewise, removing the first through $m$-th rows.
\begin{prop}[Odd Pieri rule]\label{prop_odd_pieri} Let $\alpha$ be a partition and let $1\leq k\leq a$.  Then in $\osym_a$,
\begin{equation}\label{eqn-e-right-pieri}
s_\alpha s_{(1^k)}=\sum_\mu(-1)^{\left|\frac{i_1}{\alpha}\right|+\ldots+\left|\frac{i_k}{\alpha}\right|}s_\mu,
\end{equation}
where the sum is over all partitions $\mu$ with Young diagram obtained from that of $\alpha$ by adding one box each to the rows $i_1<\cdots<i_k$.
\end{prop}
Note that mod 2, this is the usual Pieri rule.  The diagrams $\mu$ described in the statement of the Proposition are precisely the diagrams obtained from $\alpha$ by what is called ``adding a vertical strip of length $k$''; but we also need to know which rows we are adding to in order to get the correct sign.
\begin{proof} Identifying $s_{(1^k)}$ with $(-1)^{\binom{k}{2}}\varepsilon_k$ by Lemma \ref{lem-schur-elementary}, we compute:
\begin{equation*}\begin{split}
s_\alpha s_{(1^k)}&\refequal{\eqref{eqn-defn-schur}}(-1)^{\binom{a}{3}}D_a(\undx^\alpha\undx^{\delta_a})^{w_0}\cdot(-1)^{\binom{k}{2}}\varepsilon_k\\
&\refequal{\eqref{eqn-w0-e},\eqref{eqn-leibniz}}(-1)^{\binom{a}{3}+k\binom{a-1}{2}}D_a(\undx^\alpha\undx^{\delta_a} \varepsilon_k)^{w_0}\\
&\refequal{\eqref{eqn-defn-e}}(-1)^{\binom{a}{3}+k\binom{a-1}{2}}\sum_{i_1<\cdots<i_k}D_a(\undx^\alpha\undx^{\delta_a} \widetilde{x}_{i_1}\cdots\widetilde{x}_{i_k})^{w_0}.
\end{split}\end{equation*}
Commuting the factor $\widetilde{x}_{i_1}\cdots\widetilde{x}_{i_k}$ past $\undx^{\delta_a}$ and normal ordering it with $\undx^\alpha$ cancels the factor $(-1)^{k\binom{a-1}{2}}$ and introduces a factor of $(-1)^{\left|\frac{i_1}{\alpha}\right|+\ldots+\left|\frac{i_k}{\alpha}\right|}$ to each term.  So we appear to have the desired equation \eqref{eqn-e-right-pieri}, except that there are certain terms where $\mu$ is not non-decreasing, that is, does not correspond to a Young diagram (the term involving $x_{i_1}\cdots x_{i_k}$ corresponds to adding one box to $\alpha$ in each of rows $i_1,\ldots,i_k$).  This occurs when $\mu$ is a composition but not a partition.  A term will fail to correspond to a Young diagram if and only if the resulting monomial $\undx^\alpha\undx^{\delta_a} x_{i_1}\cdots x_{i_k}$, when normal ordered, has two adjacent exponents equal:
\begin{equation*}\begin{split}
\undx^{\alpha}\undx^{\delta_a} x_{i_1}\cdots x_{i_k}=\pm\ldots x_j^m x_{j+1}^m\ldots\text{ for some }j,m\text{ if and}\\
\text{only if }\mu\text{ is a composition but not a partition.}
\end{split}\end{equation*}
Such a monomial is sent to zero by $\partial_j$.  Re-ordering $D_a=\partial_{w_0}$ so as to have $\partial_j$ act first, we see the term does not contribute to equation \eqref{eqn-e-right-pieri}.
\end{proof}
The above proves the usual (even) Pieri rule as well; just drop all the signs.

\begin{rem} Assume for a moment that the Schur functions of this paper agree with the Schur functions of \cite{EK} (see Conjecture \ref{conj-schur} below).  Combining Proposition \ref{prop_odd_pieri} with the automorphism $\psi_1\psi_2$ and the anti-involution $\psi_3$ of $\osym$ which were introduced in Section 2.3 of \cite{EK}, we have three more forms of the Pieri rule:
\begin{eqnarray}
&(-1)^{\ell(w_\alpha)}s_\alpha s_{(k)}=\sum_\mu(-1)^{\ell(w_\mu)}(-1)^{\left|\frac{i_1}{\overline{\alpha}}\right|+\ldots+\left|\frac{i_k}{\overline{\alpha}}\right|}s_\mu,\\
&\eta_\alpha s_{(1^k)}s_\alpha=\sum_\mu(-1)^{\left|\frac{i_1}{\alpha}\right|+\ldots+\left|\frac{i_k}{\alpha}\right|}\eta_\mu s_\mu,\\
&\eta_\alpha(-1)^{\ell(w_\alpha)}s_{(k)}s_\alpha=\sum_\mu(-1)^{\ell(w_\mu)}(-1)^{\left|\frac{i_1}{\overline{\alpha}}\right|+\ldots+\left|\frac{i_k}{\overline{\alpha}}\right|}\eta_\mu s_\mu.
\end{eqnarray}
For definitions of the signs $(-1)^{\ell(w_\alpha)}$ and $\eta_\alpha$, see \cite{EK}.  The second sum (respectively first and third sums) is over all $\mu$ obtained from $\alpha$ by adding a vertical (respectively horizontal) strip; the integers $i_1,\ldots,i_k$ are the rows (respectively columns) to which a box was added.\end{rem}

\begin{rem} Again assuming Conjecture \ref{conj-schur}, no na\"{i}ve analogue of the Jacobi-Trudi relation can hold.  The usual (even) Jacobi-Trudi relation is
\begin{equation}
s_\alpha=\det(\varepsilon_{\overline{\alpha}_i+j-i})=\det(h_{\alpha_i+j-i}).
\end{equation}
A minimal example of the problem is $\alpha=(2,2)$ (we order the terms of our determinants top-to-bottom, though this isn't essential):
\begin{equation}\begin{split}
\det(\varepsilon_{\overline{\alpha}_i+j-i})&=\varepsilon_{22}-\varepsilon_{31}\\
\det(h_{\alpha_i+j-i})&=h_{22}-h_{31}=\left(\varepsilon_{22}-2\varepsilon_{211}+\varepsilon_{1111}\right)-\left(\varepsilon_{31}-\varepsilon_{111}\right)\\
s_{22}&=\varepsilon_{22}+\varepsilon_{31}-2\varepsilon_4.
\end{split}\end{equation}
Since $2\varepsilon_3=\varepsilon_{12}+\varepsilon_{21}$, it seems some clever rearrangement of terms and of signs could make the two determinants equal each other.  But neither determinant along with any possible re-arrangement of signs can lead to a term involving $\varepsilon_4$.  That is, $\varepsilon_4$ is not in the subring generated by $h_1,h_2,h_3,\varepsilon_1,\varepsilon_2,\varepsilon_3$.\end{rem}

%
\section{Graphical calculus}\label{sec-graphical-calculus}
%

%
\subsection{Graphical calculus for the odd nilHecke algebra}
%

We find it convenient to use a graphical calculus to represent
elements in $\ONH_a$. The diagrammatic representation of elements in $\ONH_a$ is given by braid-like dotted diagrams $D$ equipped with the height Morse function $h\maps D \to \R$, such that $h(g_1) \neq h(g_2)$ for any generators $g_1, g_2$ that appear in the diagram.

We write
\begin{eqnarray}
   \xy  (0,0)*{\includegraphics[scale=0.4]{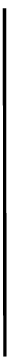}};  \endxy
   \;\;  \dots \;\;
   \xy  (0,0)*{\includegraphics[scale=0.4]{1thin.eps}};  \endxy
   \;\;  \dots \;\;
   \xy  (0,0)*{\includegraphics[scale=0.4]{1thin.eps}};  \endxy \quad := \quad
 1 \in \ONH_a
\end{eqnarray}
with a total of $a$ strands. The polynomial generators can be
written as
\begin{eqnarray}
  \xy  (0,0)*{\includegraphics[scale=0.4]{1thin.eps}};  \endxy
   \;\;  \dots \;\;
   \xy  (0,0)*{\includegraphics[scale=0.4]{1thin.eps}}; (0,0)*{\bullet}; \endxy
   \;\;  \dots \;\;
   \xy  (0,0)*{\includegraphics[scale=0.4]{1thin.eps}};  \endxy\quad := \quad
   x_r
\end{eqnarray}
with the dot positioned on the $r$-th strand counting from the left,
and
\begin{eqnarray}
  \xy  (0,0)*{\includegraphics[scale=0.4]{1thin.eps}};  \endxy
    \;\;  \dots \;\;
   \xy  (0,0)*{\includegraphics[scale=0.4]{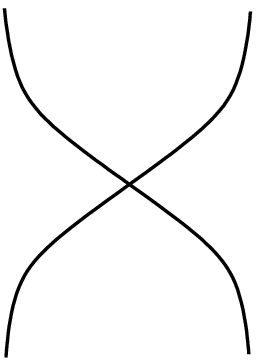}};  \endxy
   \;\;  \dots \;\;
   \xy  (0,0)*{\includegraphics[scale=0.4]{1thin.eps}};  \endxy\quad := \quad
  \partial_r,
\end{eqnarray}
with the crossing interchanging the $r$th and $(r+1)$-st strands.

In the diagrammatic notation multiplication is given by stacking
diagrams on top of each other, left-to-right becoming top-to-bottom.  Relations in the odd nilHecke ring acquire a graphical interpretation. For example, the equalities $\partial_rx_r+x_{r+1}\partial_r = 1 =
x_r \partial_r + \partial_r x_{r+1}$ become diagrammatic
identities:
\begin{eqnarray}        \label{new_eq_iislide}
  \xy  (0,0)*{\includegraphics[scale=0.4]{2thin-cross.eps}};
  (-4.3,-4)*{\bullet};
  \endxy
    \quad + \quad
   \xy  (0,0)*{\includegraphics[scale=0.4]{2thin-cross.eps}};
   (4.2,3)*{\bullet};
  \endxy
  & = &
   \xy  (-3,0)*{\includegraphics[scale=0.4]{1thin.eps}};
        (3,0)*{\includegraphics[scale=0.4]{1thin.eps}};
  \endxy
  \quad = \quad
  \xy  (0,0)*{\includegraphics[scale=0.4]{2thin-cross.eps}};
  (-4.2,3)*{\bullet};
  \endxy
    \quad + \quad
   \xy  (0,0)*{\includegraphics[scale=0.4]{2thin-cross.eps}};
   (4.2,-4)*{\bullet};
  \endxy
\end{eqnarray}
and the relation $\partial_r \partial_r= 0$ becomes
\begin{eqnarray}
\xy (0,0)*{\includegraphics[scale=0.4]{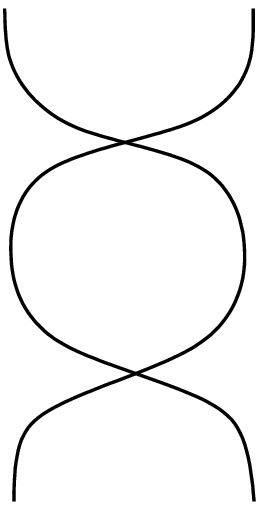}}; \endxy
  &=&
0.
\end{eqnarray}
The relation $\partial_r\partial_{r+1}\partial_r =
\partial_{r+1}\partial_r \partial_{r+1}$ is depicted as
\begin{eqnarray}      \label{new_eq_r3_easy}
\xy (0,0)*{\includegraphics[scale=0.4]{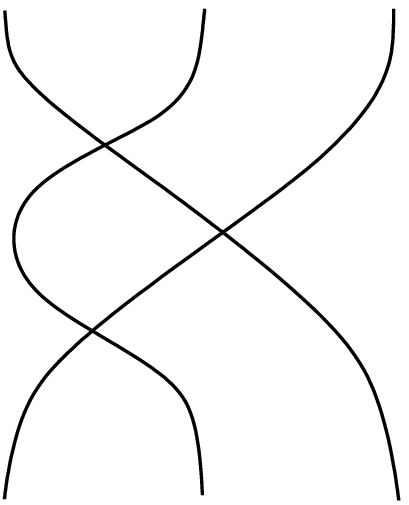}}; \endxy
  &=&
 \xy (0,0)*{\reflectbox{\includegraphics[scale=0.4]{3thin-R3.eps}}}; \endxy.
\end{eqnarray}
The remaining relations in the nilHecke ring can be encoded by
the requirement that the diagrams super commute under braid-like isotopies:

\begin{align}
\dots\;\;
   \xy  (0,0)*{\includegraphics[scale=0.4]{1thin.eps}}; (0,-3)*{\bullet}; \endxy
\;\; \dots \;\;
    \xy  (0,0)*{\includegraphics[scale=0.4]{1thin.eps}}; (0,3)*{\bullet}; \endxy \;\;\dots
&\quad + \quad
\dots\;\;
    \xy  (0,0)*{\includegraphics[scale=0.4]{1thin.eps}}; (0,3)*{\bullet}; \endxy
\;\; \dots \;\;
    \xy  (0,0)*{\includegraphics[scale=0.4]{1thin.eps}}; (0,-3)*{\bullet}; \endxy \;\;\dots \quad = 0,\\ \nn \\
\dots\;\;
    \xy  (0,0)*{\includegraphics[scale=0.4,angle=180]{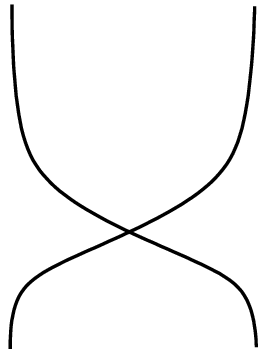}}; \endxy
\;\; \dots \;\;
    \xy  (0,0)*{\includegraphics[scale=0.4]{1thin.eps}}; (0,3)*{\bullet}; \endxy \;\;\dots
&\quad + \quad
\dots\;\;
    \xy  (0,0)*{\includegraphics[scale=0.4]{2thin-crossp.eps}}; \endxy
\;\; \dots \;\;
    \xy  (0,0)*{\includegraphics[scale=0.4]{1thin.eps}}; (0,-3)*{\bullet}; \endxy\;\;\dots  \quad = 0,\\ \nn \\
\dots\;\;
    \xy  (0,0)*{\includegraphics[scale=0.4]{1thin.eps}}; (0,-3)*{\bullet}; \endxy
\;\; \dots \;\;
    \xy  (0,0)*{\includegraphics[scale=0.4]{2thin-crossp.eps}}; \endxy \;\;\dots
&\quad + \quad
    \dots\;\;
     \xy  (0,0)*{\includegraphics[scale=0.4]{1thin.eps}}; (0,3)*{\bullet}; \endxy
\;\; \dots \;\;
    \xy  (0,0)*{\includegraphics[scale=0.4,angle=180]{2thin-crossp.eps}}; \endxy
    \;\;\dots  \quad = 0,\\ \nn \\
    \dots\;\;
    \xy  (0,0)*{\includegraphics[scale=0.4,angle=180]{2thin-crossp.eps}}; \endxy
\;\; \dots \;\;
     \xy  (0,0)*{\includegraphics[scale=0.4]{2thin-crossp.eps}}; \endxy  \;\;\dots
&\quad + \quad
\dots\;\;
    \xy  (0,0)*{\includegraphics[scale=0.4]{2thin-crossp.eps}}; \endxy
\;\; \dots \;\;
    \xy  (0,0)*{\includegraphics[scale=0.4,angle=180]{2thin-crossp.eps}}; \endxy \;\;\dots  \quad = 0.
\end{align}

%
\subsection{Box notation for odd thin calculus}\label{subsec-box-notation}
%
To simplify diagrams, write
\[ 
  \xy
 (0,0)*{\includegraphics[scale=0.5]{1thin.eps}};
 (-3,-4)*{a};
  \endxy
  \quad : = \quad
  \xy
 (0,0)*{\includegraphics[scale=0.5]{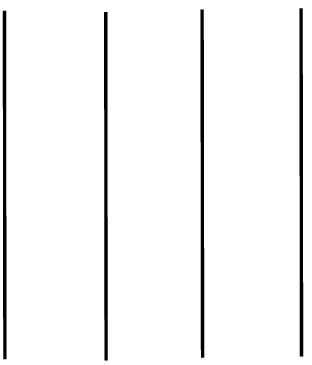}};
 (0,-11)*{\underbrace{\hspace{0.7in}}};  (0,-14)*{a};
  \endxy
 \quad = \quad 1 \in \ONH_a,
\]
where we will omit the label $a$ if it appears in a coupon as below. The operator $D_a$ is represented as
\begin{equation}\label{eq_def_Da}
    D_a \quad = \quad
    \xy
 (0,0)*{\includegraphics[scale=0.5]{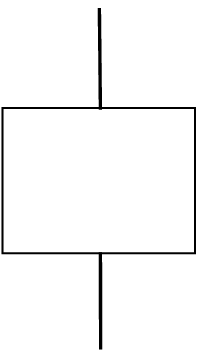}};
 (0,0)*{D_a};
  \endxy
  \quad = \quad
  \xy
 (0,0)*{\includegraphics[scale=0.5]{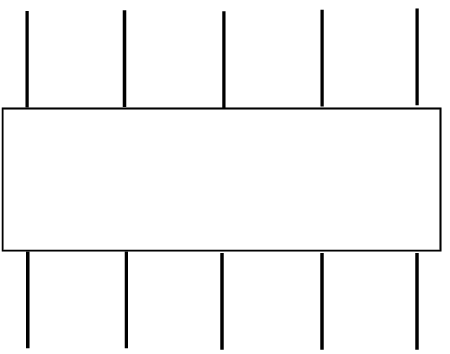}};
 (0,0)*{D_a};
  \endxy
 \quad := \quad
 \xy
 (0,0)*{\includegraphics[scale=0.5]{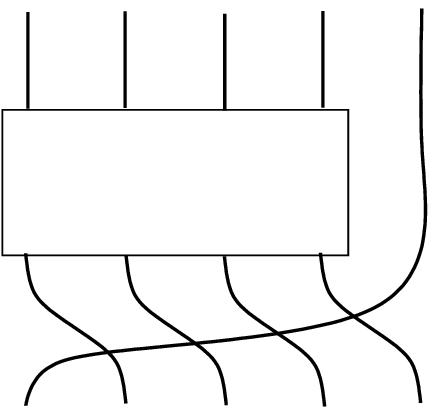}};
 (-2,1)*{D_{a-1}};
  \endxy \quad = \quad
 D_{a-1} (\partial_{a-1}\dots \partial_2\partial_1) .
\end{equation}
Next, let
\[ 
 \undx^{\delta_a}  \;\; = \;\;
     \xy
 (0,0)*{\includegraphics[scale=0.5]{1box.eps}};
 (0,0)*{\delta_a};
  \endxy
 \;\;:= \;\;
 \xy
 (-2.6,0)*{\includegraphics[scale=0.5]{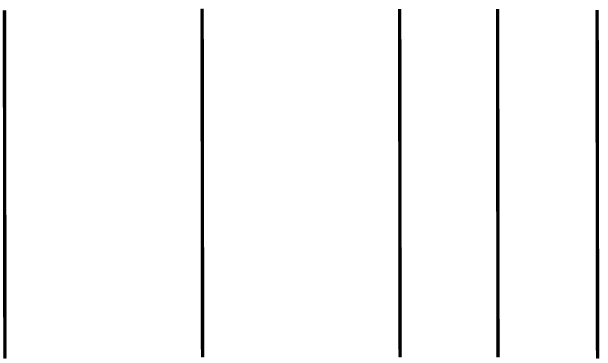}};
 (-17.7,4)*{\bullet}+(-4,1)*{\scs a-1};
 (-7.7,1)*{\bullet}+(-4,1)*{\scs a-2};
 (-3,-3)*{\cdots};
 (2.3,-2)*{\bullet}+(-2,1)*{\scs 2};
 (7.3,-5)*{\bullet};
  \endxy \quad = \quad x_1^{a-1}x_2^{a-2} \dots x_{a-2}^{2} x_{a-1}.
\]
Recall that the super degree of an element of $\ONH_a$ is the $\Z$-degree divided by 2.  In particular, we have
\begin{equation}
\deg_s(D_a) \equiv \deg_s(\undx^{\delta_a}) \equiv \binom{a}{2} \quad (\text{mod 2}).
\end{equation}

Throughout this section we make frequent use of the trivial binomial identities
\begin{align}
  \binom{a}{2}   &= \sum_{j =1}^{a-1} j = \frac{a(a-1)}{2} , \\
  \binom{a}{3}   &= \sum_{j=1}^{a-1} \binom{j}{2}=\sum_{j=1}^{a-2} \sum_{\ell=1}^{j}j = \frac{a(a-1)(a-2)}{6}, \\
  \binom{a+3}{4} &= \sum_{k=1}^a\sum_{j=1}^k \sum_{\ell=1}^j \ell = \frac{a (a+1) (a+2) (a+3)}{24}.
\end{align}

%
\subsubsection{0-Hecke crossings}
%

We will use a diagrammatic shorthand
\begin{equation}\label{eqn-0Hecke}
 \overline{\partial}_r \quad := \quad  \xy  (0,0)*{\includegraphics[scale=0.4]{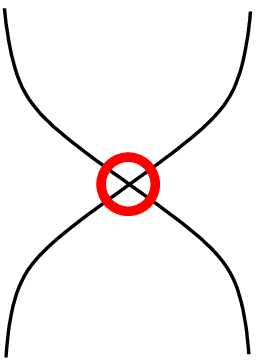}};
  \endxy
 \quad := \quad
    \xy  (0,0)*{\includegraphics[scale=0.4]{2thin-cross.eps}};
  (-4.2,3)*{\bullet};
  \endxy \quad = \quad  x_r\partial_r.
\end{equation}
It is easy to check that the elements $\overline{\partial}_r$ for $1 \leq r \leq a-1$ generate a copy of the 0-Hecke algebra:
\begin{equation}
  \overline{\partial}_r^2  \quad = \quad
  \xy (0,0)*{\includegraphics[scale=0.4]{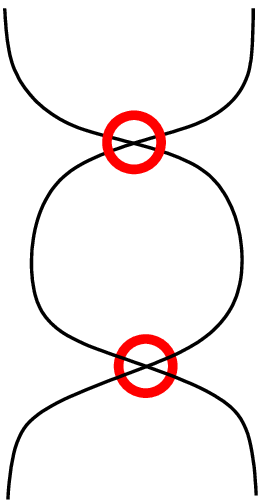}}; \endxy
  \quad = \quad
  \xy  (0,0)*{\includegraphics[scale=0.4]{2thin-Ocross.eps}};
  \endxy
 \quad  = \quad \overline{\partial}_r,
\end{equation}
\begin{equation}
  \overline{\partial}_{r}\overline{\partial}_{r+1}\overline{\partial}_{r}
  \quad = \quad
\xy (0,0)*{\includegraphics[scale=0.4]{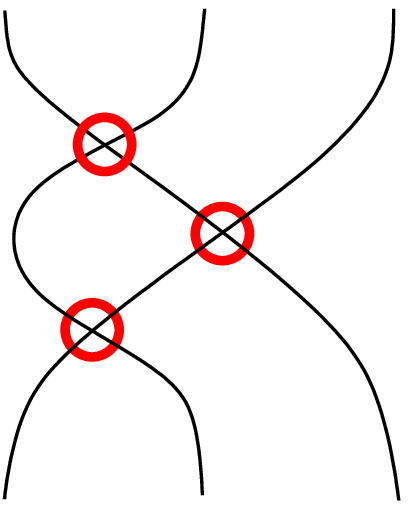}}; \endxy
  \quad = \quad
  \xy (0,0)*{\reflectbox{\includegraphics[scale=0.4]{3thin-OR3.eps}}}; \endxy
  \quad = \quad
  \overline{\partial}_{r+1}\overline{\partial}_{r}\overline{\partial}_{r+1},
\end{equation}
and, since the element $\overline{\partial}_r$ has degree zero, we have
\[
 \overline{\partial}_r \overline{\partial}_s = \overline{\partial}_s \overline{\partial}_r \qquad \text{if $|r-s|>1$.}
\]

Let $w_0$ denote a reduced word presentation of the longest element in $S_a$ and define
\[
 e_a \quad = \quad
    \xy
 (0,0)*{\includegraphics[scale=0.5]{1box.eps}};
 (0,0)*{e_a};
  \endxy
  \quad = \quad
  \xy
 (0,0)*{\includegraphics[scale=0.5]{5box.eps}};
 (0,0)*{e_a};
  \endxy
 \quad = \quad
 \overline{\partial}_{w_0}.
\]
By the 0-Hecke relations above it is clear that $\overline{\partial}_{w_0}$ does not depend on the choice of reduced word presentation $w_0$.

%
\subsubsection{Relations for $D_a$}
%

\begin{lem}[Crossing slide lemma] \label{lem_crossing_slide}
\begin{eqnarray} \label{eq_cross_exchange}
(\partial_{a-2} \dots \partial_{2}\partial_1)(\partial_{a-1} \dots\partial_2\partial_1)&\quad = \quad&(\partial_{a-1}\partial_{a-2} \dots \partial_{2}\partial_1)(\partial_{a-1} \dots\partial_2)
\\
    \xy (0,0)*{\includegraphics[scale=0.5,]{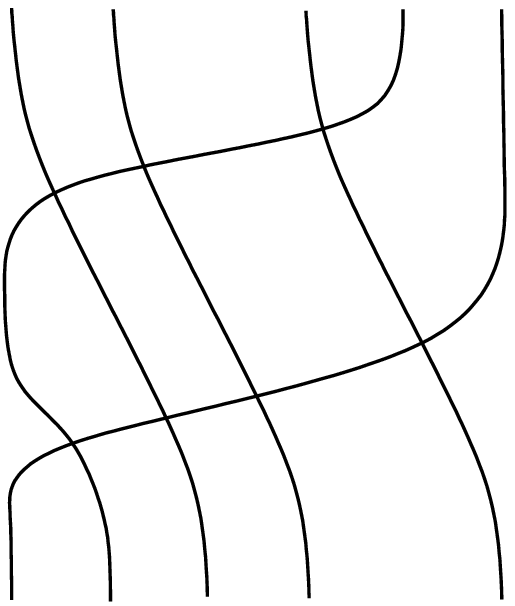}};
  (-2,12)*{\cdots};(7,-12)*{\cdots};\endxy
  &\quad = \quad&
    \xy (0,0)*{\includegraphics[scale=0.5,]{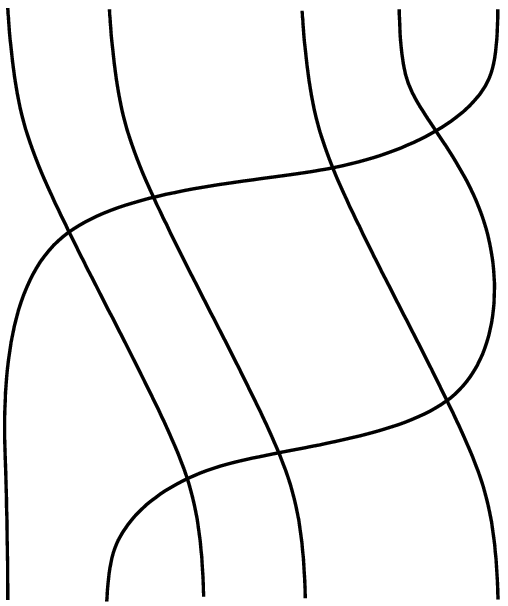}};
  (-2,12)*{\cdots};(7,-12)*{\cdots};\endxy \nn
\end{eqnarray}
\end{lem}

\begin{proof}
The proof is by induction on the number of strands.  The base case of $a=3$ strands follows from the odd nilHecke relation~\eqref{new_eq_r3_easy}.  Assume the result holds for $a-1$ strands.  Then
\begin{eqnarray}
    \xy (0,0)*{\includegraphics[scale=0.5,]{5bl-cross.eps}};
  (-2,12)*{\cdots};(7,-12)*{\cdots};\endxy
  \;\; =\;\; (-1)^{a-3}\;
     \xy (0,0)*{\includegraphics[scale=0.5,]{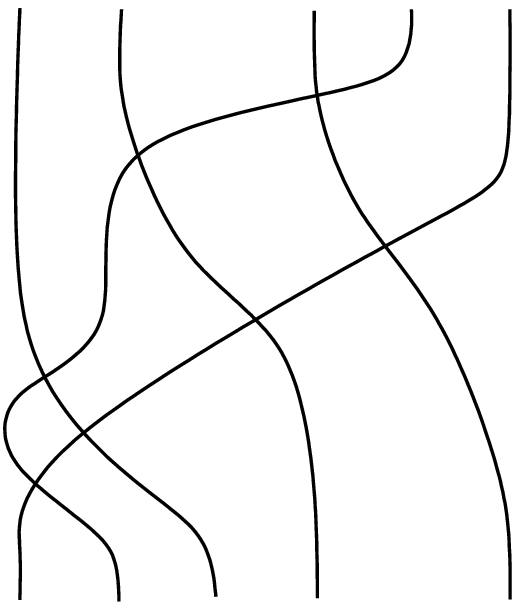}};
  (-2,12)*{\cdots};(7,-12)*{\cdots};\endxy
  \;\;\refequal{\eqref{new_eq_r3_easy}}\;\;
   (-1)^{a-3}\;
     \xy (0,0)*{\includegraphics[scale=0.5,]{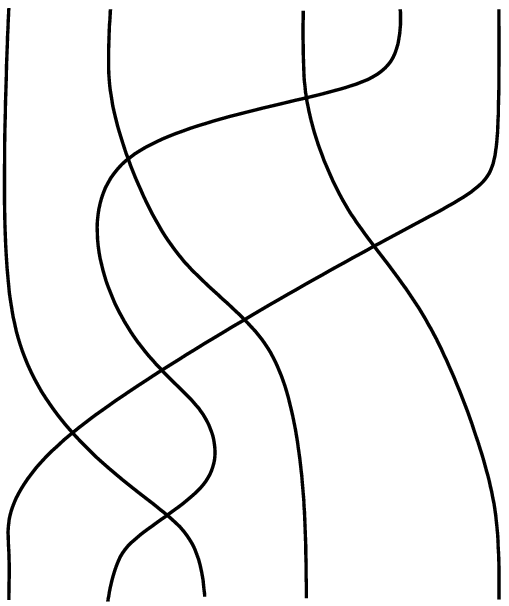}};
  (-2,12)*{\cdots};(7,-12)*{\cdots};\endxy
\end{eqnarray}
so that using the induction hypothesis the left side of \eqref{eq_cross_exchange} can be written as
\begin{equation}
  =\quad (-1)^{a-3}\;
     \xy (0,0)*{\includegraphics[scale=0.5,]{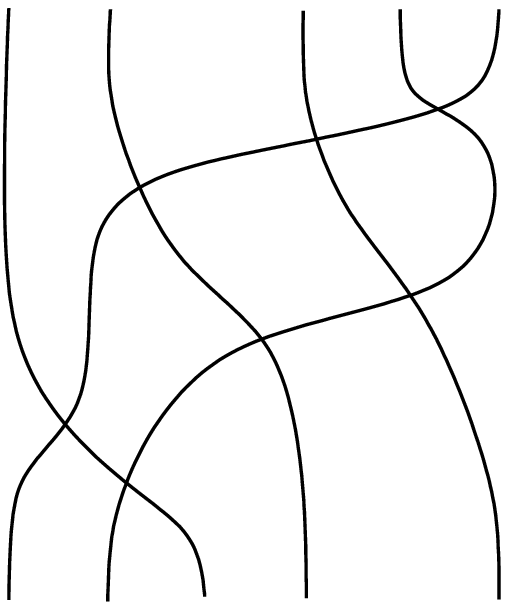}};
  (-2,12)*{\cdots};(7,-12)*{\cdots};\endxy
  \;\; =\;\;
 \xy (0,0)*{\includegraphics[scale=0.5,]{5ur-cross.eps}};
  (-2,12)*{\cdots};(7,-12)*{\cdots};\endxy.
\end{equation}
\end{proof}

\begin{lem}[Alternative definition of $D_a$] \label{lem_alt_def}
The element $D_a$ defined in \eqref{eq_def_Da} can also be inductively defined by the equation
\begin{equation} \label{eq_altdefDa}
  \xy
 (0,0)*{\includegraphics[scale=0.5]{5box.eps}};
 (0,0)*{D_a};
  \endxy
 \quad = \quad
 \xy
 (0,0)*{\reflectbox{\includegraphics[scale=0.5]{4box-1bl-r.eps}}};
 (2,1)*{D_{a-1}};
  \endxy
\end{equation}
\end{lem}

\begin{proof}
The proof is by induction with the base case following from \eqref{new_eq_r3_easy}. The left hand side can be re-written using the definition of $D_a$
\begin{equation}
 \xy
 (0,0)*{\reflectbox{\includegraphics[scale=0.5]{4box-1bl-r.eps}}};
 (2,1)*{D_{a-1}};
  \endxy
\quad \refequal{\eqref{eq_def_Da}} \quad
 \xy
 (0,0)*{\includegraphics[scale=0.5]{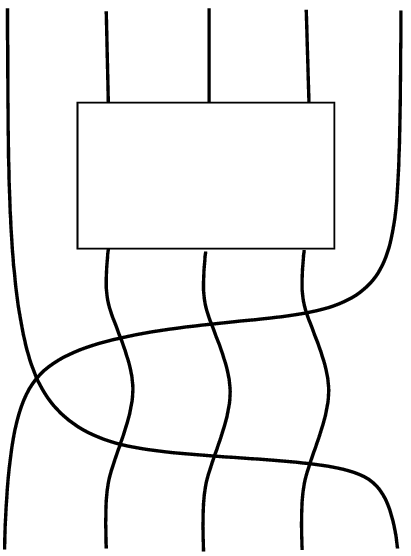}};
 (0,5)*{D_{a-2}};
  \endxy
\quad = \quad
 \xy
 (0,0)*{\includegraphics[scale=0.5]{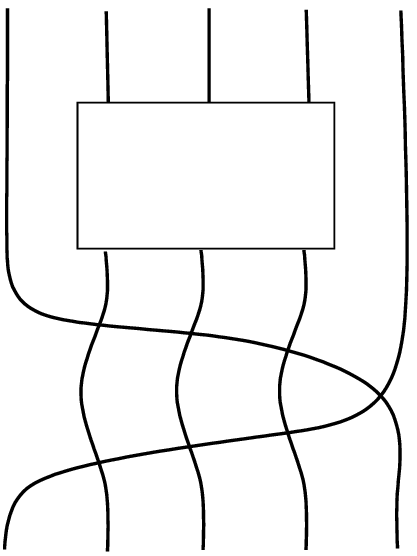}};
 (0,5)*{D_{a-2}};
  \endxy
   \quad = \quad
 \xy
 (0,0)*{\includegraphics[scale=0.5]{4box-1bl-r.eps}};
 (-2,1)*{D_{a-1}};
  \endxy,
\end{equation}
where the second to last equality follows by repeatedly applying \eqref{new_eq_r3_easy}, and the last equality follows from the induction hypothesis.
\end{proof}

\begin{lem}[$D_a$ slide] For $a\geq 1$ the equation
\begin{equation} \label{eq_Da-slide}
  \xy (0,0)*{\includegraphics[scale=0.5,]{4box-1bl-r.eps}};
  (-3,1)*{D_a};\endxy
  \quad = \quad (-1)^{\binom{a}{3}}\;\;
  \xy (0,0)*{\includegraphics[scale=0.5,angle=180]{4box-1bl-r.eps}};
  (3,-1)*{D_a}; \endxy
\end{equation}
holds in $\ONH_a$.
\end{lem}

\begin{proof}
The proof is by induction on the number of strands. The base case is trivial, so we assume the result holds for up to $a-1$ strands.
\begin{eqnarray}
  \xy (0,0)*{\includegraphics[scale=0.5,]{4box-1bl-r.eps}};
  (-3,1)*{D_a};\endxy
  \quad \refequal{\eqref{eq_def_Da}} \quad
    \xy (0,0)*{\includegraphics[scale=0.5,]{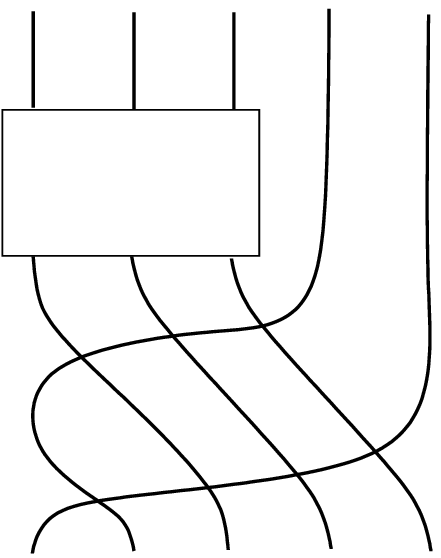}};
  (-4,5)*{D_{a-1}};\endxy
    \quad \refequal{\eqref{eq_cross_exchange}} \quad
    \xy (0,0)*{\includegraphics[scale=0.5,]{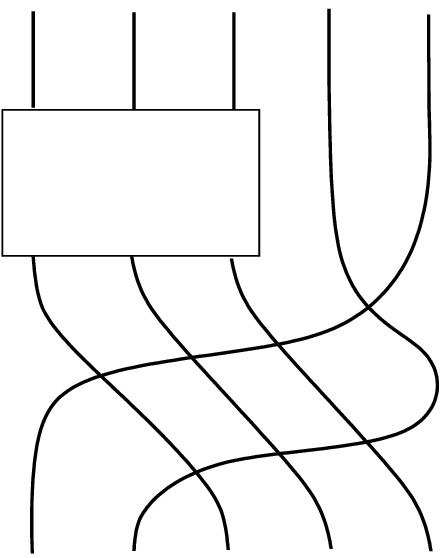}};
  (-4,5)*{D_{a-1}};\endxy \hspace{1.2in}
\\
 = \quad (-1)^{\binom{a-1}{2}}
 \xy (0,0)*{\includegraphics[scale=0.5,]{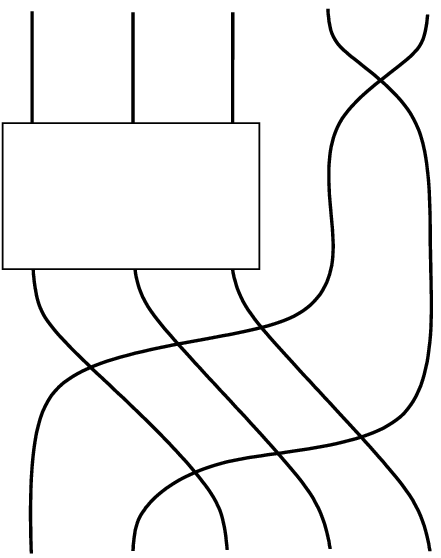}};
  (-4,4)*{D_{a-1}};\endxy
\quad = \quad
(-1)^{\binom{a-1}{2}+\binom{a-1}{3}}
 \xy (0,0)*{\includegraphics[scale=0.5,]{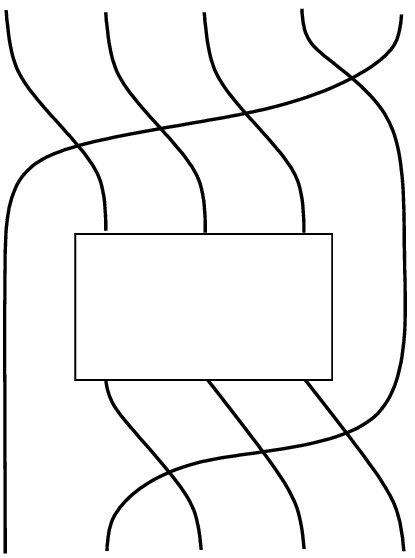}};
  (0,-2)*{D_{a-1}};\endxy,
\end{eqnarray}
where we used the induction hypothesis in the last step.  The result follows using the definition \eqref{eq_def_Da} of $D_a$ and the identity
\begin{equation}\label{eqn-binom}
\binom{a-1}{2}+\binom{a-1}{3} = \binom{a}{3}.
\end{equation}
\end{proof}

%
\subsubsection{Relations involving the 0-Hecke generators}
%

It is a simple calculation to see that
\begin{equation}\label{eqn-r3-one-dotted}
  \xy (0,0)*{\reflectbox{\includegraphics[scale=0.4]{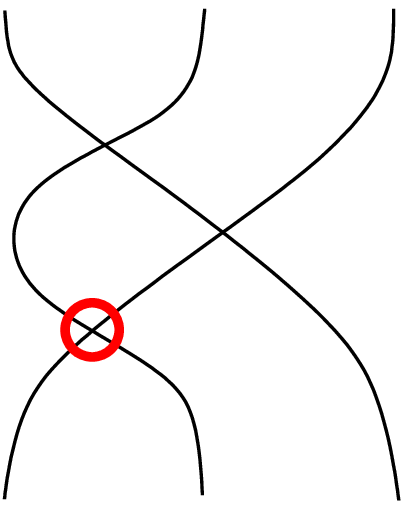}}}; \endxy
  \quad = \quad
  \xy (0,0)*{\includegraphics[scale=0.4,]{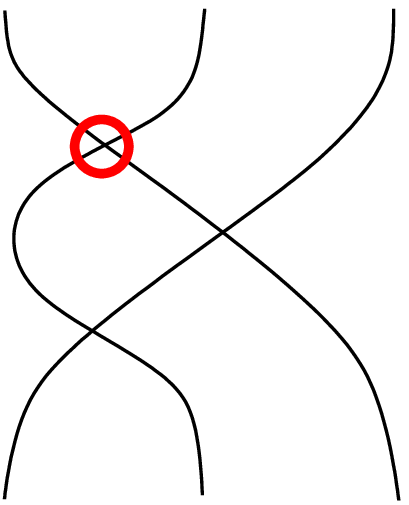}}; \endxy
\end{equation}
so that the relation
\begin{equation}\label{eq_ea_slide_left}
  \xy (0,0)*{\includegraphics[scale=0.5,angle=180]{4box-1bl-r.eps}};
  (3,-1)*{e_a}; \endxy
  \quad = \quad
  \xy (0,0)*{\includegraphics[scale=0.5,]{4box-1bl-r.eps}};
  (-3,1)*{e_a};\endxy
\end{equation}
holds.

\begin{rem} Note that an analogous relation
\begin{equation}
  \xy (0,0)*{\reflectbox{\includegraphics[scale=0.5,angle=180]{4box-1bl-r.eps}}};
  (-3,-1)*{e_a}; \endxy
  \quad \neq \quad
  \xy (0,0)*{\reflectbox{\includegraphics[scale=0.5,]{4box-1bl-r.eps}}};
  (3,1)*{e_a};\endxy
\end{equation}
does not hold for sliding a projector $e_a$ from the bottom left to the top right through a strand.  This is because the reflection of equation \eqref{eqn-r3-one-dotted} through a horizontal axis is false.\end{rem}

It follows from \eqref{new_eq_iislide} that
\begin{equation} \label{eq_Or2}
\xy (0,0)*{\includegraphics[scale=0.4]{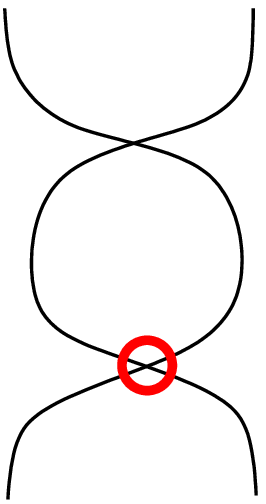}}; \endxy
 \quad = \quad
 \xy  (0,0)*{\includegraphics[scale=0.4]{2thin-cross.eps}};\endxy.
\end{equation}

It will be convenient to express elements $e_a$ in terms of one of the bases \eqref{eqn-onh-bases}.

\begin{prop}[Projector in standard basis] \label{prop_ea-standard}
\begin{equation}
    \xy
 (0,0)*{\includegraphics[scale=0.5]{1box.eps}};
 (0,0)*{e_a};
  \endxy
  \quad = \quad   (-1)^{\binom{a}{3}}\;\;\xy (0,0)*{\includegraphics[scale=0.5]{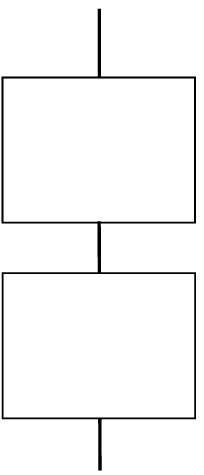}};
   (0,4)*{\delta_a };  (0,-6)*{D_a };    \endxy
\end{equation}
\end{prop}

\begin{proof}
We prove the result by induction. The base case is trivial. Assume the result follows for $e_{a-1}$.  From the definition of $e_a$ we have
\begin{equation}
      \xy
 (0,0)*{\includegraphics[scale=0.5]{1box.eps}};
 (0,0)*{e_a};
  \endxy
  \quad = \quad
  \xy
 (0,0)*{\includegraphics[scale=0.5]{5box.eps}};
 (0,0)*{e_a};
  \endxy
  \quad = \quad
  \xy
 (0,0)*{\includegraphics[scale=0.5,angle=180]{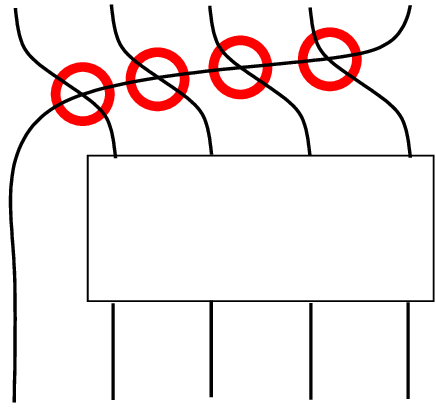}};
 (3,-2)*{e_{a-1}};
  \endxy
\end{equation}
Sliding the dots from the 0-Hecke generators up to the top of the diagram we can write
\begin{equation}
 \xy (0,0)*{\includegraphics[scale=0.5]{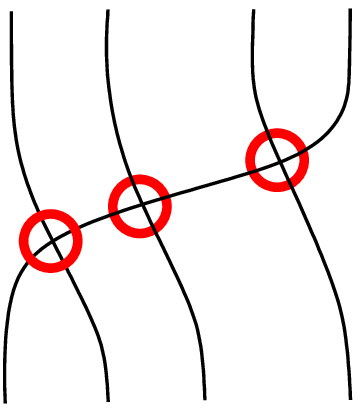}};
   (0,3)*{\cdots };   \endxy
  \quad = \quad (-1)^{\sum_{j=1}^{a-2}j}\;\;
   \xy (0,0)*{\includegraphics[scale=0.5]{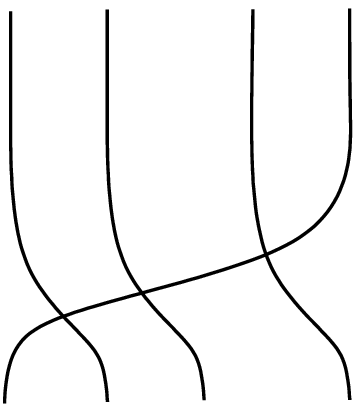}};
   (0,3)*{\cdots };
  (-8.5,3)*{\bullet};(-3.5,5.5)*{\bullet}; (3.7,7)*{\bullet};  \endxy
   \quad = \quad
   \xy (0,0)*{\includegraphics[scale=0.5]{4thin-leftright.eps}};
   (0,3)*{\cdots };
  (-8.5,7)*{\bullet};(-3.5,5.5)*{\bullet}; (3.7,3)*{\bullet};  \endxy
\end{equation}
so that
\begin{equation}
  \xy
 (0,0)*{\includegraphics[scale=0.5,angle=180]{4box-O1bl-r.eps}};
 (3,-2)*{e_{a-1}};
  \endxy \quad = \quad
 \xy (0,0)*{\includegraphics[scale=0.5,angle=180]{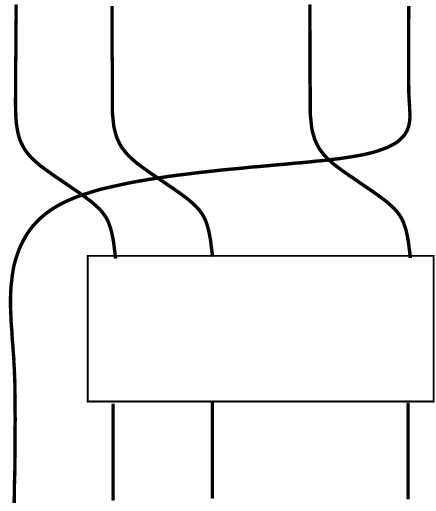}};
   (0,8)*{\cdots };  (5,-10)*{\cdots }; (3,-4)*{e_{a-1}};
   (-10.5,11)*{\bullet};(-5.5,9.5)*{\bullet}; (4.5,7)*{\bullet}; \endxy
  \quad \refequal{\eqref{eq_ea_slide_left}} \quad
   \xy (0,0)*{\includegraphics[scale=0.5]{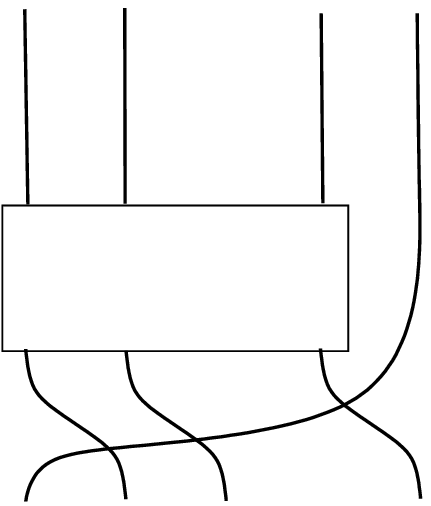}};
   (0,8)*{\cdots };  (1,-7)*{\cdots }; (-1,-1)*{e_{a-1}};
   (-9.5,9)*{\bullet};(-4.5,7.5)*{\bullet}; (5.5,5)*{\bullet}; \endxy
\end{equation}
Using the induction hypothesis we write $e_{a-1}$ in standard form
\begin{equation}
\ldots\quad = \quad
 (-1)^{\binom{a-1}{3}}\;\;\xy (0,-6)*{\includegraphics[scale=0.5]{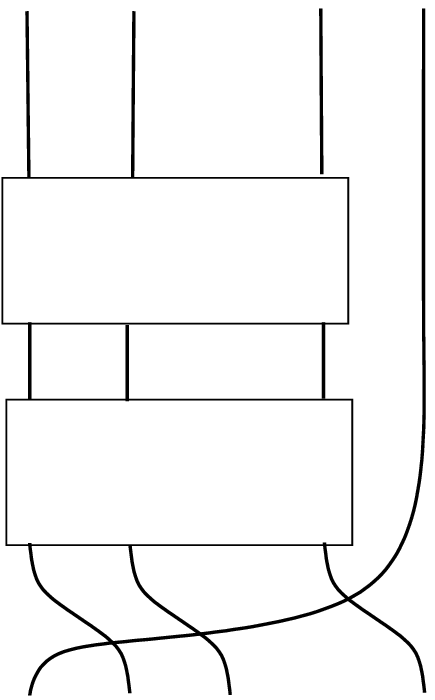}};
   (0,6)*{\cdots };  (1,-7)*{\cdots }; (-1,-1)*{\delta_{a-1}}; (-1,-12)*{D_{a-1}};
   (-9.5,9)*{\bullet};(-4.2,7.5)*{\bullet}; (5.5,5)*{\bullet}; \endxy
\end{equation}
then slide the dots into the correct position using the identity
\begin{equation}
 \xy (0,0)*{\includegraphics[scale=0.5]{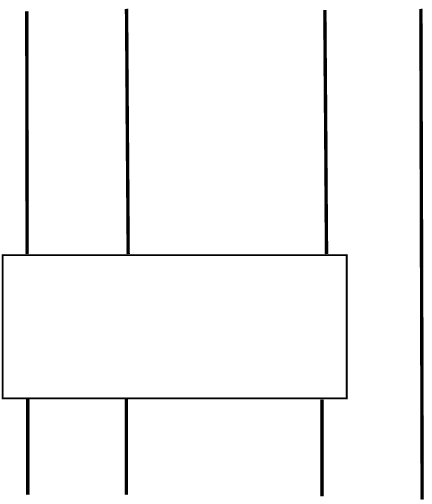}};
  (-2,-4)*{\delta_{a-1}}; (0,3)*{\cdots }; (0,-11)*{\cdots };
  (-9.5,7)*{\bullet};(-4.5,5.5)*{\bullet}; (5.5,3)*{\bullet};  \endxy
  \quad = \quad (-1)^{\sum_{j=1}^{a-2}\sum_{\ell=j}^{a-2}\ell} \;
 \xy (0,0)*{\includegraphics[scale=0.5]{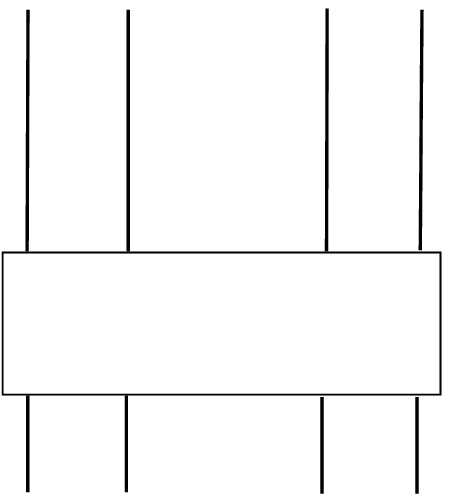}};
  (0,-4)*{\delta_a}; (0,3)*{\cdots }; (0,-11)*{\cdots };\endxy
   \quad = \quad (-1)^{\binom{a-1}{2}} \;
 \xy (0,0)*{\includegraphics[scale=0.5]{4box-middlegap.eps}};
  (0,-4)*{\delta_a}; (0,3)*{\cdots }; (0,-11)*{\cdots };\endxy
\end{equation}
which follows by sliding the lower right most dot down past each of the $\sum_{j=1}^{a-2}j$ dots, then the second dot from the right, and so on. The Proposition follows from the inductive definition \eqref{eq_def_Da} of $D_a$ and the binomial identity \eqref{eqn-binom}.
\end{proof}
The following proposition shows that the elements $e_a$ are idempotents in $\ONH_a$.
\begin{prop} \hfill \newline \label{prop_ea_idemp}
The diagrammatic identities
\begin{enumerate}[1)]
\item
$  \xy
 (0,0)*{\includegraphics[scale=0.5]{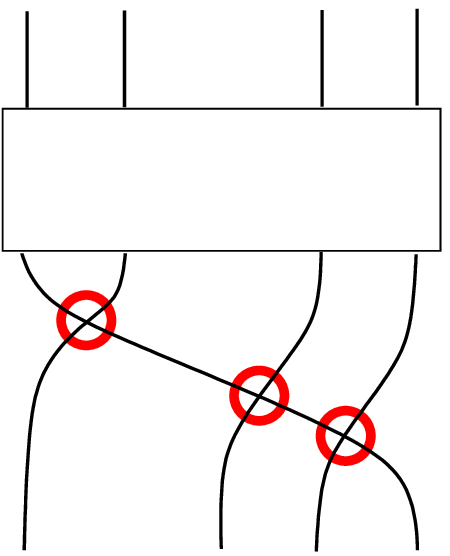}};
 (0,5)*{D_a}; (-5,-10)*{\cdots}; (0,12)*{\cdots};
  \endxy \quad = \quad
   \xy
 (0,0)*{\includegraphics[scale=0.5]{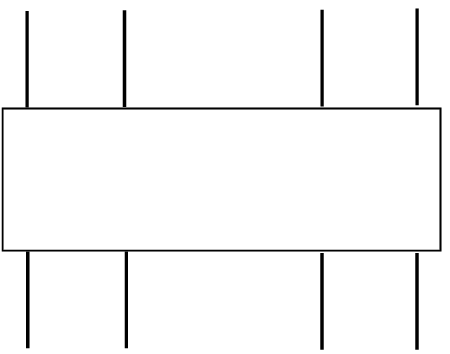}};
 (0,0)*{D_a}; (0,-6)*{\cdots}; (0,6)*{\cdots};
  \endxy,$

  \item $  \xy (0,0)*{\includegraphics[scale=0.5]{1two-box.eps}};
   (0,4)*{D_a };  (0,-6)*{e_a };    \endxy
  \quad = \quad  \xy
 (0,0)*{\includegraphics[scale=0.5]{1box.eps}};
 (0,0)*{D_a};
  \endxy, $

 \item $  \xy (0,0)*{\includegraphics[scale=0.5]{1two-box.eps}};
   (0,4)*{e_a };  (0,-6)*{e_a };    \endxy
  \quad = \quad  \xy
 (0,0)*{\includegraphics[scale=0.5]{1box.eps}};
 (0,0)*{e_a};
  \endxy $
\end{enumerate}
hold in $\ONH_a$.
\end{prop}

\begin{proof}
Part 1) follows from the computation
\begin{align}
  \xy
 (0,0)*{\includegraphics[scale=0.5]{Da_underOdot.eps}};
 (0,5)*{D_a}; (-5,-10)*{\cdots}; (0,12)*{\cdots};
  \endxy
  &\quad \refequal{\eqref{eq_def_Da}} \quad
  \xy
 (0,0)*{\includegraphics[scale=0.5]{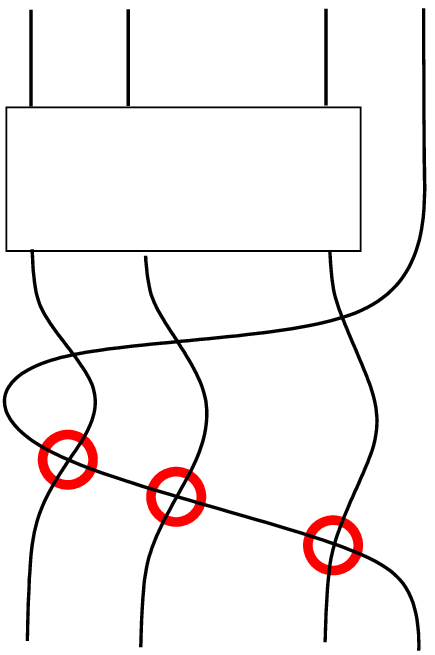}};
 (-2,7)*{D_{a-1}}; (1,-13)*{\cdots}; (0,13)*{\cdots};
  \endxy \quad \refequal{\eqref{eq_Or2}} \quad
    \xy
 (0,0)*{\includegraphics[scale=0.5]{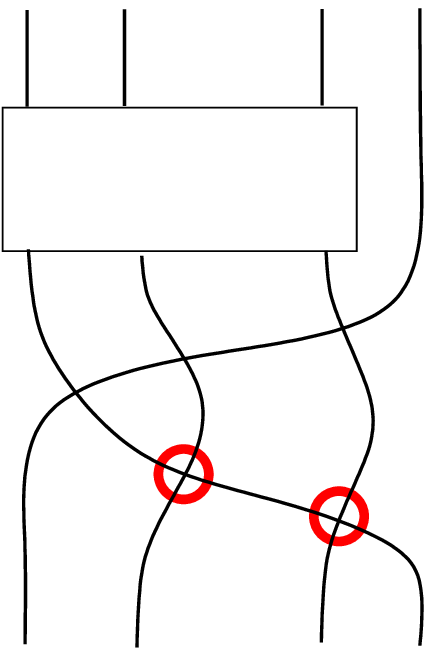}};
 (-2,7)*{D_{a-1}}; (1,-13)*{\cdots}; (0,13)*{\cdots};
  \endxy \nn\\
&\quad \refequal{\eqref{eq_Da-slide}} \quad (-1)^{\binom{a-1}{3}}
    \xy
 (0,0)*{\includegraphics[scale=0.5]{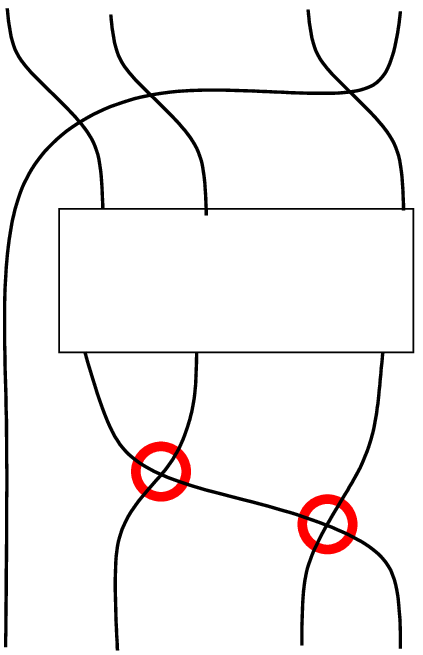}};
 (2,2)*{D_{a-1}}; (0,-14)*{\cdots}; (0,14)*{\cdots};
  \endxy
 \quad = \quad  (-1)^{\binom{a-1}{3}}
 \xy (0,0)*{\includegraphics[scale=0.5,angle=180]{4box-1bl-r.eps}};
  (3,-1)*{D_{a-1}}; \endxy,
  \nn
\end{align}
where the last equality follows by the induction hypothesis.  This proves part 1) using equation \eqref{eq_Da-slide} to slide $D_{a-1}$ back through the line.

For the second claim observe that
\begin{align}
 \xy (0,0)*{\includegraphics[scale=0.5]{1two-box.eps}};
   (0,4)*{D_a };  (0,-6)*{e_a };    \endxy
   \;\; = \;\;  \xy
 (0,0)*{\includegraphics[scale=0.5]{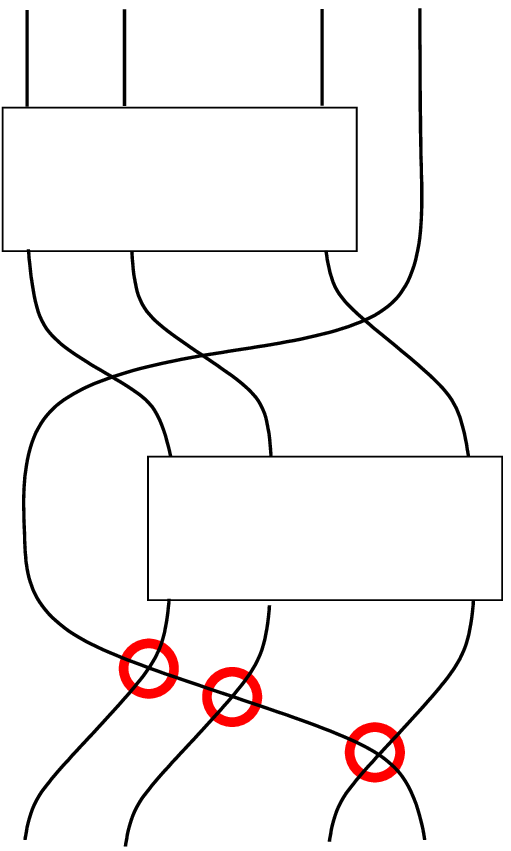}};
 (-4,13)*{D_{a-1}}; (4,-6)*{e_{a-1}}; (5,-13)*{\cdots}; (-2,19)*{\cdots};
  \endxy
     \;\; \refequal{\eqref{eq_ea_slide_left}} \;\;  \xy
 (0,0)*{\includegraphics[scale=0.5]{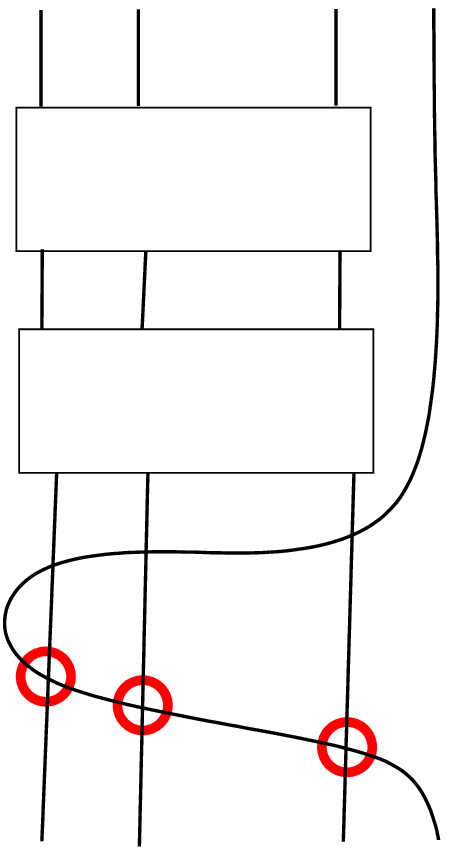}};
 (-2,13)*{D_{a-1}}; (-2,1)*{e_{a-1}}; (1,-13)*{\cdots}; (1,19)*{\cdots};
  \endxy
       \;\; = \;\;  \xy
 (0,0)*{\includegraphics[scale=0.5]{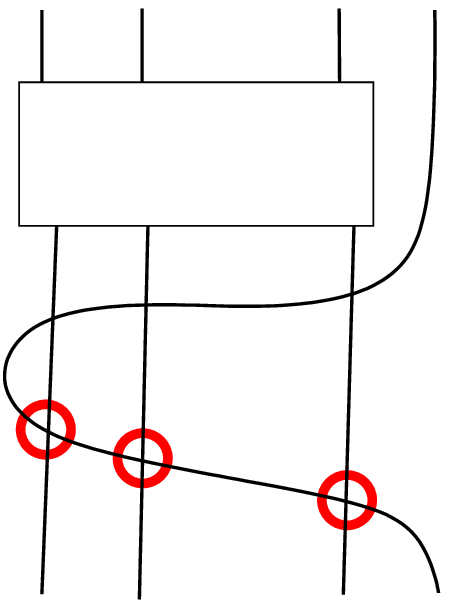}};
 (-2,8)*{D_{a-1}};  (1,-13)*{\cdots}; (1,13)*{\cdots};
  \endxy
 \;\; \refequal{\eqref{eq_altdefDa}} \;\;
   \xy
 (0,0)*{\includegraphics[scale=0.5]{Da_underOdot.eps}};
 (0,5)*{D_a}; (-5,-10)*{\cdots}; (0,12)*{\cdots};
  \endxy,
\end{align}
where the third equality follows by induction. The second claim follows from part 1). Part 3) follows from 2) using Proposition~\ref{prop_ea-standard} since
\begin{equation}
  e_a e_a = (-1)^{\binom{a}{3}} \undx^{\delta_a} D_a e_a = (-1)^{\binom{a}{3}} \undx^{\delta_a} D_a =e_a.
\end{equation}
\end{proof}

One can also prove the following equalities from the 0-Hecke relations.
\begin{equation} \label{eq_ea_absorb}
\xy
 (0,0)*{\includegraphics[scale=0.5]{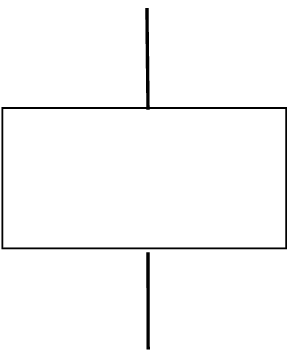}};
 (0,0)*{e_{a+b+c}};
  \endxy\;\; = \;\;
    \xy
 (0,0)*{\includegraphics[scale=0.5]{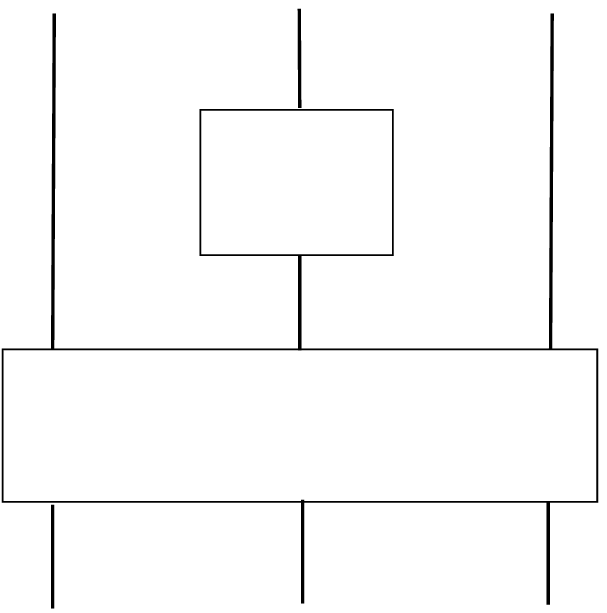}};
 (0,6)*{e_{b}}; (0,-6)*{e_{a+b+c}}; (-15,12)*{a}; (15,12)*{c};
  \endxy
 \;\; = \;\;
     \xy
 (0,0)*{\includegraphics[scale=0.5, angle=180]{pp3.eps}};
 (0,-6)*{e_{b}}; (0,6)*{e_{a+b+c}}; (-15,-12)*{a}; (15,-12)*{c};
  \endxy
\end{equation}

%
\subsubsection{Crossings and projectors}
%

Just as $a$ vertical lines can be represented by a single line labelled $a$, we denote a specific choice of crossings of $a$ strands with $b$ strands as follows:
\begin{equation}\label{eqn-defn-crossing}
 \xy  (0,0)*{\includegraphics[scale=0.4]{2thin-cross.eps}};
 (-8,6)*{a};
 (8,6)*{b};\endxy
 \;\; = \;\;
 \xy (0,0)*{\includegraphics[scale=0.5,]{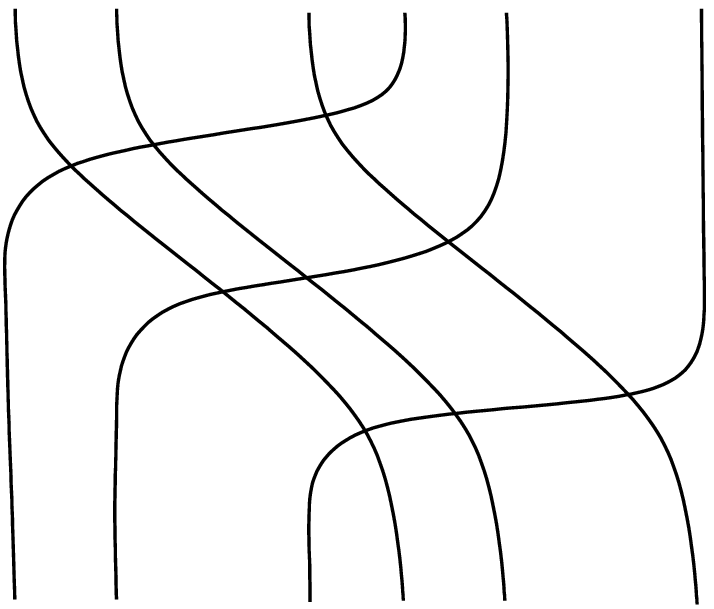}};
  (-7,12)*{\cdots}; (-7,-12)*{\cdots};(12,12)*{\cdots};(12,-12)*{\cdots};
  (-10,17)*{\overbrace{\hspace{0.7in}}};
  (10,17)*{\overbrace{\hspace{0.7in}}};
  (-10,20)*{a}; (10,20)*{b};\endxy
\end{equation}
The first of the $b$ strands crosses all of the $a$ strands, then the second of the $b$ strands crosses all of the $a$ strands, and so on.  One can easily verify that
\begin{equation} \label{eq_crossing-combine}
  \xy (0,0)*{\reflectbox{\includegraphics[scale=0.5,]{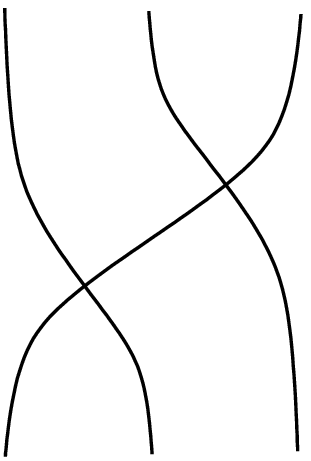}}};
   (-10,10)*{a}; (-2,10)*{b}; (10,10)*{c};
 \endxy
 \;\; = \;\;
 \xy  (0,0)*{\includegraphics[scale=0.4]{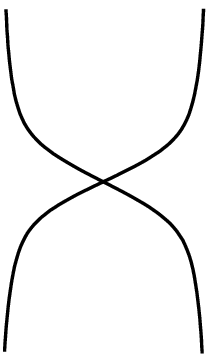}};
 (-7,6)*{a};
 (12,6)*{b+c};\endxy
 \qquad \qquad
  \xy (0,0)*{\includegraphics[scale=0.5,]{crossing-combine2.eps}};
   (-10,10)*{a}; (-2,10)*{b}; (10,10)*{c};
 \endxy
 \;\; = \;\; (-1)^{ab\binom{c}{2}}\;
 \xy  (0,0)*{\includegraphics[scale=0.4]{crossing-combine1.eps}};
 (-12,6)*{a+b};
 (7,6)*{c};\endxy
\end{equation}

The respective heights of $e_a$ and $e_b$ in the diagram
 \begin{equation}
      \xy
 (0,0)*{\includegraphics[scale=0.5]{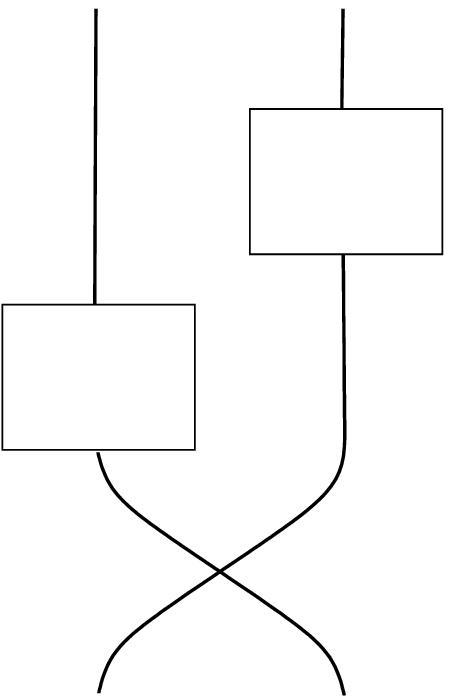}};
 (-7,-15)*{b};(7,-15)*{a};(-6,-2)*{e_a};(6,9)*{e_b}; 
  \endxy \;\; = \;\;
        \xy
 (0,0)*{\includegraphics[scale=0.5]{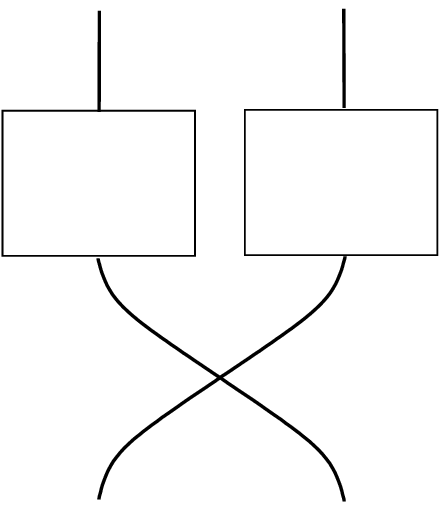}};
 (-7,-10)*{b};(7,-10)*{a};(-6,4)*{e_a};(6,4)*{e_b};
  \endxy
\end{equation}
are not relevant as these elements have even super-degrees and commute with each other.  The equation
\begin{equation} \label{eq_eaebcross}
  \xy
 (0,0)*{\includegraphics[scale=0.5]{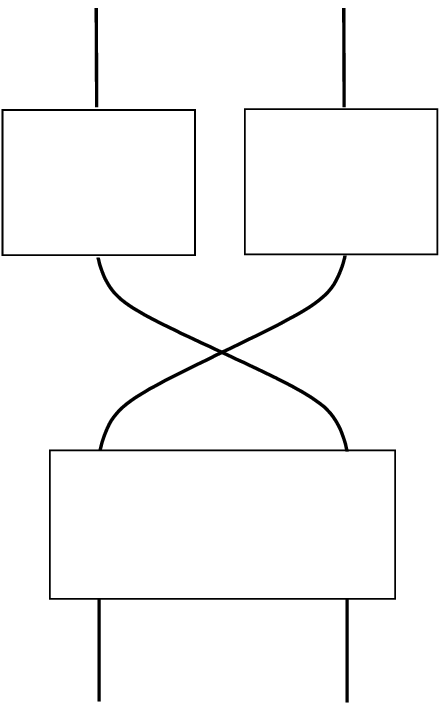}};
 (0,-8.5)*{e_{a+b}};(-6,8.5)*{e_{a}};(6,8.5)*{e_{b}};
  \endxy
 \;\; = \;\;
      \xy
 (0,0)*{\includegraphics[scale=0.5]{def-tsplitu2.eps}};
 (-7,-10)*{b};(7,-10)*{a};(-6,4)*{e_a};(6,4)*{e_b};
  \endxy
\end{equation}
holds.

\begin{prop}
\begin{align} \label{eq_DaDbcross1}
     \xy
 (0,0)*{\reflectbox{\includegraphics[scale=0.5]{def-tsplitu.eps}}};
 (-7,-15)*{b};(7,-15)*{a};(-6,9)*{D_a};(6,-2)*{D_b}; 
  \endxy      \;\; = \;\;
   \;
  \xy
 (0,0)*{\includegraphics[scale=0.5]{1box-wide.eps}};
 (0,0)*{D_{a+b}};
  \endxy  \;\; = \;\; (-1)^{\binom{a}{2}\binom{b}{2}}\;\;
 \xy
 (0,0)*{\includegraphics[scale=0.5]{def-tsplitu.eps}};
 (-7,-15)*{b};(7,-15)*{a};(-6,-2)*{D_a};(6,9)*{D_b}; 
  \endxy
\end{align}
\end{prop}

\begin{proof}
The second equality follows immediately from the first. To prove the first, observe that
\[
     \xy
 (0,0)*{\reflectbox{\includegraphics[scale=0.5]{def-tsplitu.eps}}};
 (-7,-15)*{b};(7,-15)*{a};(-6,9)*{D_a};(6,-2)*{D_b}; 
  \endxy  \;\; = \;\;
  \xy
 (0,0)*{\includegraphics[scale=0.4]{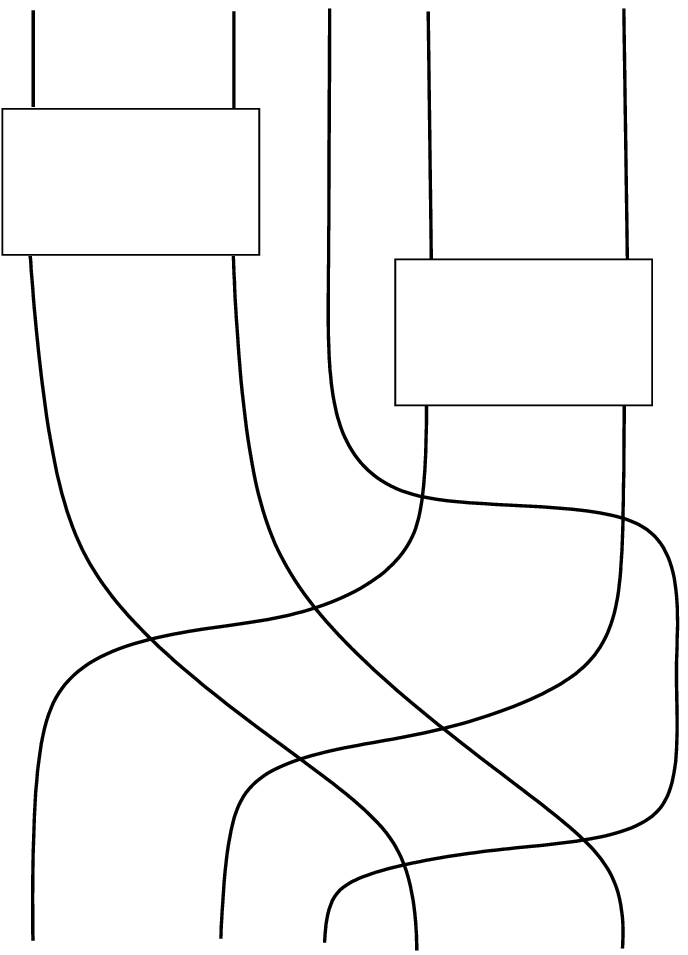}};
 (-9,12)*{D_{a}}; (7.5,6)*{D_{b-1}};
  \endxy
  \;\; = \;\;(-1)^{a\binom{b-1}{2}}
    \xy
 (0,0)*{\includegraphics[scale=0.4]{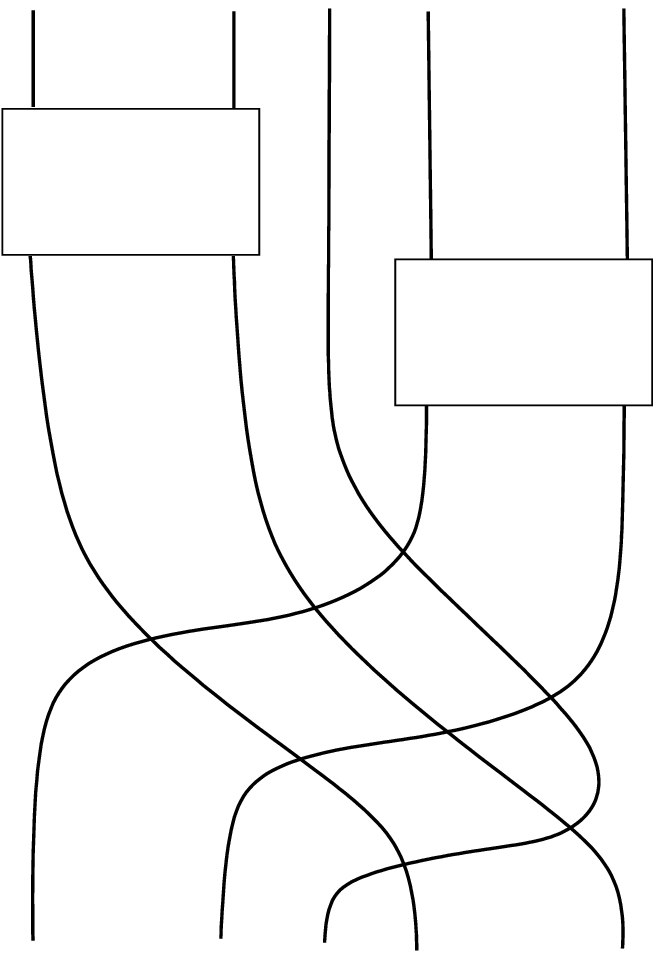}};
 (-9,12)*{D_{a}}; (7.5,6)*{D_{b-1}};
  \endxy
      \;\; = \;\;(-1)^{a\binom{b-1}{2}}
    \xy
 (0,0)*{\includegraphics[scale=0.4]{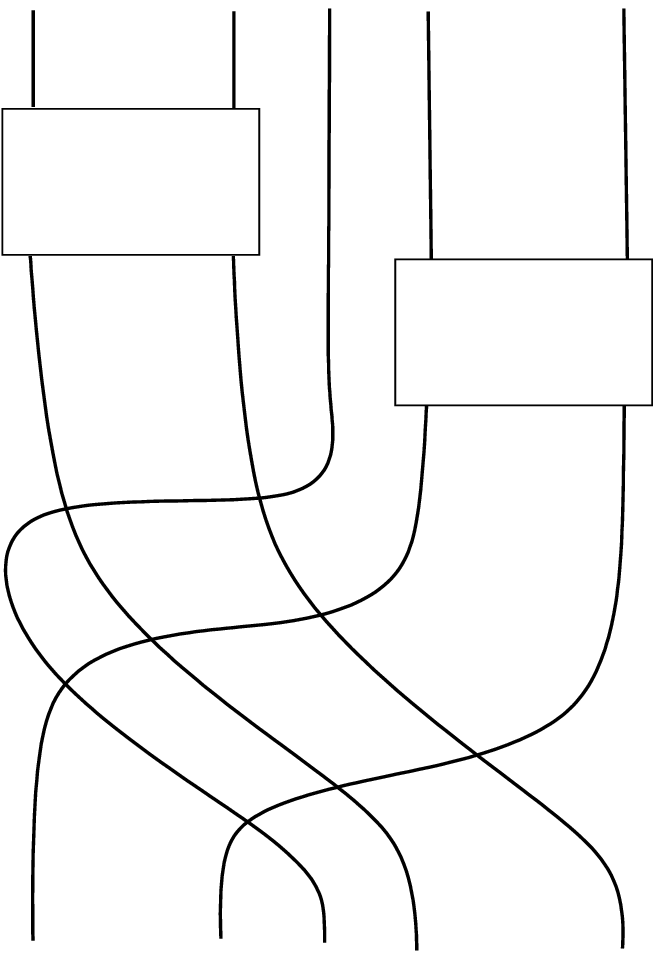}};
 (-9,12)*{D_{a}}; (7.5,6)*{D_{b-1}};
  \endxy
\]
where the sign in the second equality arises from sliding each of the $b-1$ crossings from $D_b$ down. The third equality follows by repeated application of the Crossing Slide Lemma~\ref{lem_crossing_slide}.  Sliding the $D_{b-1}$ down past the $a$ crossings and using the inductive definition of $D_{a+1}$ shows that
\[
     \xy
 (0,0)*{\reflectbox{\includegraphics[scale=0.5]{def-tsplitu.eps}}};
(-6,9)*{D_a};(6,-2)*{D_b}; 
  \endxy  \;\; = \;\;
       \xy
 (0,0)*{\reflectbox{\includegraphics[scale=0.5]{def-tsplitu.eps}}};
(-6,9)*{D_{a+1}};(6,-2)*{D_{b-1}}; 
  \endxy,
  \]
  so that the result holds by induction.
\end{proof}

%
\subsection{Automorphisms of the odd nilHecke algebra} \label{sec_switch}
%

Let $\sigma \maps \ONH_a \to \ONH_a$ denote the automorphism given by reflecting diagrams across the vertical axis.  Let $\psi \maps \ONH_a \to \ONH_a^{\op}$ denote the anti-automorphism given by reflecting diagrams across the horizontal axis. These automorphisms commute with each other.

Below we collect the effect of applying these automorphisms to the elements $\undx^{\delta_a}$, $D_a$, and $e_a$.
One can check that
\begin{alignat}{3}
  \undx^{\delta_a}
   \;\;&= \;\; x_1^{a-1}x_2^{a-2} \dots x_{a-1}^{1}x_a^{0} & \\
\sigma(\undx^{\delta_a})
  \;\;&= \;\;x_a^{a-1} x_{a-1}^{a-2} \dots x_2^1 x_1^0& \\
  \psi(\undx^{\delta_a})
   \;\;&= \;\;  x_a^{0} x_{a-1}^{1} \dots x_2^{a-2} x_1^{a-1}& \\
  \psi\sigma(\undx^{\delta_a})
   \;\;&= \;\; x_1^0 x_2^1 \dots x_{a-2}^{a-1}x_a^{a-1}. &
\end{alignat}

These monomials are related by the equations
\begin{equation}
  \undx^{\delta_a}=(-1)^{\binom{a}{4}} \psi(\undx^{\delta_a}), \qquad
   \sigma(\undx^{\delta_a}) =(-1)^{\binom{a}{4}} \sigma\psi(\undx^{\delta_a}).
\end{equation}
Similarly,
\begin{align}
  D_a
  \;\; &:= \;\;
   \partial_1(\partial_2\partial_1) \dots (\partial_{a-2} \dots \partial_1)(\partial_{a-1} \dots\partial_1) \nn \\
 \sigma(D_a)
  \;\; &= \;\;
  \partial_{a-1}(\partial_{a-2} \partial_{a-1})(\partial_2 \dots \partial_{a-1})(\partial_1 \dots \partial_{a-1}) \nn \\
 \psi(D_a)
 \;\; &= \;\;
  (\partial_{1} \dots \partial_{a-1})(\partial_1 \dots \partial_{a-2}) \dots (\partial_1\partial_2) \partial_1 \nn \\
 \psi\sigma(D_a)
 \;\; &= \;\;
  (\partial_{a-1} \dots \partial_1)(\partial_{a-1} \dots \partial_2) \dots(\partial_{a-1}\partial_{a-2}) \partial_{a-1}, \nn
\end{align}
and it follows from Lemma~\ref{lem_alt_def} that
\begin{align}
  D_a = \sigma(D_a) = (-1)^{\binom{a-1}{4}}\psi(D_a)  = (-1)^{\binom{a-1}{4}}\psi\sigma(D_a) .
\end{align}

We have shown in Proposition~\ref{prop_ea-standard} that
\begin{equation} \label{eq_ea_standard}
  e_a = (-1)^{\binom{a+1}{4}+\binom{a}{4}} \undx^{\delta_a}D_a = (-1)^{\binom{a}{3}} \undx^{\delta_a}D_a
\end{equation}
It is also worth while to write this equation in another form,
\begin{align}
  e_a &= (-1)^{\binom{a+1}{4}} \psi(\undx^{\delta_a}) D_a.
\end{align}
It follows from Proposition~\ref{prop_ea_idemp} that
\begin{equation} \label{eq_Da-undxdelta}
D_a(\undx^{\delta_a}) = (-1)^{\binom{a}{3}} \qquad  \text{and} \qquad D_a(\psi(\undx^{\delta_a})) = (-1)^{\binom{a+1}{4}}.
\end{equation}

%
\section{Extended graphical calculus for the odd nilHecke algebra}\label{sec-thick}
%

%
\subsection{Thick calculus}
%
In this section we introduce an extension of the graphical calculus from the previous section.

%
\subsubsection{Splitters}
%

Define
\[ 
  \xy
 (0,0)*{\includegraphics[scale=0.5]{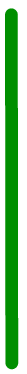}};
 (-2.5,-6)*{a}; 
  \endxy
  \quad : = \quad
\xy
 (0,0)*{\includegraphics[scale=0.5]{1box.eps}};
 (0,0)*{e_a};
  \endxy
\]
The notation is consistent, since $e_a$ is an idempotent, so that
cutting a thick line into two pieces and converting each of them
into an $e_a$ box results in the same element of the odd nilHecke ring.

The following diagrams will be referred to as {\em splitters} or
splitter diagrams.
\begin{equation} \label{eq_def-splitters}
    \xy
 (0,0)*{\includegraphics[scale=0.5]{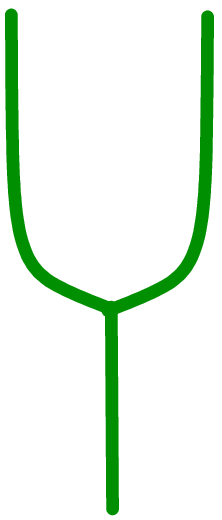}};
 (-5,-10)*{a+b};(-8,4)*{a};(8,4)*{b};
  \endxy
 \;\; :=\;\;
   \xy
 (0,0)*{\includegraphics[scale=0.5]{def-tsplitu2.eps}};
 (-7,-10)*{b};(7,-10)*{a};(-6,4)*{e_a};(6,4)*{e_b};
  \endxy
\qquad  \qquad 
    \xy
 (0,0)*{\includegraphics[scale=0.5,angle=180]{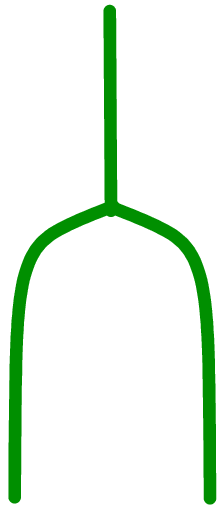}};
 (-5,10)*{a+b};(-8,-4)*{a};(8,-4)*{b}; 
  \endxy
  \;\; :=\;\;
     \xy
 (0,0)*{\includegraphics[scale=0.5,angle=180]{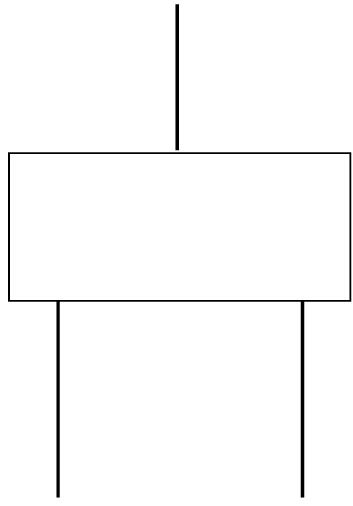}};
 (-8,-7)*{a};(8,-7)*{b};(0,1)*{e_{a+b}};(-5,10)*{a+b}; 
  \endxy
\end{equation}

Define a thick crossing by
\begin{align} \label{eq_defn_thick_cross}
\xy
 (0,0)*{\includegraphics[scale=0.5]{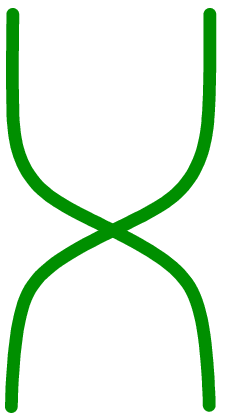}};
(-7,-7.5)*{a};(7,-7.5)*{b};
  \endxy \quad := \quad  \xy
 (0,0)*{\includegraphics[scale=0.45]{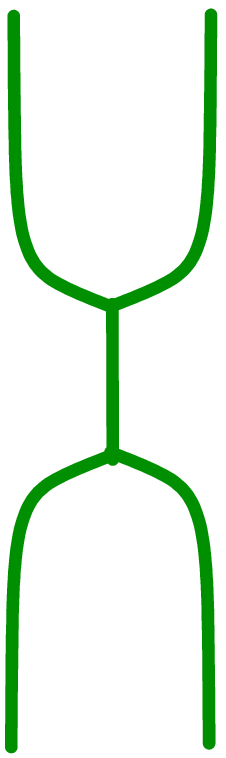}};
(-8,-11)*{a};(8,-11)*{b};(-8,11)*{b};(8,11)*{a};
  \endxy
 \quad = \quad
  \xy
 (0,0)*{\includegraphics[scale=0.5]{pp1.eps}};
 (0,-8.5)*{e_{a+b}};(-6,8.5)*{e_{b}};(6,8.5)*{e_{a}};
  \endxy
   \quad \refequal{\eqref{eq_eaebcross}} \quad
   \xy
 (0,0)*{\includegraphics[scale=0.5]{def-tsplitu2.eps}};
 (-7,-10)*{a};(7,-10)*{b};(-6,4)*{e_b};(6,4)*{e_a};
  \endxy
\end{align}

\begin{prop}[Associativity of splitters]

\begin{align} 
  \xy
 (0,0)*{\includegraphics[scale=0.5]{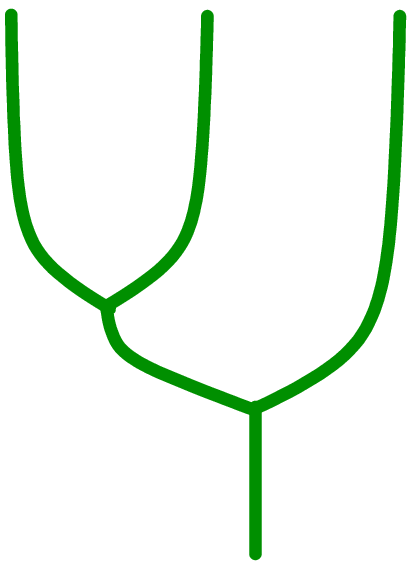}};
 (-5.5,-12)*{a+b+c};(-12,4)*{a};(-2.5,4)*{b};(12,4)*{c};
  \endxy
  \quad &= \quad (-1)^{ab\binom{c}{2}}
   \xy
 (0,0)*{\reflectbox{\includegraphics[scale=0.5]{uassoc.eps}}};
 (-10.5,-12)*{a+b+c};(-12,4)*{a};(-2.5,4)*{b};(12,4)*{c}; 
  \endxy \\
   \xy
 (0,0)*{\reflectbox{\includegraphics[scale=0.5, angle=180]{uassoc.eps}}};
 (-5.5,10)*{a+b+c};(-12,-4)*{a};(-2.5,-4)*{b};(12,-4)*{c};
  \endxy
  \quad &= \quad
 \xy
 (0,0)*{\includegraphics[scale=0.5,angle=180]{uassoc.eps}};
 (-10.5,10)*{a+b+c};(-12,-4)*{a};(-2.5,-4)*{b};(12,-4)*{c}; 
  \endxy \label{eq_split_assoc}
 \end{align}
\end{prop}

\begin{proof}
The proof of the first claim is a direct calculation making use of Proposition~\ref{prop_ea-standard}, equation~\ref{eq_crossing-combine},  and the identity
\[
     \xy
 (0,0)*{\includegraphics[scale=0.4]{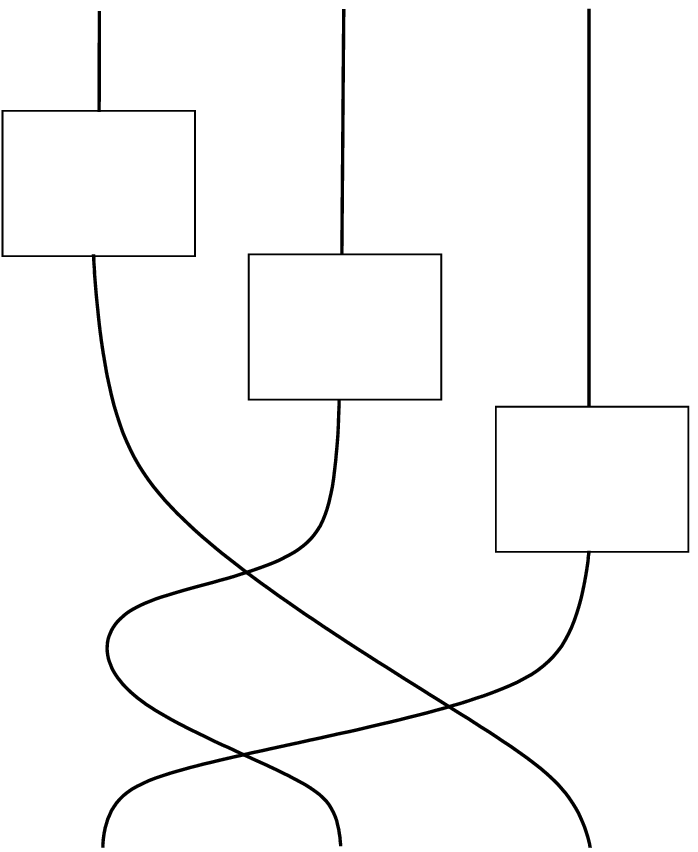}};
(-10,10)*{D_a};(0,4)*{D_b};(10,-2.5)*{D_c}; 
  \endxy  \;\; = \;\;
     \xy
 (0,0)*{\includegraphics[scale=0.4]{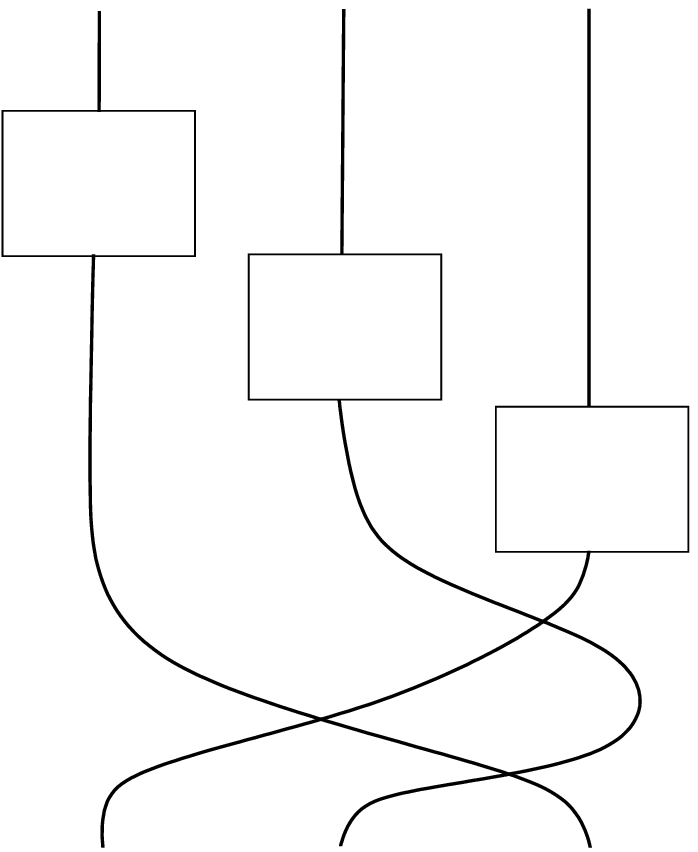}};
(-10,10)*{D_a};(0,4)*{D_b};(10,-2.5)*{D_c};
  \endxy
  \]
The second claim follows from equation~\eqref{eq_ea_absorb}.
\end{proof}

The apparent asymmetry between $a,c$ in the splitter associativity relation is a consequence of the asymmetry in our choice of crossing between multiple strands, equation \eqref{eqn-defn-crossing}.

\begin{prop}\label{prop_almostRthree}
For $a,b,c \geq 0$
\begin{align} \label{eq_triangle}
 \xy
 (0,0)*{\includegraphics[scale=0.5]{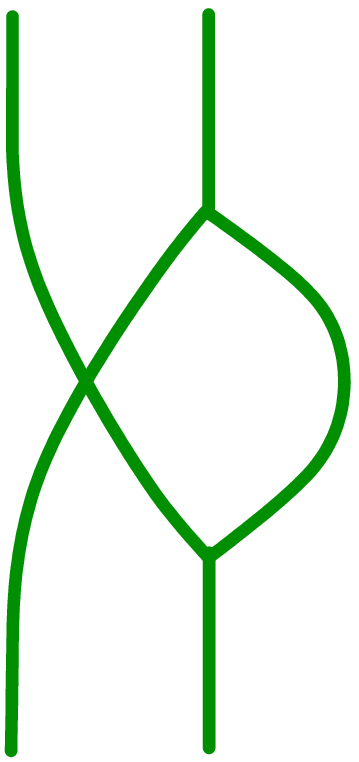}};
 (-10,-13)*{c};(8,-13)*{a+b};(-10,6)*{a};(11,4)*{b};
  \endxy
 \quad &= \quad
         \xy
 (0,0)*{\includegraphics[scale=0.5]{tH.eps}};
(-8,-11)*{c};(12,-11)*{a+b};(-8,11)*{a};(12,11)*{b+c};
  \endxy \\
 \xy
 (0,0)*{\reflectbox{\includegraphics[scale=0.5]{triangle.eps}}};
 (10,-13)*{c};(-8,-13)*{a+b};(10,6)*{a};(-11,4)*{b};
  \endxy
 \quad &= \quad (-1)^{\binom{a}{2}bc}
         \xy
 (0,0)*{\includegraphics[scale=0.5]{tH.eps}};
(8,-11)*{c};(-12,-11)*{a+b};(8,11)*{a};(-12,11)*{b+c};
  \endxy
\end{align}
\end{prop}

\begin{proof}
The proof uses equation~\ref{eq_crossing-combine} and equation~\eqref{eq_ea_absorb}.
\end{proof}

%
\subsubsection{Dotted thick strands}
%

\begin{defn}For any $f \in \osym_a$ write
 \begin{equation}
    \xy
 (0,0)*{\includegraphics[scale=0.5]{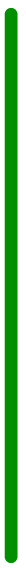}};
 (-2.5,-11)*{a};(0,-2)*{\bigb{f}};
  \endxy
  \quad := \quad e_a f e_a.
\end{equation}
\end{defn}

For any polynomial $f \in \opol_a$ observe that $D_a f D_a = D_a(f)D_a$ in $\ONH_a$. Now suppose $f$ is odd symmetric.  It follows that
\begin{equation} \label{eq_ea-to-delta}
e_a fe_a = \undx^{\delta_a} D_a f \undx^{\delta_a} D_a = \undx^{\delta_a} D_a(f \undx^{\delta_a}) D_a
= (-1)^{\binom{a}{3}} \undx^{\delta_a} f^{w_0} D_a.
\end{equation}
Then by the corollary to the OWL \eqref{eqn-Da-left-linearity},
\begin{equation} \label{eq_ea-to-unddelta}
 e_a fe_a = (-1)^{\binom{a}{3}} \undx^{\delta_a} f^{w_0} D_a=e_af.
\end{equation}
It follows that
\begin{equation}
e_agfe_a=(e_age_a)(e_afe_a),
\end{equation}
so we have the diagrammatic identity
\begin{equation} \label{eq_yz_greenbox}
  \xy
 (0,0)*{\includegraphics[scale=0.5]{tlong-up.eps}};
 (-2.5,-12)*{a};(0,2)*{\bigb{g}};(0,-5)*{\bigb{f}};
  \endxy
 \quad = \quad
   \xy
 (0,0)*{\includegraphics[scale=0.5]{tlong-up.eps}};
 (-2.5,-12)*{a};(0,-1)*{\bigb{gf}};
  \endxy
\end{equation}
whenever $f,g\in\osym_a$.

%
\subsection{Explosions}
%

Sometimes it is convenient to split thick edges into thin edges.
Define \textit{exploders}
\begin{equation} 
  \xy
 (0,0)*{\includegraphics[scale=0.5]{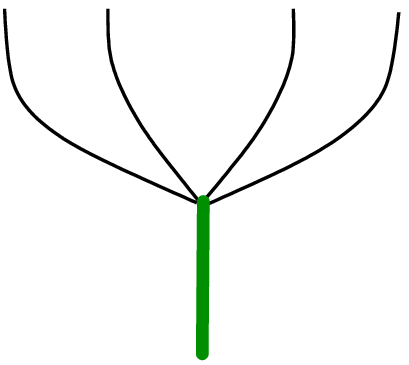}};
 (-3,-5)*{a}; (0,6)*{\cdots};
  \endxy
  \quad = \quad
  \xy
 (0,0)*{\includegraphics[scale=0.5]{1box-wide.eps}};
 (0,0)*{D_{a}};
  \endxy
  \qquad \qquad
   \xy
 (0,0)*{\includegraphics[scale=0.5,angle=180]{ufork-u.eps}};
 (-3,5)*{a}; (0,-6)*{\cdots};
  \endxy
  \quad = \quad
  \xy
 (0,0)*{\includegraphics[scale=0.5]{1box-wide.eps}};
 (0,0)*{e_{a}};
  \endxy
\end{equation}

The associativity rules for exploded splitters follow immediately from the associativity rules for general splitters derived above.  Equation~\eqref{eq_DaDbcross1} gives the following diagrammatic identities.
\begin{equation} \label{eq_Uasssplit1}
  \xy
 (0,0)*{\includegraphics[scale=0.45]{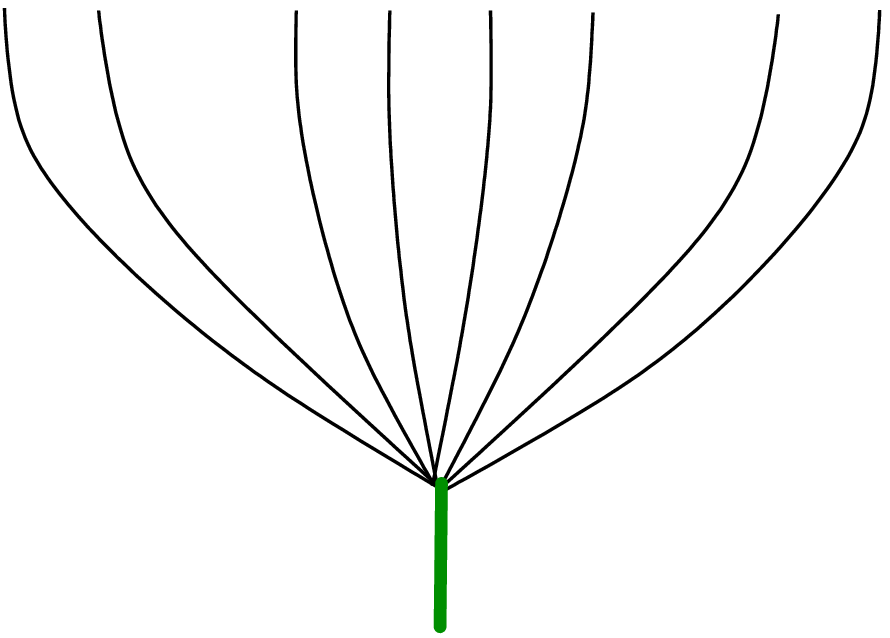}};
 (5,-13)*{a+b};(-10,11)*{\cdots}; (10,11)*{\cdots};
  \endxy
  \;\; = \;\;
  \xy
 (0,0)*{\includegraphics[scale=0.45]{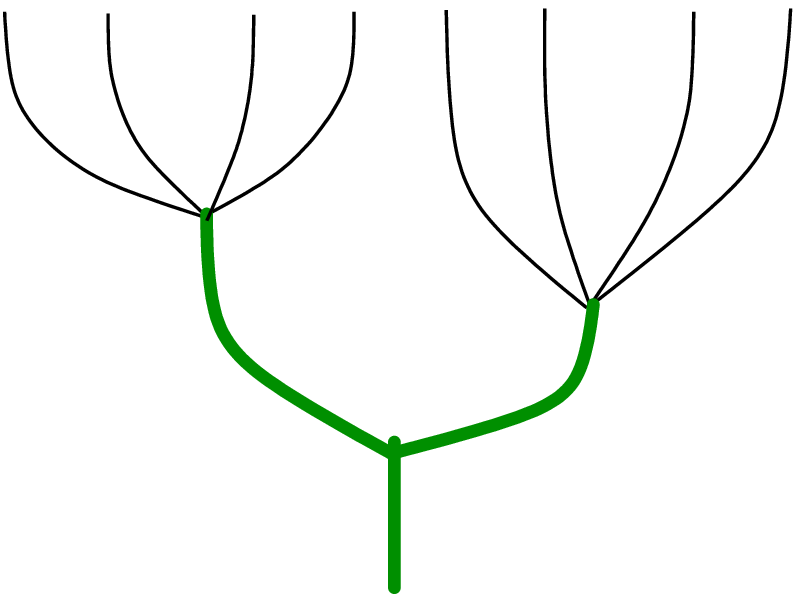}};
 (5,-13)*{a+b};(-8,-6)*{a};(10,-6)*{b}; (-10,11)*{\cdots}; (10,11)*{\cdots};
  \endxy
  \;\; = \;\; (-1)^{\binom{a}{2}\binom{b}{2}}
   \xy
 (0,0)*{\reflectbox{\includegraphics[scale=0.45]{ass-fullexplode1.eps}}};
 (5,-13)*{a+b};(-10,-6)*{a};(8,-6)*{b}; (-10,11)*{\cdots}; (10,11)*{\cdots};
  \endxy
\end{equation}

Using equation~\eqref{eq_eaebcross} the equation below follows.
\begin{equation} \label{eq_Uasssplit2}
  \xy
 (0,0)*{\includegraphics[scale=0.45]{ass-fullexplode2.eps}};
 (5,-13)*{a+b};(-10,11)*{\cdots}; (10,11)*{\cdots};
  \endxy
\;\; = \;\;
   \xy
 (0,0)*{\includegraphics[scale=0.45]{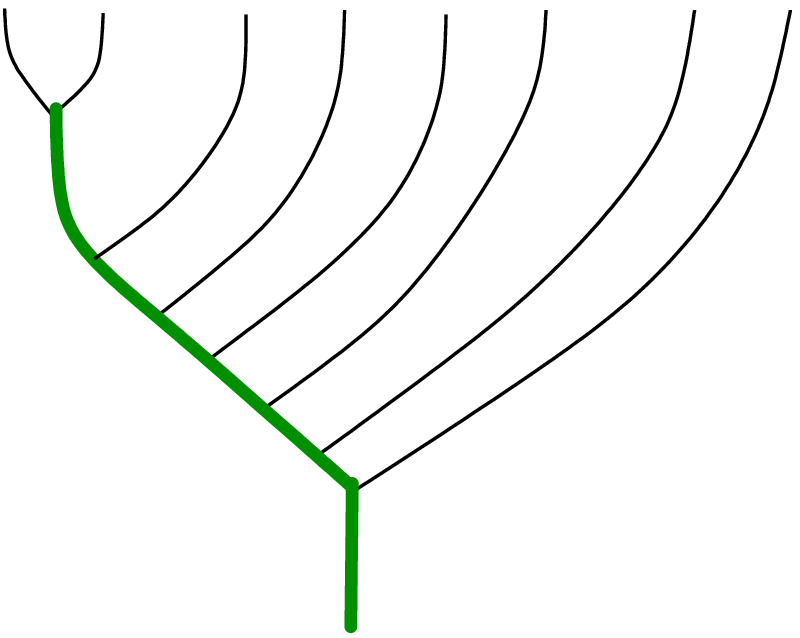}};
 (3,-13)*{a+b}; (-10,11)*{\cdots}; (10,11)*{\cdots};
  \endxy
\;\; =\;\;
   \xy
 (0,0)*{\reflectbox{\includegraphics[scale=0.45]{ass-fullexplode3.eps}}};
 (7,-13)*{a+b}; (-10,11)*{\cdots}; (10,11)*{\cdots};
  \endxy
\end{equation}

It is straightforward to show using the properties of the projectors $e_a$ that the following relations hold,
\begin{equation} \label{eq_Dasssplit1}
  \xy
 (0,0)*{\includegraphics[scale=0.45,angle=180]{ass-fullexplode2.eps}};
 (7,13)*{a+b};(-10,-11)*{\cdots}; (10,-11)*{\cdots};
  \endxy
  \;\; = \;\;
  \xy
 (0,0)*{\includegraphics[scale=0.45,angle=180]{ass-fullexplode1.eps}};
 (6,13)*{a+b};(-8,6)*{a};(10,6)*{b}; (-10,-11)*{\cdots}; (10,-11)*{\cdots};
  \endxy
  \;\; = \;\;
   \xy
 (0,0)*{\reflectbox{\includegraphics[scale=0.45,angle=180]{ass-fullexplode1.eps}}};
 (5,13)*{a+b};(-10,6)*{a};(8,6)*{b}; (-10,-11)*{\cdots}; (10,-11)*{\cdots};
  \endxy
\end{equation}
\begin{equation} \label{eq_Dasssplit2}
  \xy
 (0,0)*{\includegraphics[scale=0.45,angle=180]{ass-fullexplode2.eps}};
 (7,13)*{a+b};(-10,-11)*{\cdots}; (10,-11)*{\cdots};
  \endxy
\;\; = \;\;
   \xy
 (0,0)*{\includegraphics[scale=0.45,angle=180]{ass-fullexplode3.eps}};
 (7.5,13)*{a+b}; (-10,-11)*{\cdots}; (10,-11)*{\cdots};
  \endxy
\;\; =\;\;
   \xy
 (0,0)*{\reflectbox{\includegraphics[scale=0.45,angle=180]{ass-fullexplode3.eps}}};
 (4,13)*{a+b}; (-10,-11)*{\cdots}; (10,-11)*{\cdots};
  \endxy
\end{equation}

%
\subsubsection{Exploding Schur polynomials}
%

Recall that given a partition $\alpha= (\alpha_1, \dots, \alpha_a) \in P(a)$ we have defined the odd Schur polynomial $s_{\alpha}(x_1, \dots, x_a)$ as
\begin{equation}
 s_{\alpha}(x_1, \dots, x_a) :=  (-1)^{\binom{a}{3}} \left(D_a(\undx^{\alpha}\undx^{\delta_a}) \right)^{w_0}.
\end{equation}
It will be convenient to normal order the variables using the notation
\begin{equation}
  \undx^{\delta_a+\alpha} = x_1^{a-1+\alpha_1} x_2^{a-2+\alpha_2} \dots x_a^{\alpha_a}.
\end{equation}
One can show
\begin{equation}
 s_{\alpha}(x_1, \dots, x_a) :=  (-1)^{\chi_{\alpha}^a} \left(D_a(\undx^{\delta_a+\alpha}) \right)^{w_0},
\end{equation}
where for $\alpha \in P(a)$ we write
\begin{equation}
 \chi_{\alpha}^a := \binom{a}{3}+|\alpha|\binom{a}{2}+\sum_{j=1}^a \alpha_j\binom{a-j+1}{2}.
\end{equation}
That is, $(-1)^{\binom{a}{3}}\chi_{\alpha}^a$ is the sign needed to normal order the product $\undx^\alpha\undx^{\delta_a}$.  For example, if $\alpha = (1^r)$ then $\alpha_j=1$ for $1 \leq j \leq r$ and
\[
 \chi_{(1^r)}^a = \binom{a}{3}+r\binom{a}{2} + \frac{r(3a^2-3ar+r^2-1)}{6}.
\]

Observe that
 \begin{align}
 e_a s_{\alpha}e_a
  \;\; =\;\; (-1)^{\chi_{\alpha}^a}e_a D_a(\undx^{\alpha+\delta_a})^{w_0} e_a
 \;\; \refequal{\eqref{eq_ea-to-unddelta}} \;\; (-1)^{\binom{a}{3}+\chi_{\alpha}^a}\undx^{\delta_a} D_a(\undx^{\alpha+\delta_a}) D_a \\ \hspace{1in}
  \;\; =\;\; (-1)^{\binom{a}{3}+\chi_{\alpha}^a}
   \undx^{\delta_a} D_a \undx^{\alpha+\delta_a} D_a
   \; \;\refequal{\eqref{eq_ea_standard}} \;\;(-1)^{\chi_{\alpha}^a}
   e_a \undx^{\alpha + \delta_a} D_a.
\end{align}
This implies the diagrammatic identity
\begin{equation}
      \xy
 (0,0)*{\includegraphics[scale=0.5]{tlong-up.eps}};
 (-2.5,-11)*{a};(0,-2)*{\bigb{s_{\alpha}}};
  \endxy
  \quad = \quad (-1)^{\chi_{\alpha}^a} \;\;
\xy
 (0,0)*{\includegraphics[scale=0.5]{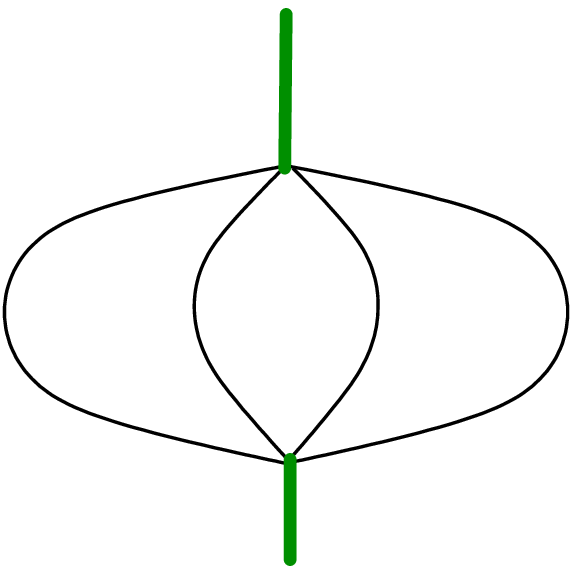}};
 (13.5,-4)*{\bullet}+(+2.5,1)*{\scs \alpha'_a};
 (4.5,-1)*{\bullet}+(+4.5,1)*{\scs \alpha'_{a-1}};
 (-4.5,1)*{\bullet}+(-2.5,1)*{\scs \alpha'_2};
 (-9.5,4)*{\bullet}+(-3,1)*{\scs\alpha'_{1}};
 (0,-2)*{\cdots}; 
  \endxy
\end{equation}
where $\alpha'_j = \alpha_j+a-j$.

Notice that writing such an equation would not be possible if we did not include the action of the longest symmetric group element $w_0$ in the definition of $s_{\alpha}$.

%
\subsection{Orthogonal idempotents}\label{subsec-orthogonal-idempotents}
%

%
\subsubsection{Some helpful lemmas}
%

\begin{lem}[Shuffle Lemma] \label{lem_right-slide}
\begin{equation}
  \partial_i(x_i^m x_{i+1}^{m+k}) =
  \left\{
   \begin{array}{cl}
     (-1)^m \partial_i(x_i^{m+k} x_{i+1}^m) & \text{if $k=1$,} \\
     -\partial_i(x_i^{m+k} x_{i+1}^m) & \text{if $k$ is even,} \\
     (-1)^{m} \partial_i(x_i^{m+k} x_{i+1}^m) - \partial_i(x_i^{m+k-1} x_{i+1}^{m+1}) +
     (-1)^{m} \partial_i(x_i^{m+1} x_{i+1}^{m+k-1} ) & \text{if $k$ is odd and $k \geq 3$.}
   \end{array}
  \right. \nn
\end{equation}
\end{lem}

\begin{proof}
Using the Leibniz rule for odd divided difference operators we have
\begin{equation}
  \partial_i(x_i^m x_{i+1}^{m+k}) = \partial_i(x_i^m x_{i+1}^{m})x_{i+1}^k
  + (x_i^m x_{i+1}^{m})^{s_i} \partial_i(x_{i+1}^k) = (x_{i+1}^m x_{i}^{m}) \partial_i(x_{i+1}^k).
\end{equation}
The sum $\xt_i^k+\xt_{i+1}^k$ is annihilated by the divided difference operator $\partial_i$
when $k=1$ and when $k$ is even.  When $k$ is odd and $k \geq 3$, the sum
\begin{equation}
  \xt_i^k+\xt_{i+1}^k + \xt_i^{k-1}\xt_{i+1} + \xt_i\xt_{i+1}^{k-1},
\end{equation}
is annihilated by $\partial_i$.  Using these facts the result follows.
\end{proof}

Note that for $k$ odd and $k\geq 3$ the relations above imply
\begin{align}
(-1)^{m} \partial_i(x_i^{m+k} x_{i+1}^m) - \partial_i(x_i^{m+k-1} x_{i+1}^{m+1})
+(-1)^{m} \partial_i(\xt_i^{m+1} x_{i+1}^{m+k-1} ) \nn \\
\qquad =
(-1)^{m} \partial_i(x_i^{m+k} x_{i+1}^m) - 2\sum_{j=1}^{\lfloor \frac{k}{2}\rfloor}
(-1)^{m(j+1)} \partial_i \left( x_i^{m+k-j}x_{i+1}^{m+j} \right)
\end{align}
so that
\begin{equation}
   \partial_i(x_i^m x_{i+1}^{m+k}) =
   (-1)^{m} \partial_i(x_i^{m+k} x_{i+1}^m) - 2\sum_{j=1}^{\lfloor \frac{k}{2}\rfloor}
(-1)^{m(j+1)} \partial_i \left( x_i^{m+k-j}x_{i+1}^{m+j} \right) .
\end{equation}
We call the relation above the  {\em big odd shuffle}.

\begin{prop}[Vanishing of matched exponent in a partial staircase] \label{prop_vanish}\hfill \newline
Let $a \geq m \geq 2$ and $a-(m-1) \leq p \leq a-1$.  Then
\begin{equation}
  D_m(x_1^{a-1}x_{2}^{a-2} \dots x_{m-2}^{a-(m-2)}x_{m-1}^{a-(m-1)} x_{m}^p ) = 0.
\end{equation}
\end{prop}

\begin{proof}
We prove the proposition by induction on $m$. If $m=2$ we are done since then $p=a-1$ and $\partial_1(x_1^{a-1}x_2^{a-1})=0$.   Assume by induction that
\[
D_{m-1}(x_1^{a-1}x_{2}^{a-2} \dots x_{m-2}^{a-(m-2)}x_{m-1}^{p'}) = 0
\]
for $a-(m-2) \leq p' \leq a-1$.

Suppose that $p=a-j$ for some $1 \leq j \leq m-1$ in the expression
\begin{equation}
  D_m(x_1^{a-1}x_{2}^{a-2} \dots x_{m-1}^{a-(m-1)} x_{m}^p ).
\end{equation}
If $p=a-(m-1)$ then we are done since $D_m= D_{m-1} (\partial_1 \dots \partial_{m-1})$ and $\partial_{m-1}(x_1^{a-1}\dots x_{m-1}^{a-(m-1)} x_{m}^{a-(m-1)})=0$.   Assume then that $a-(m-2) \leq p \leq a-1$  so that
\begin{align}
 D_m(x_1^{a-1}\dots x_{m-1}^{a-(m-1)} x_{m}^{p})
 &= D_{m-1} (\partial_1 \dots \partial_{m-1})(x_1^{a-1}\dots x_{m-1}^{a-(m-1)} x_{m}^{p})\nn\\ &=
 D_{m-1} (\partial_1 \dots \partial_{m-2})(x_1^{a-1}\dots x_{m-2}^{a-(m-2)} \partial_{m-1}(x_{m-1}^{a-(m-1)} x_{m}^{p}) ). \nn
\end{align}
Let $q=p-(a-(m-1))$.   Lemma~\ref{lem_right-slide} together with the induction hypothesis show that the expression vanishes whenever $q=1$ or $q$ is even.  When $q$ is odd and $q\geq 3$ we  write
\begin{align}
 D_m(x_1^{a-1}\dots x_{m-1}^{a-(m-1)} x_{m}^{p}) =
  (-1)^{a-(m-1)}D_m(x_1^{a-1}\dots x_{m-2}^{a-(m-2)} x_{m-1}^{p} x_{m}^{a-(m-1)})  \hspace{1in}\nn \\
 - D_m(x_1^{a-1}\dots x_{m-2}^{a-(m-2)} x_{m-1}^{p-1} x_{m}^{a-(m-2)})
 +
 (-1)^{a-(m-1)} D_m(x_1^{a-1}\dots x_{m-2}^{a-(m-2)} x_{m-1}^{a-(m-2)} x_{m}^{p-1}).  \nn
\end{align}
The induction hypothesis shows that each of these three terms is zero, unless $p=a-(m-2)$ in which case the induction hypothesis does not apply to the second term.  But in this case, the term still vanishes since the exponent of $x_m$ is also $a-(m-2)$.
\end{proof}

\begin{prop}[Adding a step to a full staircase] \label{prop_addstep}
The equation
\begin{equation}
  D_a(x_1^{a-2} x_2^{a-3} \dots x_{a-1}^0 x_a^{a-1}) =
 (-1)^{\binom{a-1}{3}} D_a(x_1^{a-1} x_2^{a-2} \dots x_{a-1}^1 x_a^{0})=(-1)^{\binom{a-1}{2}}
\end{equation}
holds in $\ONH_a$.
\end{prop}

\begin{proof}
The proof follows by shuffling the exponent of $x_a$ left using Lemma~\ref{lem_right-slide} and the big odd shuffle.  Additional terms that arise from the big odd shuffle vanish by Proposition~\ref{prop_vanish}.  In this calculation, the sign is $(-1)^{\lfloor \frac{a-2}{2}\rfloor}$ when $a$ is even, and $(a-1)$ when $a$ is odd.
\end{proof}

\begin{prop}[Reordering a reverse staircase] \label{prop_reorder-revstair}
\begin{equation}
  D_a(x_1^0 x_2^1 \dots x_{a-1}^{a-2} x_a^{a-1}) = (-1)^{\binom{a}{4}}
  D_a(x_1^{a-1} x_2^{a-2} \dots x_{a-1}^1 x_a^{0}).
\end{equation}
\end{prop}

\begin{proof}
The proposition follows by induction from Proposition~\ref{prop_addstep} together with some simplification of the exponent of $(-1)$.
\end{proof}

\begin{lem}
Let $\alpha\in P(a,b)$ and $\beta\in P(b,a)$ be two partitions. Then
\begin{equation}
  D_{a+b}\left(x_1^{a-1+\alpha_1} x_2^{a-2+\alpha_2} \dots x_{a-1}^{1+\alpha_{a-1}}x_a^{\alpha_a}
  x_{a+1}^{\beta_b} x_{a+2}^{1+\beta_{b-1}} \dots x_{a+b-1}^{b-2+\beta_2} x_{a+b}^{b-1+\beta_1}\right) = 0
\end{equation}
unless $|\alpha|+|\beta|=ab$.
\end{lem}

\begin{proof}
If  the total degree $\deg(\delta_a)+\deg(\delta_b)+|\alpha|+|\beta|$ of the dots  is less than the degree $\binom{a+b}{2}$ of $\delta_{a+b}$, then the total degree of the left-hand side of equation~\ref{eq_needtoprove} is negative, and the expression must therefore be zero.  If the degree is strictly greater than $\binom{a+b}{2}$, then using the Shuffle Lemma we can reorder the exponents so that they are decreasing, acquiring additional terms each time the big odd slide is performed.  The requirements on $\alpha$ and $\beta$ imply $b \geq \alpha_1 \geq \dots \geq \alpha_a$ and $a \geq \beta_1 \geq \dots \geq \beta_a \geq 1$.  Since the total degree of the dots is greater than $ \binom{a+b}{2}$ it follows that there must be a repeated exponent in each term that arises from the shuffling process.  Since the action of $D_{a+b}$ annihilates adjacent repeated exponents, all of these terms must vanish.
\end{proof}

\begin{lem} \label{lem_Dapb}
Let $\alpha\in P(a,b)$ and $\beta\in P(b,a)$. Then
\begin{equation} \label{eq_needtoprove}
  D_{a+b}\left(x_1^{a-1+\alpha_1} x_2^{a-2+\alpha_2} \dots x_{a-1}^{1+\alpha_{a-1}}x_a^{\alpha_a}
  x_{a+1}^{\beta_b} x_{a+2}^{1+\beta_{b-1}} \dots x_{a+b-1}^{b-2+\beta_2} x_{a+b}^{b-1+\beta_1}\right)
  = \delta_{\alpha,\hat{\beta}} (-1)^{\Omega(\beta)+\binom{a+b}{3}},
\end{equation}
where
\begin{equation}\label{eqn-Omega}
 \Omega(\beta) := \sum_{j=0}^{b-1} \binom{\beta_{b-j}+j}{3}.
\end{equation}
\end{lem}

\begin{proof}
By the previous lemma it suffices to consider the case when $|\alpha|+|\beta|=ab$.  The condition that $\alpha= \hat{\beta}$ implies that the only time \eqref{eq_needtoprove} is nonzero is when each of the exponents is unique, so that all elements of the set $\{0,1, \dots , a+b-1\}$ occur as exponents.  Note that the partition requirements for $\alpha$ and $\beta$ imply that there are chains of strict inequalities of exponents $a-1+\alpha_1 > a-2 + \alpha_2 > \dots > \alpha_a$ with the first $a$ variables and $\beta_b < \beta_{b-1}+1 < \dots < \beta_1 + b-1$ with the last $b$ variables.

We show that if an exponent is ever repeated in the left hand side of \eqref{eq_needtoprove} then the expression vanishes.  We will prove this in the case $\alpha_a <\beta_b$ (the case $\alpha_a>\beta_b$ is similar; and if $\alpha_a = \beta_b$, then the expression contains repeated adjacent exponents and is therefore zero). In this case, the exponents of $x_{a-{\beta_b}-1}$ through $x_{a}$ must form a staircase
\[
 x_{a-\beta_b-1}^{\beta_b-1} x_{a-\beta_b+1}^{\beta_b-2} \dots x_{a-1}^{1}x_{a}^{0},
\]
or else the expression vanishes. To see this suppose that $x_{a-j}$ for $0 \leq j \leq \beta_b-1$ is the first $j$ where the exponent $f$ of $x_{a-j}^{f}$ is not part of a reverse staircase, that is, $f>j$.

Since $\beta_b>f$ all of the exponents of the last $b$ strands must be larger than $f$, and all the variables before $x_{a-j}$ must also be larger than $f$.  Hence the exponent $j$ does not occur in the expression
\[
D_{a+b}\left(x_1^{a-1+\alpha_1}\dots
x_{a-j-1}^{j+1+\alpha_{a-(j+1)}} x_{a-j}^{f}
x_{a-(j-1)}^{(j-1)} \dots x_{a-1}^{1}x_{a}^{0}x_{a+1}^{\beta_b}
 \dots  x_{a+b}^{b-1+\beta_1}\right).
\]
Using the Shuffle Lemma and the big odd slide we reorder the exponents above so that they are decreasing.  Each of the resulting terms will have one element of the set $\{0, 1, \dots, j-1\}$ missing since the shuffling procedure takes pairs of exponents $(t_1,t_1+k)$ with $k >0$ and either shuffles them $(t_1+k,t_1)$ or else creates terms $(t_1+k-\ell,t_1+\ell)$ for $1 \leq \ell \leq \lfloor\frac{k}{2}\rfloor$.  Observe that if a missing exponent appears from a big odd slide $(t_1+k-\ell,t_1+\ell)$, then the lower exponent $t_1$ has been removed.  If $t_1$ then appears from a subsequent big odd slide, then again a lower exponent will have had to be removed.  Hence, at least one element in the set $\{0, 1, \dots, j-1\}$ is missing in each term arising in the shuffling procedure.  For degree reasons, if one exponent in the set $\{ 0,1, \dots, a+b-1\}$ is missing, then at least one exponent must be repeated, hence the expression contains adjacent repeated exponents and must therefore  vanish.

Thus, if our expression is nonvanishing, it must be of the form
\[
D_{a+b}\left(x_1^{a-1+\alpha_1}\dots
x_{a-\beta_b}^{\beta_b+\alpha_{(a-\beta_b)}} x_{a-(\beta_b-1)}^{\beta_b-1}
x_{a-(\beta_b-2)}^{\beta_b-2} \dots x_{a-1}^{1}x_{a}^{0}x_{a+1}^{\beta_b}
 \dots  x_{a+b}^{b-1+\beta_1}\right).
\]
Now use Proposition~\ref{prop_addstep} to slide the exponent $\beta_b$ of $x_{a+1}$ left to add a step to the staircase giving the expression
\[
(-1)^{\Omega(\beta_b)}D_{a+b}\left(x_1^{a-1+\alpha_1}\dots
x_{a-\beta_b}^{\beta_b+\alpha_{(a-\beta_b)}} x_{a-(\beta_b-1)}^{\beta_b}
x_{a-(\beta_b-2)}^{\beta_b-1} \dots x_{a-1}^{2}x_{a}^{1}x_{a+1}^{0} x_{a+2}^{1+\beta_{(b-1)}}
 \dots  x_{a+b}^{b-1+\beta_1}\right).
\]

If $\beta_{b-1}=\beta_b$, then slide the exponent $\beta_{b-1}+1 = \beta_b+1$ of $x_{a+2}$ left adding a step to the staircase with sign $(-1)^{\Omega(\beta_{b-1}+1)}$. If this exponent $\beta_b+1$ is repeated then it must be repeated as the exponent $\beta_b+\alpha_{(a-\beta_b)}$ of $x_{a-\beta_b}$, then the expression is zero.  Otherwise we can assume that $\beta_{b-1}+1=\beta_{b}+1$ is not a repeated exponent and that $\beta_b+\alpha_{(a-\beta_b)}> \beta_{b-1}+1$.

If $\beta_{b-1} > \beta_{b}$ so that $\beta_{b-1}=\beta_b+g+1$ for some $g\geq0$, then the staircase must continue
\[
x_{a-(\beta_b+g)}^{\beta_b+g}
\dots
x_{a-(\beta_b+1)}^{\beta_b+2}
x_{a-\beta_b}^{\beta_b+1}
x_{a-(\beta_b-1)}^{\beta_b}
x_{a-(\beta_b-2)}^{\beta_b-1} \dots x_{a}^{1}x_{a+1}^{0}
\]
or else the expression vanishes.  If the staircase did not continue, then one of the exponents in the set $\{\beta_b+1, \beta_b+2 \dots, \beta_b+g\}$ would not occur in the expression, and arguing as above one can show that all terms resulting from shuffling the expression to decreasing order must be missing at least one exponent in the set $\{0,1,\dots,\beta_b+g\}$ and must therefore vanish.

Continuing in this way, it follows that if any exponent in the expression is repeated then the expression is zero. Otherwise, all the exponents of the last $b$ variables can be slid through staircases so that the resulting expression has strictly decreasing exponents.  Since $D_{a+b}(\undx^{a+b})=(-1)^{\binom{a+b}{3}}$ the result follows.
\end{proof}

\begin{defn}
Define {\em dual Schur functions} as
$$\hat{s}_{\alpha}(x_1,\dots, x_a):=
(-1)^{\chi_{\alpha}^a} D_a( \sigma\psi( \undx^{\delta_a+\alpha}))^{w_0}
=
(-1)^{\chi_{\alpha}^a} D_a(x_1^{\alpha_a} x_2^{1+\alpha_{a-1}} \dots x_{a}^{a-1+\alpha_1})^{w_0}.$$
\end{defn}

It easy to see that
$e_a \hat{s}_{\alpha} e_a = (-1)^{\chi_{\alpha}^a}e_a\sigma\psi( \undx^{\delta_a+\alpha}) D_a = (-1)^{\chi_{\alpha}^a} e_a x_1^{\alpha_a} x_2^{1+\alpha_{a-1}} \dots x_{a}^{a-1+\alpha_1} D_a$.

\begin{prop} \label{prop_oval_small}
Let $\alpha\in P(a,b)$ and $\beta\in P(b,a)$ be two partitions. Then
\begin{equation} \label{eq_cor_oval_small}
  \xy
 (0,0)*{\includegraphics[scale=0.5]{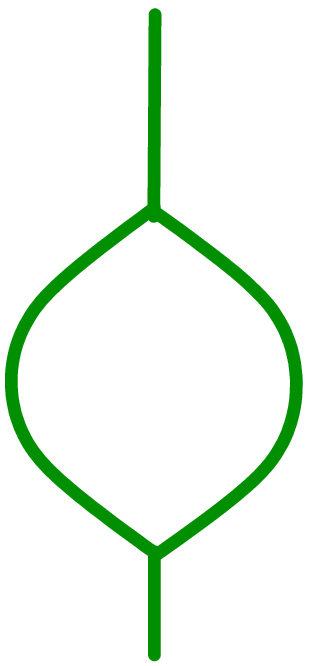}};
 (-8,-8)*{a};  (6,5)*{b};(-6,-14)*{a+b};
 (-6,2)*{\bigb{s_{\alpha}}};(6,-5)*{\bigb{\hat{s}_{\beta}}};
  \endxy
\;\; = \;\;
    \left\{
\begin{array}{ccc}
  (-1)^{\chi_{\alpha}^a+\chi_{\hat{\alpha}}^b+\binom{a}{2}\left(|\hat{\alpha}|
  +\binom{b}{2}\right)+\Omega(\hat{\alpha})+\binom{a+b}{3}} \left(\;\;\xy
 (0,0)*{\includegraphics[scale=0.5]{single-tup.eps}};
 (-6 ,-6)*{a+b};
  \endxy\;\; \right) & \quad & \text{if $\beta=\hat{\alpha}$,} \\ & & \\
  0 & \quad & \text{otherwise,}
\end{array}
  \right.
\end{equation}
 where $\Omega(\beta)$ is as in Lemma~\ref{lem_Dapb} and $\hat{\alpha}$ is as in Subsection \ref{sec_oschur}.
\end{prop}

\begin{proof}
From the definitions we have
\begin{equation} \label{eq_proof_explode}
  \xy
 (0,0)*{\includegraphics[scale=0.5]{tline-bubble.eps}};
 (-8,-8)*{a};  (6,5)*{b};(-6,-14)*{a+b};
 (-6,2)*{\bigb{s_{\alpha}}};(6,-5)*{\bigb{\hat{s}_{\beta}}};
  \endxy\;\;=\;\; (-1)^{\chi_{\alpha}^a +\chi_{\beta}^b}
     \xy
 (0,0)*{\includegraphics[scale=0.5]{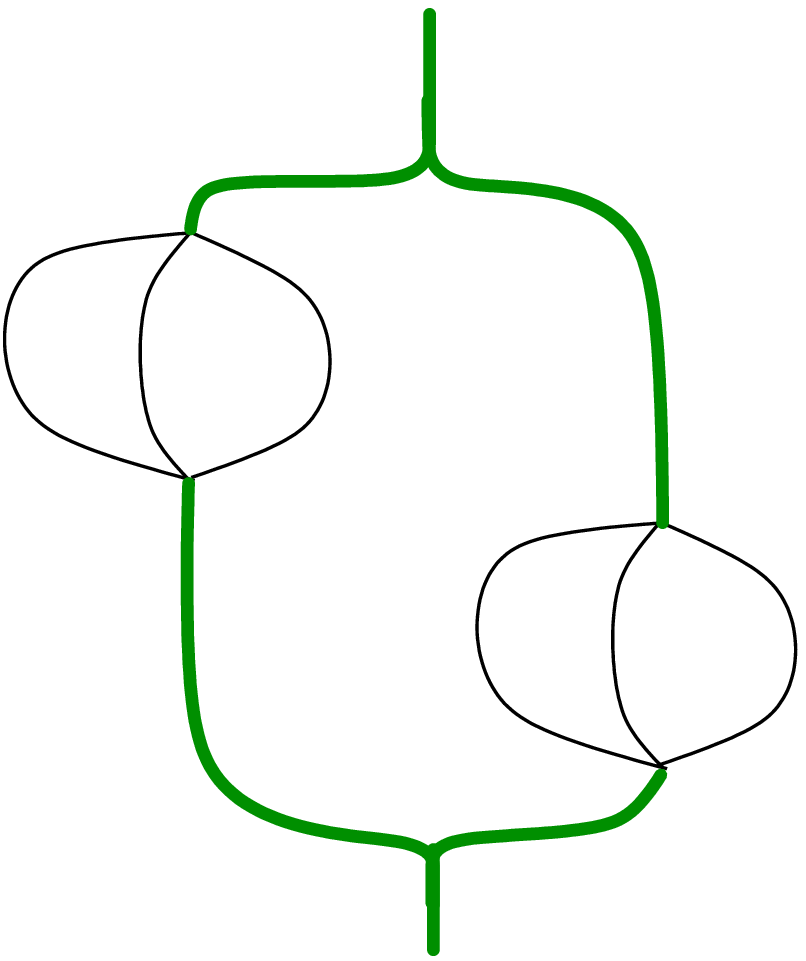}};
 (-6,-22)*{a+b};
 (-20.5,9)*{\bullet}+(-2.5,1)*{\scs s_1};
 (-13,7)*{\bullet}+(-2.5,1)*{\scs s_2};
 (-4,5)*{\bullet}+(-2.5,1)*{\scs s_a};
 (4,-6)*{\bullet}+(-2.5,1)*{\scs r_b};
  (11,-8)*{\bullet}+(-3,1)*{\scs r_{b-1}};
  (19.5,-10)*{\bullet}+(-2.5,1)*{\scs r_1};
 (15,-6)*{\cdots};(-9,9)*{\cdots};
  \endxy
 \end{equation}
where $s_j = a-j+\alpha_j$ and $r_j = b-j+\beta_j$.  After sliding the splitters past dot terms and using associativity for exploded splitters in equations~\eqref{eq_Uasssplit1} and  \eqref{eq_Dasssplit1}, the left-hand side of the above equation becomes
\begin{equation}
  \;\; = \;\; (-1)^{\chi_{\alpha}^a +\chi_{\beta}^b}(-1)^{\binom{a}{2}\left(|\beta|+\binom{b}{2}\right)}
       \xy
 (0,0)*{\includegraphics[scale=0.5]{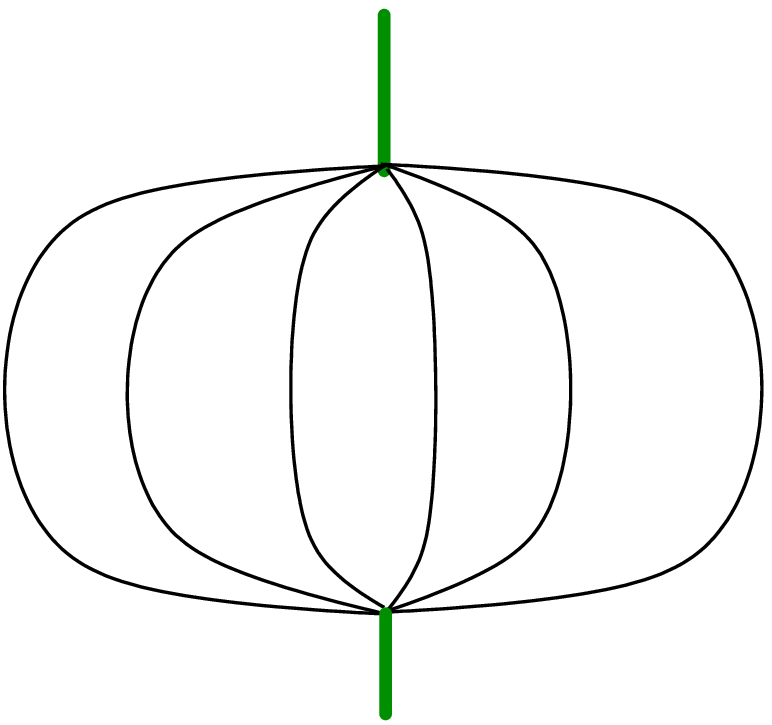}};
 (-7,-16)*{a+b};
 (-18,5)*{\bullet}+(-2.5,2)*{\scs s_1};
 (-12,3)*{\bullet}+(-2.,2)*{\scs s_2};
 (-4.5,1)*{\bullet}+(-2.5,2)*{\scs s_a};
 (2.5,-1)*{\bullet}+(-2.5,2)*{\scs r_b};
  (9.5,-3)*{\bullet}+(-2.5,2)*{\scs r_{b-1}};
  (18,-5)*{\bullet}+(-2.5,2)*{\scs r_1};
 (15,4)*{\cdots};(-9,-4)*{\cdots};
  \endxy.
\nn
\end{equation}
 The result follows by Lemma~\ref{lem_Dapb}.
\end{proof}

\begin{rem}
The ordering of the $x$'s in the above equation is critical for the above to work.  Notice that this is a different ordering from what appears in~\cite{KLMS}.  To see why it is necessary consider the example with $a=2$, $b=2$, $\alpha=(2,0)$, and $\beta=(2,0)$. In this case,  $D_4(x_1^3x_3^3) \neq 0$, while $D_4(x_1^3x_4^3)=0$.
\end{rem}

%
\subsection{The odd nilHecke algebra as a matrix algebra}
%

In this section we find it convenient to work with odd elementary symmetric functions $\varepsilon_k$ rather than their Schur analogues $s_{(1^k)}$.  Recall from \eqref{eq_schur_elem} that the two are related by $s_{(1^k)}=(-1)^{\binom{k}{2}}\varepsilon_k$.

Let $\Sq(a)$ denote the set of sequences
\begin{equation}
  \Sq(a) := \{
  \und{\ell} = \ell_1 \dots \ell_{a-1} \mid 0 \leq \ell_{\nu} \leq \nu, \nu =  1,2, \dots, a-1
  \}.
\end{equation}
This set has size $|\Sq(a)|=a!$.
For $\und{\ell} \in  \Sq(a)$ let
\begin{equation}
 \sigma_{\und{\ell}} \;\; = \;\;
   \vcenter{\xy
   (-2,-19)*{\includegraphics[scale=0.5]{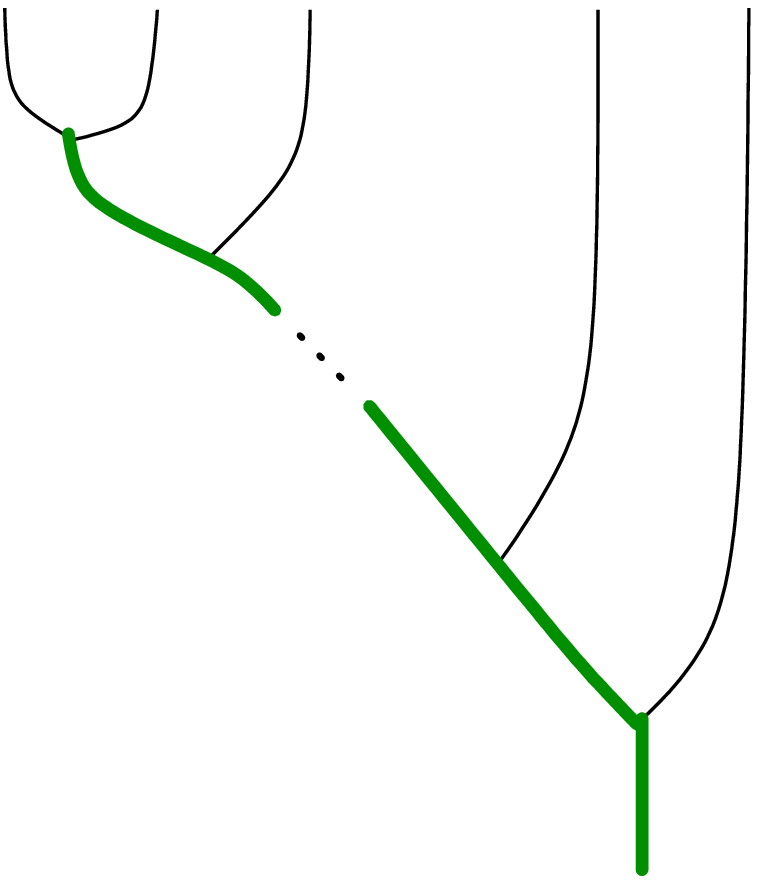}};
   (-16,-8)*{ \bigb{ \varepsilon_{\ell_{2}}}};
    (-1,-21)*{ \bigb{ \varepsilon_{\ell_{a-2}}}};
    (6,-29)*{ \bigb{ \varepsilon_{\ell_{a-1}}}};
    (-21,0)*{\bullet}+(-2.5,1.5)*{\scs \ell_{1}};
    (8,-39)*{a};
  \endxy}
  \qquad  \quad \lambda_{\und{\ell}} \;\; = \;\; (-1)^{\binom{a}{3}}
  \xy
 (0,0)*{\includegraphics[angle=180,scale=0.7]{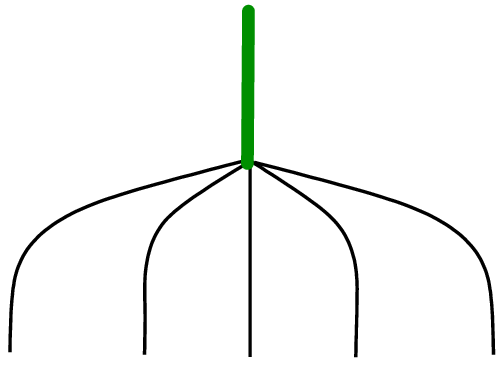}};
 (-7.5,-10)*{\bullet}+(-3,-1.5)*{\scs \hat{\ell}_{1}};;
 (0,-8)*{\bullet}+(2.5,-1.5)*{\scs \hat{\ell}_{r}};;
 (7.5,-6)*{\bullet}+(4,-1.5)*{\scs \hat{\ell}_{a-2}};;
 (15,-4)*{\bullet}+(4.5,-0)*{\scs \hat{\ell}_{a-1}};
 (3,-3)*{\cdots}; (-3.3,-3)*{\cdots}; (-3,8)*{a};
  \endxy
\end{equation}

\begin{lem} \label{lem_nil_orth}
\begin{equation}
  \lambda_{\und{\ell'}}\sigma_{\und{\ell}} \;\; = \;\;
  \delta_{\underline{\ell},\underline{\ell'}}
 \xy
  (0,0)*{\includegraphics[scale=0.5]{single-tup.eps}};
  (-2.5,-6)*{a};
 \endxy
 \;\; = \;\;
  \delta_{\und{\ell},\und{\ell'}} e_a
\end{equation}
for $\und{\ell}$, $\und{\ell'} \in \Sq(a)$.
\end{lem}

\begin{proof}
Consider the composite
\begin{equation}
 \lambda_{\und{\ell'}}\sigma_{\und{\ell}}
  \;\; = \;\; (-1)^{\binom{a}{3}} \;
  \vcenter{\xy
   (2,3)*{\includegraphics[scale=0.5]{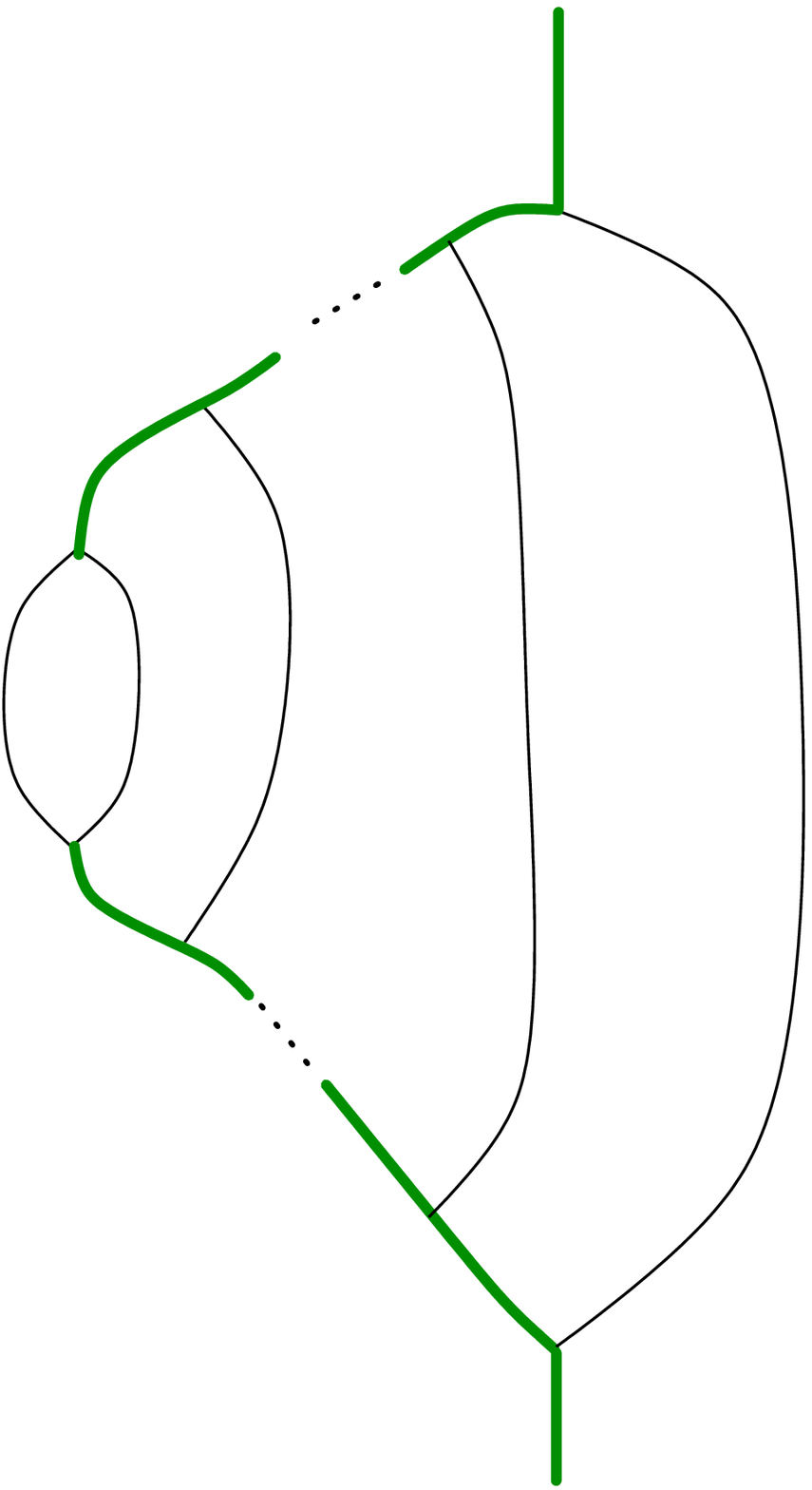}};
   (-16,-8)*{ \bigb{ \varepsilon_{\ell_{2}}}};
    (-1,-21)*{ \bigb{ \varepsilon_{\ell_{a-2}}}};
    (6,-29)*{ \bigb{ \varepsilon_{\ell_{a-1}}}};
    (-21,0)*{\bullet}+(-2.5,1.5)*{\scs \ell_{1}};
    (-14.5,3)*{\bullet}+(4,1.5)*{\scs 1-\ell_{1}'};
    (-5,6)*{\bullet}+(4,1.5)*{\scs 2-\ell_{2}'};
    (9.5,9)*{\bullet}+(7,1.5)*{\scs a-2-\ell_{a-2}'};
    (25.5,12)*{\bullet}+(7.5,1.5)*{\scs a-1-\ell_{a-1}'};
    (8,-39)*{a};
  \endxy}
\end{equation}
Repeatedly apply the equality
\begin{equation} \label{eq_oval_rthin}
  \xy
  (0,0)*{\reflectbox{\includegraphics[scale=0.5]{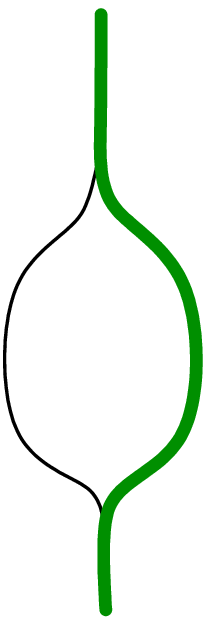}}};
 (-5,-13)*{\nu+1};
  (4,2)*{\bullet}+(6,1)*{\nu-\ell'_{\nu}};
  (-4,-5)*{ \bigb{ \varepsilon_{\ell_{\nu}}}};
  \endxy \quad = \quad \delta_{\ell_{\nu},\ell'_{\nu}}(-1)^{\binom{\ell_{\nu}}{2}+X_{(1^{\ell_{\nu}})}^{\nu,1}}
  \xy
   (0,0)*{\includegraphics[scale=0.5]{tlong-up.eps}};
    (-5,-12)*{\nu+1};
  \endxy
 \quad = \quad \delta_{\ell_{\nu},\ell'_{\nu}}(-1)^{\binom{\nu}{2}}
  \xy
   (0,0)*{\includegraphics[scale=0.5]{tlong-up.eps}};
    (-5,-12)*{\nu+1};
  \endxy
\end{equation}
for $1 \leq \nu \leq a-1$, where the first equality follows from equation \eqref{eq_schur_elem} and Proposition~\ref{prop_oval_small}.  The lemma follows since $\sum_{\nu=1}^{a-1} \binom{\nu}{2} = \binom{a}{3}$.
\end{proof}

Using the definitions and equation \eqref{eqn-Da-left-linearity} one can show that for $0 \leq k \leq a$ the diagrammatic identity
  \begin{equation} \label{eq_e_slide}
   \xy
 (0,0)*{\reflectbox{\includegraphics[scale=0.5,angle=180]{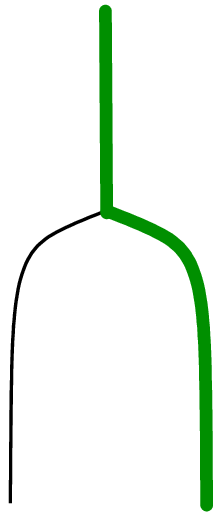}}};
 (-3,11)*{c};(-11,-11)*{c-1};(8,8)*{};
  (0,7)*{\bigb{\varepsilon_k}};
  \endxy
  \quad =
 \quad
    \xy
 (0,0)*{\reflectbox{\includegraphics[scale=0.5,angle=180]{tonesplit.eps}}};
 (-3,11)*{c};(-11,-11)*{c-1};(8,8)*{};
 (-5,-2)*{\bigb{\varepsilon_{k}}};
  \endxy
  \quad + \quad (-1)^{c-1}
      \xy
 (0,0)*{\reflectbox{\includegraphics[scale=0.5,angle=180]{tonesplit.eps}}};
 (-3,11)*{c};(-11,-11)*{c-1};(8,-8)*{};
 (-5,-2)*{\bigb{\varepsilon_{k-1}}};(5,-6)*{\bullet};
  \endxy
\end{equation}
holds, where the first diagram on the right-hand side is zero by convention when $k=c$.

\begin{lem} \label{lem_EaEone}
\begin{equation}
  \xy
 (4,0)*{\includegraphics[scale=0.5]{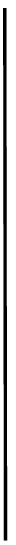}};
 (-4,0)*{\includegraphics[scale=0.5]{tlong-up.eps}};
 (-7 ,-12)*{a};
 (7 ,-12)*{\scs 1};
  \endxy
 \quad = \quad (-1)^{\binom{a}{2}}\sum_{s=0}^a   \xy
 (0,0)*{\includegraphics[scale=0.4]{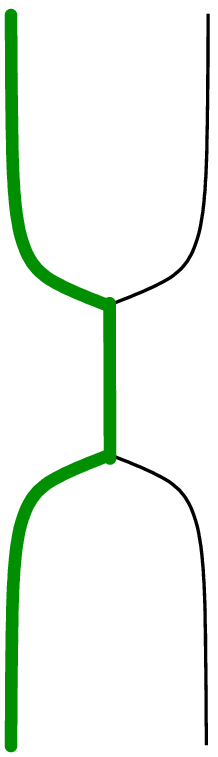}};
 (-7,-13)*{a};(8,-13)*{\scs 1};(-7,13)*{a};(7,13)*{\scs 1};
 (-5,7)*{\bigb{\varepsilon_{a-s} }};
 (4,-9)*{\bullet}+(3,1)*{s};
  \endxy
\end{equation}
\end{lem}

\begin{proof} A direct computation gives
\begin{equation}\label{eq_lemEaEb}
    \xy
 (4,0)*{\includegraphics[scale=0.5]{long-up.eps}};
 (-4,0)*{\includegraphics[scale=0.5]{tlong-up.eps}};
 (-7 ,-12)*{a};
 (7 ,-12)*{\scs 1};
  \endxy
\; \refequal{\eqref{eq_oval_rthin}} \;
(-1)^{\binom{a-1}{2}}\;\;
  \xy
  (-14,0)*{\reflectbox{\includegraphics[scale=0.45]{split-thinthick2.eps}}};
 (-16.5,-12)*{a};(-19,-2)*{\bigb{\varepsilon_{a-1}}}; (-22,0)*{};
 (-4,0)*{\includegraphics[scale=0.5]{long-up.eps}}; (0 ,-12)*{\scs 1};
  \endxy
 \refequal{\eqref{new_eq_iislide}} (-1)^{\binom{a-1}{2}} \left(
     \xy
 (-5,0)*{\reflectbox{\includegraphics[scale=0.4]{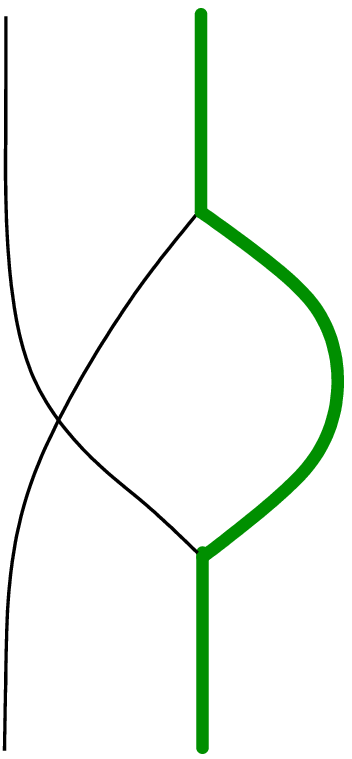}}};
 (-9,-13)*{a};(4,-13)*{\scs 1};
 (-11,4)*{\bigb{\varepsilon_{a-1}}}; (-22,0)*{};
 (-2,1)*{\bullet};
  \endxy
\;\; +
      \xy
 (-5,0)*{\reflectbox{\includegraphics[scale=0.4]{triangle-4.eps}}};
 (-9,-13)*{a};(4,-13)*{\scs 1};
 (-11,3)*{\bigb{\varepsilon_{a-1}}}; (-22,0)*{};
 (1,-5)*{\bullet};
  \endxy \;\; \right)
\end{equation}
By \eqref{eq_e_slide} the odd elementary symmetric function can be slid to the top of the diagram.
\begin{equation}= \quad
(-1)^{\binom{a-1}{2}}  \left((-1)^{(a-1)}
\xy
 (-5,0)*{\reflectbox{\includegraphics[scale=0.4]{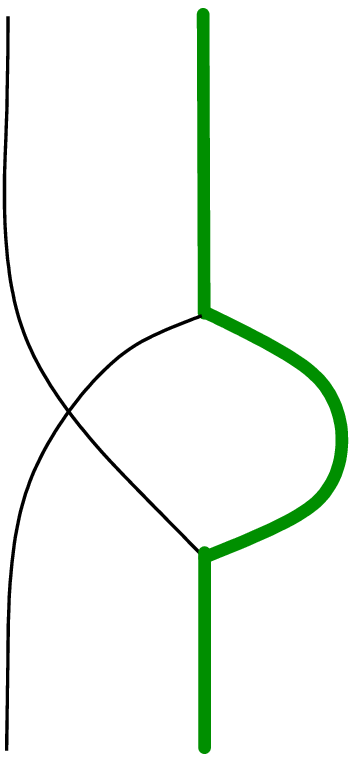}}};
 (-10,-13)*{a};(4,-13)*{\scs 1};
 (-6,6)*{\bigb{\varepsilon_{a}}};
  \endxy
   \;\;  + \quad
 \sum_{s=0}^{a-1}(-1)^{as}    \xy
 (-5,0)*{\reflectbox{\includegraphics[scale=0.4]{triangle-3.eps}}};
 (-10,-13)*{a};(4,-13)*{\scs 1};
 (-6,6)*{\bigb{\varepsilon_{a-1-s}}};
 (1,-5)*{\bullet};  (-3,1)*{\bullet}+(2,1)*{s};
  \endxy \;\;\right)
\end{equation}
Sliding the $s$ dots down on the second term using
the odd nilHecke relation and the fact that
\begin{equation} \label{eq_split_right_thin}
  \xy
  (-14,0)*{\reflectbox{\includegraphics[scale=0.45]{split-thinthick2.eps}}};
 (-16.5,-11)*{a};(-9.5,-2)*{\bullet}+(2,1)*{x};
  \endxy \quad = \quad 0, \qquad \quad \text{if $x< a-1$}
\end{equation}
 gives
\begin{equation}
\eqref{eq_lemEaEb} \quad = \quad
(-1)^{\binom{a-1}{2}}  \left((-1)^{(a-1)}\xy
 (-5,0)*{\reflectbox{\includegraphics[scale=0.4]{triangle-3.eps}}};
 (-8,-13)*{a};(4,-13)*{\scs 1};
 (-6,6)*{\bigb{\varepsilon_{a}}};
  \endxy
   \;\;  + \quad   \sum_{s=0}^{a-1}(-1)^{as+s}   \xy
 (-5,0)*{\reflectbox{\includegraphics[scale=0.4]{triangle-3.eps}}};
 (-8,-13)*{a};(4,-13)*{\scs 1};
 (-6,6)*{\bigb{\varepsilon_{a-1-s}}};
 (1,-5)*{\bullet}+(5,1)*{s+1};
  \endxy\right).
\end{equation}
In the last diagram on the right we can slide the $s+1$ dots down past the degree $a-1$ splitter giving
\[
\xy
 (-5,0)*{\reflectbox{\includegraphics[scale=0.4]{triangle-3.eps}}};
 (-8,-13)*{a};(4,-13)*{\scs 1};
 (-6,6)*{\bigb{\varepsilon_{a-1-s}}};
 (1,-5)*{\bullet}+(5,1)*{s+1};
  \endxy \;\; =\;\; (-1)^{(a-1)(s+1)}
  \xy
 (-5,0)*{\reflectbox{\includegraphics[scale=0.4]{triangle-3.eps}}};
 (-8,-13)*{a};(4,-13)*{\scs 1};
 (-6,6)*{\bigb{\varepsilon_{a-1-s}}};
 (2,-10)*{\bullet}+(5,1)*{s+1};
  \endxy
  \;\; =\;\; (-1)^{(a-1)(s+1)}
  \xy
 (0,0)*{\includegraphics[scale=0.4]{tH-one.eps}};
 (-7,-13)*{a};(8,-13)*{\scs 1};(-7,13)*{a};(7,13)*{\scs 1};
 (-5,7)*{\bigb{\varepsilon_{a-1-s} }};
 (4,-9)*{\bullet}+(5,1)*{s+1};
  \endxy
\]
where we used Proposition~\ref{prop_almostRthree} in the last equality. Shifting the summation shows that both terms in the parenthesis come with sign $(a-1)$.  Since $(a-1)+\binom{a-1}{2} = \binom{a}{2}$ the result follows.
\end{proof}

\begin{thm} \label{thm_nil_matrix}
Let $e_{\und{\ell}} = \sigma_{\und{\ell}}\lambda_{\und{\ell}}$. The set $\{e_{\und{\ell}}\}_{\und{\ell} \in \Sq(a)}$ consists of
mutually orthogonal idempotents that add up to $1\in \ONH_a$:
\begin{equation}
  e_{\und{\ell}}e_{\und{\ell'}} = \delta_{\und{\ell},\und{\ell'}} e_{\und{\ell}},
  \qquad \sum_{\und{\ell} \in \Sq(a)} e_{\und{\ell}} =1.
\end{equation}
\end{thm}

\begin{proof}
Lemma~\ref{lem_nil_orth} shows that $e_{\und{\ell}}$ are orthogonal idempotents.  To see that they decompose the identity we proceed by induction, the base case being trivial.
\begin{equation} \label{eq_thm_nil}
    \ONH_{a+1} \ni 1 \;\; = \;\; \vcenter{\xy
 (-6,-5)*{\includegraphics[scale=0.5]{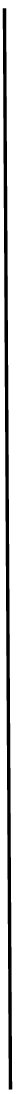}};
 (0,-5)*{\includegraphics[scale=0.5]{single-up-long2.eps}};
 (12,-5)*{\includegraphics[scale=0.5]{single-up-long2.eps}};
 (20,-5)*{\includegraphics[scale=0.5]{single-up-long2.eps}};
 (6,0)*{\cdots};
  \endxy}
  \;\; = \;\; (-1)^{\binom{a}{3}}\sum_{\und{\ell} \in \Sq(a)}
  \vcenter{\xy
  (-15,6)*{\xy
   (-2,-19)*{\includegraphics[scale=0.5]{AA2.eps}};
   (-16,-8)*{ \bigb{ \varepsilon_{\ell_{2}}}};
    (-1,-21)*{ \bigb{ \varepsilon_{\ell_{a-2}}}};
    (6,-29)*{ \bigb{ \varepsilon_{\ell_{a-1}}}};
    (-21,0)*{\bullet}+(-2.5,1.5)*{\scs \ell_{1}};
    (8,-39)*{a};
  \endxy};
 (3.5,-18)*{  \xy
 (0,0)*{\includegraphics[angle=180,scale=0.5]{dexplode.eps}};
 (-5.5,-8)*{\bullet}+(-3,0)*{\scs \hat{\ell}_{1}};;
 (0,-6)*{\bullet}+(2.5,-1.5)*{\scs \hat{\ell}_{r}};;
 (5.3,-4)*{\bullet}+(4,-1.5)*{\scs \hat{\ell}_{a-2}};;
 (11,-3)*{\bullet}+(4.5,1)*{\scs \hat{\ell}_{a-1}};
 (3,-3)*{\cdots}; (-3.3,-3)*{\cdots};
  \endxy};
 (24,0)*{\includegraphics[scale=0.5]{single-up-long2.eps}};
  \endxy}
\end{equation}
then apply Lemma~\ref{lem_EaEone} in the form
\begin{equation}
  \xy
 (4,0)*{\includegraphics[scale=0.5]{long-up.eps}};
 (-4,0)*{\includegraphics[scale=0.5]{tlong-up.eps}};
 (-7 ,-12)*{a};
 (7 ,-12)*{\scs 1};
  \endxy
 \quad = \quad (-1)^{\binom{a}{2}}\sum_{\ell_a=0}^a \xy
 (0,0)*{\includegraphics[scale=0.4]{tH-one.eps}};
 (-7,-13)*{a};(8,-13)*{\scs 1};(-7,13)*{a};(7,13)*{\scs 1};
 (-5,7)*{\bigb{\varepsilon_{\ell_a} }};
 (4,-9)*{\bullet}+(5,1)*{\scs a-\ell_a};
  \endxy
\end{equation}
and use that $\binom{a}{3}+\binom{a}{2}\equiv \binom{a+1}{3} \mod 2$, proving the inductive step.
\end{proof}

Elements $\sigma_{\und{\ell}}$, $\lambda_{\und{\ell}}$ give an explicit
realization of $\ONH_a$ as the algebra of $a! \times a!$ matrices
over the ring of odd symmetric functions.  Suppose that
rows and columns of $a!\times a!$ matrices are enumerated by
elements of $\Sq(a)$.  The isomorphism takes the matrix with $x \in
\osym_a$ in the $(\und{\ell},\und{\ell'})$ entry and zeros elsewhere to
$\sigma_{\und{\ell}}x\lambda_{\und{\ell'}}$.

%
\subsubsection{Decomposition of $\mathcal{E}^{(a)}\mathcal{E}^{(b)}$ }
%

Given $\alpha \in P(a,b)$ let
\begin{equation} \label{eq_Xbeta}
 X_{\alpha}^{a,b} :=|\alpha|\cdot |\hat{\alpha}|+\chi_{\alpha}^a+\chi_{\hat{\alpha}}^b+\binom{a}{2}\left(|\hat{\alpha}|
  +\binom{b}{2}\right)+\Omega(\hat{\alpha})+\binom{a+b}{3},
\end{equation}
 where $\Omega$ is defined by equation \eqref{eqn-Omega}.

When $\alpha\in P(a,1)$, so that $\alpha=(1^{a-s})$ for some $0 \leq s \leq a$, then $|\alpha| =a-s$, $|\hat{\alpha}|=s$, and $X_{\alpha}^{a,1}$ from \eqref{eq_Xbeta} simplifies to
\begin{equation}
 X_{(1^{r})}^{a,1} = a(a-r)+\binom{a-r+1}{2}.
\end{equation}
In particular, $X_{(1^{a})}^{a,1}=1$.

For every partition $\alpha\in P(a,b)$ define
\begin{align} \label{eq_def_sigmalambda_alpha}
  \sigma_{\alpha}:=
  \xy
 (0,0)*{\includegraphics[scale=0.5]{tsplit.eps}};
 (-5,-11)*{a+b};(-8,8)*{a};(8,8)*{b};
 (-5,2)*{\bigb{s_{\alpha}}};
  \endxy, \qquad \quad
 \lambda_{\alpha}:= (-1)^{X_{\alpha}^{a,b}}
     \xy
 (0,0)*{\includegraphics[scale=0.5,angle=180]{tsplitd.eps}};
 (-5,11)*{a+b};(-8,-8)*{a};(8,-8)*{b};
 (5,-2)*{\bigb{\hat{s}_{\hat{\alpha}}}};
  \endxy, \qquad \quad e_{\alpha}=\sigma_{\alpha}\lambda_{\alpha}.
\end{align}
We view $\sigma_{\alpha}$, $\lambda_{\alpha}$, and $e_{\alpha}$ as elements of $\ONH_{a+b}$ with $\deg(\sigma_{\alpha}) = -\deg(\lambda_{\alpha})=2|\alpha|-2ab$, and $\deg(e_{\alpha})=0$.
Proposition~\ref{prop_oval_small} says that
\begin{equation}
  \lambda_{\beta} \sigma_{\alpha} = \delta_{\alpha,\beta} e_{a+b}, \qquad \text{$\alpha,\beta \in P(a,b)$.}
\end{equation}
This implies
\begin{equation}
e_{\beta}e_{\alpha}=\delta_{\alpha,\beta}e_{\alpha}.
\end{equation}

\begin{thm} \label{thm_nil-eaeb}
\begin{equation}
  e_{a,b} = \sum_{\alpha \in P(a,b)} e_{\alpha}
\end{equation}
or diagrammatically
\begin{equation}
  \xy
 (4,0)*{\includegraphics[scale=0.5]{tlong-up.eps}};
 (-4,0)*{\includegraphics[scale=0.5]{tlong-up.eps}};
 (-7 ,-12)*{a};
 (7 ,-12)*{b};
  \endxy
 \quad = \quad \sum_{\scriptscriptstyle
  \alpha\in P(a,b)
} (-1)^{X_{\alpha}^{a,b}}  \xy
 (0,0)*{\includegraphics[scale=0.5]{tH.eps}};
 (-8,-13)*{a};(8,-13)*{b};(-8,13)*{a};(8,13)*{b};
 (-5,7)*{\bigb{s_{\alpha} }};
 (4,-9)*{\bigb{\hat{s}_{\hat{\alpha}} }};
  \endxy
\end{equation}
\end{thm}

\begin{proof}
Proposition~\ref{prop_oval_small} implies that the $e_{\alpha}$ form a collection of $\binom{a+b}{a}$ mutually orthogonal idempotents in $\ONH_a$.   The result follows from the isomorphism $\ONH_a\cong \Mat_{(a)^!_{q^2}}(\osym_a)$ in Corollary~\ref{cor_nilmatrix}.
\end{proof}

%
\section{Cyclotomic quotients of the odd nilHecke algebra}\label{sec-cyclotomic}
%

Let $a_1,a_2,\ldots$ be a sequence of an algebra.  We say this family satisfies the \textit{odd symmetric defining relations} if
\begin{equation}\label{eqn-odr}\begin{split}
a_ia_j=a_ja_j	\qquad	&\text{if }i+j\text{ is even,}\\
a_ia_j+(-1)^ia_ja_i=(-1)^ia_{i+1}a_{j-1}+a_{j-1}a_{i+1}		\qquad	&\text{if }i+j\text{ is odd.}
\end{split}\end{equation}
Usually, $a_i=0$ for large enough $i$.

In \cite{EK}, the algebra $\osym$ of \textit{odd symmetric functions} (there denoted $\sym$ or $\sym^{-1}$) is constructed as the quotient of the graded free algebra $\Bbbk\langle h_1,h_2,\ldots\rangle$, $\deg(h_i)=i$, by the odd symmetric defining relations among the $h_i$'s.  The degree $n$ part of $\osym$ has a basis indexed by partitions of $n$,
\begin{equation*}
\osym_n=\Bbbk\lbrace h_\lambda=h_{\lambda_1}h_{\lambda_2}\cdots h_{\lambda_r}:\lambda=(\lambda_1,\ldots,\lambda_r)\text{ is a partition of }n\rbrace.
\end{equation*}
The elements $h_\lambda$ are called \textit{odd complete symmetric functions}.

However, we need to modify these gradings in order for the definitions of \cite{EK} to be compatible with those of this paper.  In \cite{EK}, the super-degree of a homogeneous element is the mod 2 residue of its $\Z$-degree.  Instead, for our present purposes, double the $\Z$-gradings but leave unchanged the super-gradings of \cite{EK}.  So $h_i$ has $\Z$-degree $2i$ and super-degree $i$.  The odd defining relations are unchanged.

Inductively define elements $\varepsilon_n\in\osym$ by the relation
\begin{equation}
\sum_{k=0}^n(-1)^{\binom{k+1}{2}}\varepsilon_kh_{n-k}=0.
\end{equation}
By convention, $h_0=\varepsilon_0=1$ and $h_i=\varepsilon_i=0$ for $i<0$.  The products $\varepsilon_\lambda=\varepsilon_{\lambda_1}\cdots\varepsilon_{\lambda_r}$ form a basis of $\osym_a$ just like the $h_\lambda$'s; we call such products \textit{odd elementary symmetric functions}.  It follows from Lemmas \ref{lemma-osymm-relations}, \ref{lemma-e-h-relation}, and \ref{lem-h-odrs} that $\osym$ is isomorphic to an inverse limit (in the category of graded rings) of the rings $\osym_a$,
\begin{equation*}
\osym\cong\underleftarrow{\lim}\osym_a
\end{equation*}
with respect to the maps
\begin{equation*}\begin{split}
&\osym_{a+1}\to\osym_a,\\
&\varepsilon_j\mapsto\begin{cases}
\varepsilon_j&j\neq a+1,\\
0&j=a+1,\end{cases}
\end{split}\end{equation*}
in such a way that the odd complete and odd elementary symmetric polynomials in $\osym_a$ pull back to the odd complete and odd elementary functions of \cite{EK}.

In terms of generators and relations, then,
\begin{equation}
\osym_a\cong\osym/\langle \varepsilon_m:m>a\rangle.
\end{equation}
For each $N\geq0$ we define the \textit{odd Grassmannian ring} to be the quotient of $\osym_a$ by the ideal generated by all $h_m$ with $m>N-a$,
\begin{equation*}
OH_{a,N}=\osym_a/\langle h_m:m>N-a\rangle.
\end{equation*}
For $a=0,N$ this ring is just $\Z$, and the it vanishes unless $0\leq a\leq N$.  The even analogue of $OH_{a,N}$ is the cohomology ring $H^*(\Gr(a,N))$ of the Grassmannian of complex $a$-planes in $\C^N$.  The quotient $\ONH_a^N=\ONH_a/\langle x_1^N\rangle$ is called the $N$-th \textit{cyclotomic quotient} of $\ONH_a$.

Let $\mathcal{H}_a=\lbrace \widetilde{x}^\alpha:\alpha\leq\delta\rbrace$ be as in Lemma \ref{lem_sch_zbasis} and let $\varphi:\ONH_a\to\End_{\osym_a}(\opol_a)$ be the isomorphism of Corollary \ref{cor_nilmatrix}; under $\varphi$, a polynomial $f$ acts as multiplication by $f$.    A set of defining relations for the cyclotomic quotient $\ONH_a^N$ is given by the matrix entries of $\varphi(\widetilde{x}_1)^N$ with respect to the basis $\mathcal{H}_a$.  We should therefore consider the operator $\varphi(\widetilde{x}_1)$ in some detail. Our analysis closely follows~\cite[Section 5]{Lau4}.

For each multi-index $\beta$ obtained by replacing $\alpha_1$ by zero in some $\alpha$ appearing in $\mathcal{H}_a$, consider the $\osym_a$-submodule of $\opol_a$ with basis
\begin{equation*}
B_\beta=\lbrace\widetilde{x}_1^{a-1}\widetilde{x}^\beta,\widetilde{x}_1^{a-2}\widetilde{x}^\beta,\ldots,\widetilde{x}^\beta\rbrace.
\end{equation*}
The operator $\varphi(\widetilde{x}_1)$ sends the span of each $B_\beta$ to itself; let $\varphi(\widetilde{x}_1)_\beta$ be the resulting restricted map (or the corresponding matrix with respect to the basis $B_\beta$).  So the defining relations given by the entries of $\varphi(\widetilde{x}_1)^N$ are all realized in the $a\times a$ matrices $\varphi(\widetilde{x}_1)_\beta^N$.

\begin{lem}\label{lemma-x1} The restriction of $\varphi(\widetilde{x}_1)$ to the span of $B_\beta$ has the matrix expression
\begin{equation}\label{eqn-claim}
\varphi(\widetilde{x}_1)_\beta=\begin{pmatrix}
\varepsilon_1 & 1 & 0 & 0 & \cdots & 0\\
\varepsilon_2 & 0 & 1 & 0 & \cdots & 0\\
-\varepsilon_3 & 0 & 0 & 1 & \cdots & 0\\
\vdots & \vdots & \vdots & \vdots & & \vdots\\
(-1)^{\binom{a-2}{2}}\varepsilon_{a-1} & 0 & 0 & 0 & \cdots & 1\\
(-1)^{\binom{a-1}{2}}\varepsilon_a & 0 & 0 & 0 & \cdots & 0\end{pmatrix}
\end{equation}
with respect to the basis $B_\beta$.\end{lem}
\begin{proof}
Since $\widetilde{x}_1\cdot\widetilde{x}_1^j\widetilde{x}^\beta=\widetilde{x}_1^{j+1}\widetilde{x}^\beta$, all entries other than the first column are clear.  To get the first column, we need to express $\widetilde{x}_1^a\widetilde{x}^\beta$ in the basis $B_\beta$.  That is, we want to find a linear relation
\begin{equation*}
\widetilde{x}_1^a\widetilde{x}^\beta=f_1\widetilde{x}_1^{a-1}\widetilde{x}^\beta+\ldots+f_{a-1}\widetilde{x}_1\widetilde{x}^\beta,
\end{equation*}
with each $f_j\in\osym_a$.  Recall that $\varepsilon_0=1$.  Now
\begin{equation*}\begin{split}
\varepsilon_k(x_1,\ldots,x_a)\widetilde{x}_1^{a-k}\widetilde{x}^\beta&=\widetilde{x}_1\varepsilon_{k-1}(x_2,\ldots,x_a)\widetilde{x}_1^{a-k}\widetilde{x}^\beta+\varepsilon_k(x_2,\ldots,x_a)\widetilde{x}_1^{a-k}\widetilde{x}^\beta\\
&=\sum_{2\leq i_1<\cdots<i_{k-1}\leq a}\widetilde{x}_1\widetilde{x}_{i_1}\cdots\widetilde{x}_{i_{k-1}}\widetilde{x}_1^{a-k}\widetilde{x}^\beta+\sum_{2\leq i_1<\cdots<i_k\leq a}\widetilde{x}_{i_1}\cdots\widetilde{x}_{i_k}\widetilde{x}_1^{a-k}\widetilde{x}^\beta.
\end{split}\end{equation*}
Comparing this with the analogous expansion of $\varepsilon_{k+1}\widetilde{x}_1^{a-(k+1)}\widetilde{x}^\beta$, the first sum of the latter and the second sum of the former differ by the sign $(-1)^k$.  So telescoping cancellations occur if each $f_j=\pm \varepsilon_j$ and $f_{j+1}=\pm(-1)^j\varepsilon_{j+1}$.  The claim follows; that is, we have found
\begin{equation*}
\widetilde{x}_1^a\widetilde{x}^\beta=\sum_{j=1}^n(-1)^{\binom{j-1}{2}}\varepsilon_j\widetilde{x}_1^{a-j}\widetilde{x}^\beta,
\end{equation*}
and that equation \eqref{eqn-claim} is true.\end{proof}

\begin{prop}\label{prop-odd-grassmannian} The cyclotomic quotient $\ONH_a^N$ is isomorphic to a matrix algebra of size $q^{\frac{1}{2}a(a-1)}[a]!$ over the odd Grassmannian ring $OH_{a,N}$.
\end{prop}

The proof will use an alternate presentation of $OH_{a,N}$.  Let the algebra $\osym_a[t]$ be obtained from $\osym_a$ by adjoining an element $t$ of super-degree 1 which is super-central; that is, $th_k=(-1)^kh_kt$ for all $k$ (and likewise with $h_k$ replaced by $\varepsilon_k$).  The $\Z$-degree of $t$ is immaterial, so set it to 2 for consistency with the $x_i$'s.  Then the relation
\begin{equation}
(1+\varepsilon_1t+\varepsilon_2t^2+\ldots)(1+z_1t+z_2t^2+\ldots)=1
\end{equation}
holds if and only if $z_k=(-1)^{\binom{k+1}{2}}h_k$ for each $k$.  So we can define $OH_{a,N}$ by taking the quotient of $\osym_a$ by the ideal generated by the coefficients of powers $t^k$ in
\begin{equation}\label{eqn-grassmann-relations}
(1+\varepsilon_1t+\varepsilon_2t^2+\ldots+\varepsilon_at^a)(1+z_1t+z_2t^2+\ldots+z_{N-a}t^{N-a})=1,
\end{equation}
and furthermore we know that
\begin{equation}
z_k=(-1)^{\binom{k+1}{2}}h_k
\end{equation}
for all $k$.

\begin{proof} Let $M$ be the matrix $\varphi(\widetilde{x}_1)_\beta$ of equation \eqref{eqn-claim}.  By Lemma \ref{lemma-x1}, the entries of $M^N$ generate the ideal defining the cyclotomic quotient $\ONH_a^N$.  It is not hard to show that these relations are already generated by the entries in the first column of $M^{N-a+1}$ (the proof is exactly as in the even case).  These relations are homogeneous of degrees $2(N-a+1),2(N-a+2),\ldots,2N$.  Let $v=(1,0,\ldots,0)^T$ be the column vector with first entry 1 and all other entries 0, so that we are seeking to compute the entries of $M^{N-a+1}v$.  We proceed by induction on $N$.  We may as well assume $N-a\geq a$; the case $N-a<a$ is similar.  Let $f_{j,N-a}$ be the relation of degree $2(N-a+j)$ in equation \eqref{eqn-grassmann-relations}, for each $j=1,\ldots,a$.  We claim that
\begin{equation*}
(M^{N-a+1}v)_j=(-1)^{\binom{N-a+j+1}{2}}f_{j,N-a}.
\end{equation*}
For $N=a,a+1$, this is clear.  Now replacing $N$ by $N+1$, matrix multiplication shows
\begin{equation}\label{eqn-cases1}
(M^{N-a+2}v)_j=\begin{cases}
(-1)^{\binom{j-1}{2}+\binom{N-a}{2}}\varepsilon_jf_{1,N-a}+(-1)^{\binom{N-a+j+2}{2}}f_{j+1,N-a} & \text{for }1\leq j\leq a-1,\\
(-1)^{\binom{a-1}{2}+\binom{N-a+2}{2}}\varepsilon_af_{1,N-a}	 \qquad\qquad	 \text{for }j=a.	 &
\end{cases}\end{equation}
And by equation \eqref{eqn-grassmann-relations} with $N+1$ in place of $N$,
\begin{equation}\label{eqn-cases2}
f_{j,N-a+1}=\begin{cases}
f_{j+1,N-a}-(-1)^{j(N-a+1)}\varepsilon_jf_{1,N-a}	&	\text{for }1\leq j\leq a-1,\\
(-1)^{\binom{a-1}{2}+\binom{N-a+2}{2}}\varepsilon_af_{1,N-a}	&	 \text{for }j=a.
\end{cases}\end{equation}
The expressions \eqref{eqn-cases1},\eqref{eqn-cases2} differ only by a sign, proving the proposition.\end{proof}

The algebra $\osym$ has an important basis, the \textit{odd Schur functions}.  They are indexed by partitions as the complete and elementary functions are, and they are defined by the change of basis relation
\begin{equation}
h_\mu=\sum_{\lambda\vdash n}K_{\lambda\mu}s_\lambda^H,
\end{equation}
where $|\mu|=n$ and $K_{\lambda\mu}$ is a signed count of semistandard Young tableaux of shape $\lambda$ and content $\mu$; for the signs, see \cite{EK}.  The superscript ``$H$'' is to distinguish these from the Schur functions of Subsection \ref{sec_oschur} The Schur functions form an integral basis over $\Z$, and they are signed-orthonormal:
\begin{equation}
(s_\lambda^H,s_\mu^H)=(-1)^{\ang{\overline{\lambda}}+|\lambda|}\delta_{\lambda\mu}.
\end{equation}
Here we use the notation $\ang{\mu}=\sum_j\binom{\mu_j+1}{2}$.  In the even setting, the images of those Schur functions $s_\lambda$ with $\lambda_1\leq N-a$ and $\ell(\lambda)\leq a$ form an integral basis of $H^*(\Gr(a,N);\Z)$; the class $s_\lambda$ is Poincar\'{e} dual to the corresponding Schubert cycle.  The odd Schur functions corresponding to the same partitions give a basis in the odd case as well, although no geometric description is currently known.

\begin{conj}\label{conj-schur} The image of $s_\lambda^H$ under the map $\osym\to\osym_a$ equals the odd Schur polynomial $s_\lambda$.\end{conj}

\begin{prop} Over $\Bbbk=\Z$, the images of the Schur functions $s_\lambda^H$ for $\lambda$ having at most $a$ rows and at most $N-a$ columns form a homogeneous basis for $OH_{a,N}$.  All other Schur functions are 0 in $OH_{a,N}$.\end{prop}
\begin{proof} It is shown in \cite{EK} that
\begin{equation*}\begin{split}
&s_\lambda^H=h_\lambda+\sum_{\mu>\lambda}a_\mu s_\mu^H,\\
&s_\lambda^H=\varepsilon_{\overline{\lambda}}+\sum_{\mu>\overline{\lambda}}b_\mu s_\mu^H,
\end{split}\end{equation*}
for certain integers $a_\mu,b_\mu$; the ordering on partitions here is the lexicographic one.  As a result, if $\lambda$ has more than $a$ rows, $s_\lambda^H$ is a linear combination of $\varepsilon_\mu$'s with $\mu_1>a$ for all $\mu$; hence $s_\lambda^H=0$.  Applying the involution $\psi_1\psi_2$ of \cite{EK}, analogous reasoning with the odd complete polynomials implies that $s_\lambda^H=0$ if $\lambda$ has more than $N-a$ columns.  So $OH_{a,N}$ is spanned over $\Z$ by the Schur functions $s_\lambda^H$ for $\lambda$ having at most $a$ rows and at most $N-a$ columns.  Reducing modulo 2, odd symmetric polynomials coincide with the usual (even) symmetric polynomials over $\Z/2$.  In particular, they have the same graded rank over $\Z/2$.  Since $OH_{a,N}$ is a free $\Z$-module, this implies that these Schur functions are in fact a basis over $\Z$.\end{proof}

%
\section{Categorification}\label{sec-categorification}
%

Consider the category $\ONH_a\pmod$ of finitely generated graded left projective $\ONH_a$-modules and its Grothendieck group $K_0(\ONH_a)$.  Let $\ONH := \bigoplus_{a\geq 0} \ONH_a$ and define
\[
 K_0(\ONH) := \bigoplus_{a\geq 0} K_0(\ONH_a).
\]
There is a natural inclusion of algebras $\ONH_a \otimes \ONH_b \subset \ONH_{a+b}$ given on diagrams by placing a diagram in $\ONH_a$ next to a diagram in $\ONH_b$ with the diagram from $\ONH_a$ appearing to the left of the one from $\ONH_b$.  These inclusions give rise to induction and restriction functors that equip $K_0(\ONH)$ with the structure of a $q$-bialgebra (for the notion of $q$-bialgebra, see \cite{EK}).

Denote the regular representation of $\ONH_a$ by $\mathcal{E}^a$.  By Corollary~\ref{cor_nilmatrix} this module decomposes into the direct sum of $a!$ copies of the unique indecomposable projective module of $\ONH_a$.  If we denote by $\mathcal{E}^{(a)}$ the projective module corresponding to the minimal idempotent $e_a$ with the grading shifted down by $\frac{a(a-1)}{2}$, then we get a direct sum decomposition of graded modules
\[
 \mathcal{E}^a = \ONH_a \cong \bigoplus_{[a]!} \left(\ONH_a e_a\right) \lbrace\frac{-a(a-1)}{2} \rbrace
  =  \bigoplus_{[a]!} \mathcal{E}^{(a)}.
\]
Here $[a]! = [a][a-1]\dots [1]$ is the quantum factorial, $[a]= \frac{q^a-q^{-a}}{q-q^{-1}}$, and $M^{\oplus f}$ or $\oplus_f M$, for a graded module $M$ and a Laurent polynomial $f=\sum f_j q^j \in \Z[q,q^{-1}]$, denotes the direct sum over $j \in \Z$, of $f_j$ copies of $M\{j\}$.

Diagrammatically, the projective module $\mathcal{E}^a$ corresponds to the idempotent $1 \in \ONH_a$ given by $a$ vertical lines.  The idempotent $e_a$ corresponding to the indecomposable projective module $\mathcal{E}^{(a)}$ is represented in the graphical calculus by a thick edge of thickness $a$.   Theorem~\ref{thm_nil_matrix} can be interpreted as giving an explicit isomorphism
\begin{equation}
  \sum_{\und{\ell} \in \Sq(a)} \lambda_{\und{\ell}} \maps
  \mathcal{E}^a \longrightarrow
  \oplus_{[a]!}\mathcal{E}^{(a)} = \bigoplus_{\und{\ell} \in \Sq(a)}\mathcal{E}^{(a)} \{a -1 -2|\und{\ell}|\},
\end{equation}
while Theorem~\ref{thm_nil-eaeb} gives a canonical isomorphism
\begin{eqnarray}
 \sum_{\alpha \in P(a,b)} \lambda_{\alpha} & \maps &  \mathcal{E}^{(a)}\mathcal{E}^{(b)}\onen \longrightarrow
\bigoplus_{\qbin{a+b}{a}}\mathcal{E}^{(a+b)} \;\; =\;\;\bigoplus_{\alpha \in P(a,b)}  \mathcal{E}^{(a+b)}\onen \{2|\alpha|-ab\}.
\end{eqnarray}
In particular, we have an isomorphism
\begin{align}
  \mathbf{U}_q^+(\mathbf{sl}_2)_{\mathcal{A}} &\to K_0(\ONH) \nn \\
    \theta^{(a)} & \mapsto  \mathcal{E}^{(a)}
\end{align}
where $\mathbf{U}_q^+(\mathbf{sl}_2)_{\mathcal{A}}$ is the integral version of the algebra $\mathbf{U}_q^+(\mathbf{sl}_2)$.  The same result was announced by Kang, Kashiwara, and Tsuchioka \cite{KKT}.

Denote by $\ONH^N := \bigoplus_{a=0}^N\ONH_a^N$ the direct sum of cyclotomic quotients of $\ONH_a$. The odd Grassmannian ring $OH_{a,N}$ is graded local, so Proposition~\ref{prop-odd-grassmannian} implies that $K_0(\ONH^N)$ has the same size as the integral form of the irreducible representation of $\mathbf{U}_q(\mf{sl}_2)$ of highest weight $N$ (the cyclotomic quotient $\ONH_a^N$ is zero unless $0\leq a\leq N$).  The tensor products by bimodules which are odd analogues of the cohomology of two-step flag varieties descend to the action of $E$ and $F$ on $K_0(\ONH^N)$, as in the even case.


%

%

%
\end{document}